\documentclass[dvipsnames]{amsart}

\usepackage{geometry}

\usepackage{todonotes}
\usepackage{verbatim}
\usepackage{enumitem}
\usepackage{epstopdf}

\usepackage{amsmath}  
\usepackage{amssymb}     
\usepackage{amsthm}
\usepackage{dsfont}
\usepackage{caption,subcaption}

\usepackage{xcolor}
\usepackage[pdftex, colorlinks, linkcolor=MidnightBlue, citecolor=magenta, urlcolor=MidnightBlue]{hyperref}
\usepackage[noabbrev]{cleveref}

\usepackage[backend=bibtex,bibencoding=utf8,style=alphabetic,sorting=nyt]{biblatex}
\AtBeginBibliography{\small}
\AtEveryBibitem{\clearlist{language}}

\usepackage{parskip}
\newcommand\FF{{\mathbb F}}

\newcommand\RR{{\mathbb R}}
\newcommand{\R}{\mathbb{R}}
\newcommand\ZZ{{\mathbb Z}}
\newcommand{\Z}{\mathbb{Z}}

\newcommand\cE{{\mathcal E}}

\newcommand\cL{{\mathcal L}}
\newcommand\cM{{\mathcal M}}
\newcommand\cP{{\mathcal P}}

\newcommand\cR{{\mathcal R}}

\newcommand\0{{\mathbf 0}}

\newcommand{\Rinf}{\RR \cup \{\infty\}}
\newcommand{\M}{\text{M}}
\newcommand{\Mnat}{\text{M}^{\natural}}
\newcommand{\tE}{{\widetilde{E}}}

\newcommand\SetOf[2]{\left\{#1 \mid #2\right\}}

\newcommand\biggSetOf[2]{\biggl\{#1 \biggm| #2\biggr\}}
\newcommand\BiggSetOf[2]{\Biggl\{#1 \Biggm| #2\Biggr\}}
\newcommand{\quotient}{\twoheadrightarrow}
\newcommand{\exquotient}{\twoheadrightarrow_{\text{Ex}}}
\newcommand{\minquotient}{\twoheadrightarrow_{\text{min}}}
\newcommand\ind[2]{{\rm ind}_{#2}(#1)}
\newcommand\lift[4]{L^{#1}_{#2,#3}(#4)}

\newcommand{\set}[1]{\ensuremath{ \left\lbrace  #1 \right\rbrace }}
\newcommand{\ma}{\begin{pmatrix} }      \newcommand{\trix}{\end{pmatrix}}
\newcommand{\sma}{\left(\begin{smallmatrix} }       \newcommand{\strix}{\end{smallmatrix}\right)}

\DeclareMathOperator{\conv}{conv}
\DeclareMathOperator{\im}{im}
\DeclareMathOperator{\id}{id}

\DeclareMathOperator*{\argmin}{arg\,min}
\DeclareMathOperator{\dom}{dom}
\DeclareMathOperator{\Sym}{Sym}
\DeclareMathOperator{\supp}{supp}

\let\oldforall\forall
\let\forall\undefined
\DeclareMathOperator{\forall}{\ \oldforall}
\let\oldexists\exists
\let\exists\undefined
\DeclareMathOperator{\exists}{\ \oldexists}

\newcommand{\bases}[1][]{%
	\ifthenelse{ \equal{#1}{} }
	{\ensuremath{\mathcal{B}}}
	{\ensuremath{\mathcal{B}(#1)}}
}
\newcommand{\rank}[1][]{%
	\ifthenelse{ \equal{#1}{} }
	{\text{rk}}
	{\text{rk}_{#1}}
}

\newtheorem{theorem}{Theorem}[section]

\newtheorem{manualtheoreminner}{Theorem}
\newenvironment{manualtheorem}[1]{%
	\renewcommand\themanualtheoreminner{#1}%
	\manualtheoreminner
}{\endmanualtheoreminner}

\theoremstyle{definition}

\newtheorem{construction}[theorem]{Construction}

\usepackage{aliascnt}

\newaliascnt{proposition}{theorem}
\newtheorem{proposition}[proposition]{Proposition}
\aliascntresetthe{proposition}

\newaliascnt{lemma}{theorem}
\newtheorem{lemma}[lemma]{Lemma}
\aliascntresetthe{lemma}

\newaliascnt{corollary}{theorem}
\newtheorem{corollary}[corollary]{Corollary}
\aliascntresetthe{corollary}

\newaliascnt{conjecture}{theorem}

\aliascntresetthe{conjecture}

\theoremstyle{definition}

\newaliascnt{remark}{theorem}
\newtheorem{remark}[remark]{Remark}
\aliascntresetthe{remark}

\newaliascnt{question}{theorem}
\newtheorem{question}[question]{Question}
\aliascntresetthe{question}

\newaliascnt{assumption}{theorem}

\aliascntresetthe{assumption}

\newaliascnt{definition}{theorem}
\newtheorem{definition}[definition]{Definition}
\aliascntresetthe{definition}

\newaliascnt{example}{theorem}
\newtheorem{example}[example]{Example}
\aliascntresetthe{example}

\makeatletter
\def\mscclasses{\xdef\@thefnmark{}\@footnotetext}
\makeatother

\author{Marie-Charlotte Brandenburg \and Georg Loho \and Ben Smith}

\begin{document}

\title{Quotients of M-convex sets and M-convex functions}

\begin{abstract}
We unify the study of quotients of matroids, polymatroids, valuated matroids and strong maps of submodular functions in the framework of Murota's discrete convex analysis.
As a main result, we compile a list of ten equivalent characterizations of quotients for \M-convex sets, generalizing existing formulations for (poly)matroids and submodular functions.
We also initiate the study of quotients of \M-convex functions, constructing a hierarchy of four separate characterizations.
Our investigations yield new insights into the fundamental operation of induction, as well as the structure of linking sets and linking functions, which are generalizations of linking systems and bimatroids.
\end{abstract}

\maketitle

\mscclasses{\hspace*{-1.8em} \textsc{MSC Classes.}
	Primary:
	05B35, %Combinatorial aspects of matroids and geometric lattices
	14T15,    %Combinatorial aspects of tropical varieties
	52B20, %Lattice polytopes in convex geometry
	52B40. %   Matroids in convex geometry (realizations in the context of convex polytopes, convexity in combinatorial structures, etc.)
	Secondary:
	14M15,    %Grassmannians, Schubert varieties, flag manifolds
	90C25, %    Convex programming
	90C27. %   Combinatorial optimization
}

\mscclasses{\hspace*{-1.8em} \textsc{Keywords.} discrete convex functions; flag matroids; discrete convex analysis; matroid theory; matroid quotients; polymatroids.}

\section{Introduction}

\emph{Matroids} form a combinatorial model of linear dependence in a configuration of $1$-dimensional linear spaces, as identified by Whitney~\cite{Whitney:1935} and Nakasawa~\cite{Nakasawa:1935}.  
In the same spirit, \emph{matroid quotients} form an abstraction of linear maps of configurations of $1$-dimensional linear spaces going back to Crapo~\cite{Crapo:1967} and Higgs~\cite{Higgs:1968}.
Matroid quotients are the building block for flag matroids~\cite{borovik03_coxetermatroids} which form a versatile tool in the study of flag varieties.

It is natural to ask what happens for configurations of potentially higher dimensional linear spaces.  
It turns out that the combinatorial structure of a configuration of general linear spaces is captured by a \emph{submodular function}. %, or equivalently by a generalized permutohedron, or an \M-convex set. 
Considering linear maps of configurations of linear spaces leads to \emph{strong maps} of submodular functions, a generalization of matroid quotients (\Cref{subsec:realizable-quotients}). 
These were already identified by Fujishige in the first version (1991) of~\cite{Fujishige:2005}. 
%Fujishige already identified this generalization of strong maps to submodular functions in the first version (1991) of~\cite{Fujishige:2005}. 
They also have a polyhedral description in terms of so-called \emph{generalized polymatroids}, see the paper by Frank \& Tardos~\cite{FrankTardos:1988}. 
%They can also be equivalently stated in polyhedral terms of so-called \emph{generalized polymatroids}, see the paper by Frank \& Tardos~\cite{FrankTardos:1988}. 

There is another generalization of matroids in a different direction to submodular functions, namely \emph{valuated matroids} introduced by Dress \& Wenzel~\cite{DressWenzel:1992}. 
These can be seen as height functions on the bases of a matroid that are compatible with the matroid structure.
An important class of valuated matroids arises by tropicalizing Pl\"ucker vectors over a non-Archimedean field; these represent tropical linear spaces. 
Quotients of valuated matroids were introduced by Haque~\cite{Haque:2012} and further studied in the context of tropical geometry by Brandt, Eur and Zhang~\cite{BrandtEurZhang:2021}. 

Formalised by Murota~\cite{Murota:2003}, discrete convex analysis provides a unified framework for submodular functions and valuated matroids. 
The main building blocks are \emph{\M-convex sets}, lattice points of integral generalized permutohedra, or equivalently the bases of an integral submodular base polyhedron.
They form a generalization of matroids from set systems to sets of integer points. 
In analogy with convex functions, there is the concept of \emph{\M-convex functions} where all minimizers are \M-convex sets, generalizing valuated matroids.
Note that \M-convex functions are exactly the tropicalizations of Lorentzian polynomials~\cite{BraendenHuh:2020}. 

\subsection*{Overview of main results}
The aim of this paper is to unify the theory of quotients of submodular functions and \M-convex sets in geometric terms, and to build a foundation for the study of quotients of \M-convex functions. 
Building on existing work, we compile a list of ten equivalent characterizations of quotients of \M-convex sets (\Cref{thm:quotient}).
Many of the new characterizations depend on \emph{induction}, a powerful operation generalizing the classical construction of induction by graphs~\cite{Brualdi:1971}.
This involves a thorough study of \emph{linking sets}, a reformulation of \emph{poly-linking systems} and \emph{bi-polymatroids}~\cite{Schrijver:1978}.

With these foundations, we propose four separate notions of quotients of \M-convex functions and establish their relationship with one another (\Cref{thm:functions-intro}). 
In particular, we show how $\Mnat$-convex functions are special cases of flags of \M-convex functions.
Finally, we drop the integrality assumption of \M-convex sets and identify equivalent characterizations of quotients of generalized permutohedra (\Cref{thm:real+quotients}). 
Overall, we present new tools to study quotients of matroids and valuated matroids in the geometric framework of discrete convex analysis.

\subsection{An instructive example: realizable quotients}
\label{subsec:realizable-quotients}

We motivate the study of quotients of \M-convex sets and their submodular functions with the following instructive example.
Let $V$ be a vector space and $(L_i)_{i \in E}$ a collection of subspaces of $V$.
It is folklore that the function defined by
\[
f \colon 2^E \to \mathbb{R}, \quad f(S) = \dim\left(\sum_{i \in S} L_i\right) \, \forall S \subseteq E
\]
is submodular, 
and encodes the combinatorial structure of the arrangement of linear subspaces. 
Now let $\phi \colon V \to W$ be a linear map which maps the arrangement in $V$ to an arrangement of linear spaces of $W$.
Again, one can define the submodular function
\[
g \colon 2^E \to \mathbb{R}, \quad g(S) = \dim\left(\sum_{i \in S} \phi(L_i)\right) \, \forall S \subseteq E \enspace .
\]
These two functions satisfy the relation %are in an intricate relation, namely one gets
\begin{equation}\label{eq:compliant+realizable}
   f(X) - g(X) \leq f(Y) - g(Y) \quad \forall X \subseteq Y \subseteq E \enspace .
\end{equation}
This follows from the elementary formula $\dim U = \dim \im_{\phi}(U) + \dim \ker_{\phi}(U)$, where $\im_{\phi}(U) = \im \phi|_U$ and $\ker_{\phi}(U) = \ker \phi|_U$, as
\begin{align*}
  \dim \left(\sum_{i \in X}L_i\right) - \dim \im_{\phi} \left(\sum_{i \in X}L_i\right) &= \dim \ker_{\phi} \left(\sum_{i \in X}L_i\right) \leq \dim \ker_{\phi} \left(\sum_{i \in Y}L_i\right) \\
  &= \dim \left(\sum_{i \in Y}L_i\right) - \dim \im_{\phi} \left(\sum_{i \in Y}L_i\right) \enspace .
\end{align*}
The equation~\eqref{eq:compliant+realizable} will be one of many ways to define quotients of \M-convex sets and their submodular functions.

\subsection{Matroid quotients}

There are various equivalent definitions of quotients of matroids.
We recall three conceptually different ones in terms of bases, rank functions and a bigger matroid.

\begin{definition} \label{def:matroid-quotient}
  Let $M,N$ be matroids on the same ground set $E$.
 We have $N$ is a quotient of $M$, denoted $M \quotient N$, if one of the following equivalent conditions holds. 
	\begin{enumerate}
		\item For all $ X \in \bases[N]$, $Y \in \bases[M]$ and $i \in X \setminus Y$, there exists $j \in Y \setminus X $ such that 
		\[
		X \setminus i \cup j \in \bases[N] \, , \, Y \cup i \setminus j \in \bases[M] \, .
		\]
		\item For all $A  \subseteq B \subseteq E$, the rank functions satisfy
		$\rank[N](B) - \rank[N](A) \leq \rank[M](B) - \rank[M](A)$.
		\item There exists a matroid $R$ on ground set $E' = E \sqcup X$ such that $M = R \setminus X$ is a deletion of $R$ and $N = R / X$ is a contraction of $R$.
	\end{enumerate}
\end{definition}

We refer to \cite[Proposition 7.4.7]{Brylawski:1986} for further equivalent conditions in terms of the corresponding geometric lattice, the closure operator, flats, hyperplanes, cocircuits, independent sets, circuits and the dual matroid.
Furthermore, \cite[Theorem 4]{Kung:1978} shows an equivalent condition in terms of bimatroids, or equivalently linking systems.
In the literature, matroid quotients may be disguised under the name of \emph{strong map} or \emph{morphism}. %, sometimes with minor technical differences in their constructions.
They are also formulated in terms of \emph{extensions} or \emph{lifts of matroids}. 

Using the concept of a matroid quotient, one can define a \emph{flag matroid} as a sequence of matroids $M_0,\dots,M_k$ with $M_i \quotient M_{i-1}$ for each $i \in [k]$.
Flag matroids are a special class of Coxeter matroids~\cite{borovik03_coxetermatroids}.

\goodbreak

\subsection{Main Theorem on quotients of \M-convex sets}

We derive a list of equivalent conditions for quotients of \M-convex sets. 
As already exhibited in \Cref{def:matroid-quotient}, there are three conceptually different ways to consider quotients: two sets (of bases), two functions (the rank functions), or one set (the bigger matroid). 
This point of view also appears in the following theorem.
A quotient can arise as a pair of \M-convex sets, a pair of submodular functions or a single \M-convex / $\Mnat$-convex set. 

\begin{theorem}\label{thm:quotient}
  Let $P, Q \subset \Z^E$ be \M-convex sets, and $p,q \colon 2^E \to \Z$ be the corresponding submodular set functions.
  %, i.e. $P = B(p) \cap \Z^d, Q = B(q) \cap \Z^d$, where $B(p)$ denotes the base polyhedron of a submodular function. 
	Then the following are equivalent conditions for $P \quotient Q$, that is, $Q$ is a quotient of~$P$. 
	\begin{enumerate}
	  \item (compliant functions) For all $X \subseteq Y \subseteq E$, the inequality $q(Y) - q(X) \leq p(Y) - p(X)$ holds.
		 \label{submod-fctns-compliant}
	  \item (containment of bases) For every $\sigma \in \Sym(E)$, the vertices $x_\sigma \in P$ and $y_\sigma \in Q$ satisfy $x_\sigma \geq y_\sigma$. \label{bases+containment} 
		\item (submodular polyhedron containment) For all $X \subseteq E$, the containment $S(q/X) \subseteq S(p/X)$ holds, i.e., the submodular polyhedra are contained for all contractions. \label{submod-poly-containment} 
%		\item (top and bottom) $P$ is the top layer and $Q$ is the bottom layer of the $\Mnat$-convex set $G(p,q^\#) \cap \Z^E$. \label{Mnatural}
		\item (top and bottom) There exists an $\Mnat$-convex set $R \in \ZZ^E$ such that $P$ is the top layer and $Q$ is the bottom layer of $R$.
		This $\Mnat$-convex set is $R = G(p,q^\#) \cap \Z^E$. \label{Mnatural}
		\item (deletion-contraction) There exists an \M-convex set $R \subseteq \ZZ^{\tE}$ with $\tE = E \sqcup X$ such that $P = R \setminus X$ is the deletion of $R$ and $Q = R / X$ is the contraction of $R$.
		Moreover, $X$ can be picked to be a singleton. \label{deletion-contraction}
%		There exists a submodular set function $r: E \sqcup e \to \Z$ such that $p$ is the restriction of $r$, $q^\#$ is the restriction of $r^\#$, and the generalized polymatroid $G(p,q^\#)$ is the coordinate projection of $B(r)$ onto $\{x_e = 0\}$.   \todo{Rewrite this to be analogous to Defn 1.1 (3) - first as a set but also note that a single element will do. Edit as appropriate later}
		\item (exchange property) For all $x \in Q, y \in P$ and $i \in \supp^+(x-y)$, there exists some $j \in \supp^-(x-y)$ such that $x - e_i + e_j \in Q$ and $y + e_i - e_j \in P$. \label{exchange-property}
		\item (induction)\label{induction}
		There exists an \M-convex set $\Gamma \subseteq \ZZ^E \times \ZZ^\tE$ and $W \in \ZZ^\tE$ \M-convex such that $P$ is the top layer of the left set of $\Gamma$, and $Q$ is the induction of $W$ through $\Gamma$, i.e.
		\[
		P = \pi_E(\Gamma)^\uparrow \quad , \quad Q = \ind{W}{\Gamma} \, .
		\]
		\item ($\cR$-ordered linking sets) \label{r-order}
    There exist linking sets $\Gamma, \Delta \subseteq \ZZ^E \times \ZZ^E$ such that $P$ and $Q$ are top layers of the left sets of $\Gamma$ and $\Delta$ respectively, and $\Gamma \preceq_\cR \Delta$ where $\preceq_\cR$ is Green's right preorder on the monoid of linking sets.
    \item  (matroid quotient lift) Any compatible matroid lifts of $P$ and $Q$ form a matroid quotient. \label{lift-quotient}
    \item (compressed quotient) The sum $P+Q$ is a flag \M-convex set of type $\left((\rank(Q) + \ell,\rank(P) + \ell),\phi\right)$ for some surjection $\phi$ onto $E$ and $\ell \in \ZZ$. \label{compressed-quotient}
      %characterization of compressed quotient / flag (compare with \cite{BrandtEurZhang:2021}, \cite{FujishigeHirai:2022} and Borovik, Gelfand, White)
	\end{enumerate}
\end{theorem}

We emphasize that \eqref{exchange-property} through \eqref{compressed-quotient} are new characterizations, while \eqref{submod-fctns-compliant} through \eqref{deletion-contraction} are mostly reformulations of known characterizations.

While one of the first and most prominent characterizations of matroid quotients is in terms of flats, this does not appear in our list for \M-convex sets. 
This has two reasons.
Firstly, the flats of a submodular function do not suffice to characterize the submodular function.
Secondly, the flats do appear implicitly in the containment condition \eqref{submod-poly-containment}, as they define the submodular polyhedra. 

\subsection{Main Theorem on quotients of \M-convex functions}

While some of the characterizations in \Cref{thm:quotient} were (implicitly) known, the results and subsequent insights for quotients of \M-convex functions are entirely novel.
They build on the framework developed in \Cref{thm:quotient} for \M-convex sets combined with the characterizations of quotients for valuated matroids~\cite{BrandtEurZhang:2021}.

% \todo[inline]{Georg: distinction between Mnat-convex, quotient and flag}

\begin{theorem}\label{thm:functions-intro}
  Let $f,g : \Z^E \to \R \cup \{\infty\}$ be \M-convex functions such that $\rank(g) < \rank(f)$.
  Consider the following statements.
	\begin{enumerate}[label=(\Alph*)]
		\item (top and bottom) There exists an $\Mnat$-convex function $h: \Z^E \to \Rinf$ such that $f$ and $g$ are the top and bottom layers of $h$ respectively.
		\item (induction)
			There exists a linking function $\gamma : \ZZ^E \times \ZZ^{\tE} \to \Rinf$ and an \M-convex function $r :~\ZZ^{\tE} \to \Rinf$ such that $f$ is the left function of $\gamma$, and $g$ is the induction of $r$ through $\gamma$, i.e.
			\[
			f = \pi_E(\gamma)^\uparrow \quad , \quad g = \ind{r}{\gamma} \, .
			\]
		      \item (exchange property)
                        For every $x \in \dom(f), y \in \dom(g), i \in \supp^+(y-x)$ there exists a $j \in \supp^-(y-x)$ such that
		\[
		f(x) + g(y) \geq f(x + e_i - e_j) + g(y - e_i + e_j).
		\]
		\item (minimizers) For every $u \in (\R^E)^*$ the minimizers $f^u \quotient g^u$ are quotients as \M-convex sets. 
	\end{enumerate}
	Then \ref{quo:Mnat} $\implies$ \ref{quo:induction} $\implies$ \ref{quo:exchange}$\implies$ \ref{quo:minimizers}.
	If $\rank(f) = \rank(g)+1$ then these are all equivalences.
\end{theorem}

This allows us to clarify the relationship between $\Mnat$-convex functions and quotients of \M-convex functions. 

\begin{theorem}\label{th:flags-intro}
  Let $f_0,f_1,\dots,f_k: \Z^E \to \Rinf$ be \M-convex functions such  that $(f_{i}, f_{i-1})$ satisfy $\rank(f_i)=\rank(f_{i-1})+1$ and any condition of \Cref{thm:functions-intro} for all $i \in [k]$, 
  and $P = \bigcup_{i=0}^k \dom(f_i)$ is an $\Mnat$-convex set.
	Then there exist constants $c_0,c_1,\dots,c_k \in \RR$ such that $f_0 - c_0, f_1 - c_1, \dots, f_k - c_k$ are the layers of the $\Mnat$-convex function
	\[
	h(x) = \inf(f_0(x)-c_0,\dots,f_k(x)-c_k) \, .
	\]

\end{theorem}

Applying this to valuated matroids, one obtains that a valuated flag matroid with a generalized matroid as support is the same as a valuated generalized matroid up to adding a constant to each layer.
More generally, we see that $\Mnat$-convex functions are special flags of \M-convex functions. 

\subsection{Motivation}

Our work draws from several directions of research.

\subsubsection*{Optimization}

Inspired by matroid intersection and induction by (bipartite) graphs, Schrijver introduced (poly-)linking systems \cite{Schrijver:1978}. 
They give rise to a flexible framework for constructing matroids and expressing a wide range of optimization problems.
Not long after their introduction, it was shown that matroid quotients can also be nicely interpreted in terms of linking systems as a combinatorial analog to matroid multiplication~\cite{Kung:1978}. 
Furthermore, they can also be expressed in terms of generalized polymatroids~\cite{FrankTardos:1988}.
Generalizing the Dulmage-Mendelsohn decomposition of bipartite graphs,~\cite{Nakamura:1988} derived a decomposition of poly-linking systems. 
In a similar spirit of generalizing graph constructions, \cite{GoemansIwataZenklusen:2012} used poly-linking systems to obtain a more general framework for modelling flow, further extended in \cite{Fujishige:2013}

Quotients are also a helpful structure for parametric optimization. 
One can derive efficient submodular intersection algorithms \cite{IwataMurotaShigeno:1997, FujishigeNagano:2009} in the case that a parametrized family of submodular functions form quotients in terms of the parameter. 
Under the same requirement on the family, there is also a parametric principal partition \cite{Fujishige:2009}. 
By looking at the compression of $\Mnat$-convex functions, \cite{FujishigeHirai:2022} gave an interpretation of the resulting \M-convex function in terms of parametric minimization; the crucial building blocks are again quotients of submodular functions. 

Finally, it has been shown that several natural constructions for valuated matroids do not suffice to generate all of them from some fundamental building blocks~\cite{HusicLohoSmithVegh:2022}.
A suitable choice of linking functions (\Cref{sec:functions}) may be flexible but still simple enough to be a good candidate building block for constructing all valuated matroids.

\subsubsection*{Tropical geometry}

The study of the Grassmannian and its tropicalization uses many techniques from matroid theory.
Similar to the Grassmannian, flag varieties are a promising class of varieties to be investigated using tools from tropical geometry and the theory of matroid quotients.
%Similarly, flag varieties have proved to be amenable to study with tools from tropical geometry.
This has lead to various work studying flag varieties with tropical geometry and studying tropical flag varieties, 
see~\cite{BrandtEurZhang:2021,FangFeiginFourierMakhlin:2019,BoretskyEurWilliams:2022,BorziSchleis:2023} for a few recent developments. 
An emphasis on polyhedral subdivisions and the interplay between flags of tropical linear spaces and discrete convexity was given in~\cite{JoswigLohoLuberOlarte:2023}.
One can also study matroids and their quotients over other hyperfields \cite{BakerBowler:2019,JarraLorscheid:2024}, a framework that captures both flag matroids and valuated flag matroids (\cite[Proposition 2.21]{JarraLorscheid:2024}).
% Note that M-convex functions are tropicalization of Lorentzian polynomials~\cite{BraendenHuh:2020}. 

Already \cite[Chapter 4]{Frenk:2013} identified a concept of a morphism of tropical linear spaces in terms of linking systems. 
In general it is not so clear what the right concept of linear maps in tropical geometry should be, as \cite{FinkRincon:2015} demonstrates. 
Our study of linking sets and linking functions provides new tools to study tropical analogs of linear maps. 

\subsubsection*{Discrete convex analysis for matroid theory}

The framework of `discrete convex functions' unifies the study of many combinatorial constructions and has many applications in optimization and economics.
There are various characterizations of such functions in terms of exchange properties, generalizing the Steinitz exchange lemma for bases of a vector space.
Such functions are not merely abstract generalizations but have powerful applications in matroid theory and beyond.
For example, quotients of \M-convex sets already appear implicitly in the proof of the factorization theorem for strong maps of matroids~\cite[\S 17.2, Lemma 1]{Welsh:1976}.%: if two submodular functions are compliant then their point-wise minimum is submodular.

Many linear algebraic constructions that can be formulated in terms of matroids can be generalized even further in terms of \M-convex sets.
For example, there is an analog of eigensets and dynamical systems based on poly-linking systems \cite{Murota:1990,Murota:1989}.
We use a slightly different approach in terms of \emph{linking sets}; we extend their algebraic study by exhibiting basic properties of the monoid of linking sets.
Furthermore, this framework allows us to derive many powerful generalizations of structures of matroids with geometric tools.

\subsection{Roadmap}

Section \ref{sec:M-convex} provides an introduction to fundamental objects of Discrete Convex Analysis and matroid quotients.
We unify existing results on submodular functions, \M-convex sets and $\Mnat$-convex sets, and combine them with the equivalent characterizations of matroid quotients.
This leads to the characterizations \eqref{submod-fctns-compliant} to \eqref{exchange-property} in our Main Theorem.
In Section \ref{sec:induction}, we go further by introducing induction and studying the structure of linking sets.
This allows us to derive the characterizations \eqref{induction} and \eqref{r-order}.
Extending the framework for associating a matroid to a polymatroid, we combine induction, quotients of \M-convex sets and quotients of matroids in Section \ref{sec:lifts}, yielding \eqref{lift-quotient} and \eqref{compressed-quotient}.
In Section \ref{sec:functions}, we go beyond \M-convex sets and extend the concept of quotient from \M-convex sets to M-convex functions, culminating in Theorems \ref{thm:functions-intro} and \ref{th:flags-intro}.
As an alternative generalization, we propose the study of quotients as purely polyhedral structures beyond combinatorial characterizations in Section \ref{sec:non-integral}.

\textbf{Acknowledgements.}
We thank Raman Sanyal, Diane Maclagan, Kemal Rose and Alex Levine for insightful conversations, and Alex Fink for pointing out related work to us.
We also thank Satoru Fujishige, Refael Hassin and Kazuo Murota for pointing us to related references and additional historical context.
We also thank two anonymous referees for helpful comments and suggestions which improved the exposition.
B.S. was supported by EPSRC grant EP/X036723/1.

\goodbreak

\section{\M-convex sets} \label{sec:M-convex}

We recall several preliminaries from discrete convex analysis, mainly referring to \cite{Murota:2003}. 

\subsection{Basics for \M-convex sets}

Let $E$ be a finite set and $x,y \in \Z^E$.
We define $\supp^+(x-y) = \set{i \in [n] \mid x_i - y_i > 0}$ and $\supp^-(x-y) = \supp^+(y-x)$. Given a subset $A \subseteq E$, we denote $x(A) = \sum_{i \in A} x_i$.
For each $i\in E$, we let $e_i$ denote the unit vector in the $i$-th coordinate.

\begin{definition}
%	Let $E$ be a finite set. % of size $|E|=n$.
	A non-empty set $P \subseteq \Z^E$ is \emph{\M-convex} if for all $x, y \in P$ and $i \in\supp^+(x-y)$, there exists $j \in \supp^-(x-y)$ such that $x-e_i+e_j  \in P$ and $y+e_i-e_j \in P$.
If $P \subseteq \ZZ^E$ is \M-convex, then $x(E)$ is the same quantity for all $x \in P$.
We call this the \emph{rank} of $P$ and denote it by $\rank(P)$. 
\end{definition}

In general, \M-convex sets may be unbounded.
However, much of the time it will be necessary to work with bounded \M-convex sets, i.e., those with only finitely many points.
From now on, we will assume an \M-convex set is bounded unless stated otherwise.

\begin{example}\label{ex:running1}
	Let $E = \{1,2,3\}$ and consider the set 
	\[
		P = \left\{			
			\sma 1 \\ 1 \\ -1 \strix, 
			\sma 1 \\ -1 \\ 1 \strix,  
			\sma -1 \\ 1 \\ 1 \strix,  
			\sma -1 \\ 0 \\ 2 \strix,  
			\sma 0 \\ -1 \\ 2 \strix,  
			\sma 1 \\ 0 \\ 0 \strix, 
			\sma 0 \\ 1 \\ 0 \strix,  
			\sma 0 \\ 0 \\ 1 \strix
		\right\} \subset \Z^E \ , 
	\]
	as depicted in \Cref{fig:running1}.
	It can be checked that $P$ is an \M-convex set. For example, when $x = (0,0,1)$ and $y=(1,1,-1)$, we have $\supp^+(x-y) = \{3\} = \{i\}$ and $\supp^-(x-y) = \{1,2\}$.
	Furthermore for all $j \in \supp^-(x-y)$, both elements $x - e_3 + e_j$ and $y + e_3 - e_j$ are contained in $P$.
Note that $x(E) = x_1 + x_2 + x_3 = 1$ for all $x \in P$.
	
	A second example of an \M-convex set is given by 
	\[
	Q = \left\{			
		\sma -1 \\ -3 \\ -1 \strix, \ 
		\sma -3 \\ -1 \\ -1 \strix, \ 
		\sma -1 \\ -1 \\ -3 \strix, \ 
		\sma -1 \\ -2 \\ -2 \strix, \ 
		\sma -2 \\ -2 \\ -1 \strix, \ 
		\sma -2 \\ -1 \\ -2 \strix
		\right\} \subset \Z^E \ ,
	\]
	contained in the hyperplane $x(E) = -5$, also depicted in \Cref{fig:running1}.
	
	\begin{figure}
		\centering
		\begin{subfigure}{0.39\textwidth}
	\centering
	\captionsetup{width=\linewidth}
	\begin{tikzpicture}[scale = 0.42]
		\filldraw[red!20] (8, 0) -- (4, 6) -- (0, 6) -- (-2, 3) --  (0, 0) -- (4, 0)  -- (8, 0);
		\filldraw (8, 0) circle (3pt);
		\node[anchor=north] at (8, 0) {\footnotesize{$ (1, 1, $-$1) $}};
		\filldraw (4, 6) circle (3pt);
		\node[anchor=south] at (4, 6) {\footnotesize{$ (1, $-$1, 1) $}};
		\filldraw (0, 0) circle (3pt);
		\node[anchor=north] at (0, 0) {\footnotesize{$ ($-$1, 1, 1) $}};
		\filldraw (-2, 3) circle (3pt);
		\node[anchor=east] at (-2, 3) {\footnotesize{$ ($-$1, 0, 2) $}};
		\filldraw (0, 6) circle (3pt);
		\node[anchor=south] at (0, 6) {\footnotesize{$ (0, $-$1, 2) $}};
		\filldraw (6, 3) circle (3pt);
		\node[anchor=west] at (6, 3) {\footnotesize{$ (1, 0, 0) $}};
		\filldraw (4, 0) circle (3pt);
		\node[anchor=north] at (4, 0) {\footnotesize{$ (0, 1, 0) $}};
		\filldraw (2, 3) circle (3pt);
		\node[anchor=north] at (2, 3) {\footnotesize{$ (0, 0, 1) $}};
	\end{tikzpicture}
	\caption{The M-convex set $P$.}
	\label{fig:m-convex-P}
\end{subfigure}
\begin{subfigure}{0.3\textwidth}
	\centering
	\captionsetup{width=\linewidth}
	\begin{tikzpicture}[scale = 0.42]
		\filldraw[blue!20] (8, 0) -- (4, 6) -- (0, 0) -- (4, 0)  -- (8, 0);
		\filldraw (4, 6) circle (3pt);
		\node[anchor=south] at (4, 6) {\footnotesize{$ ($-$1, $-$3, $-$1) $}};
		\filldraw (0, 0) circle (3pt);
		\node[anchor=north] at (0, 0) {\footnotesize{$ ($-$3, $-$1, $-$1) $}};
		\filldraw (8, 0) circle (3pt);
		\node[anchor=north] at (8, 0) {\footnotesize{$ ($-$1, $-$1, $-$3) $}};
		\filldraw (6, 3) circle (3pt);
		\node[anchor=west] at (6, 3) {\footnotesize{$ ($-$1, $-$2, $-$2) $}};
		\filldraw (2, 3) circle (3pt);
		\node[anchor=east] at (2, 3) {\footnotesize{$ ($-$2, $-$2, $-$1) $}};
		\filldraw (4, 0) circle (3pt);
		\node[anchor=north] at (4, 0) {\footnotesize{$ ($-$2, $-$1, $-$2) $}};
	\end{tikzpicture}
	\caption{The M-convex set $Q$.}
	\label{fig:m-convex-Q}
\end{subfigure}
\begin{subfigure}{0.25\textwidth}
	\centering
	\captionsetup{width=\linewidth}
	\begin{tikzpicture}[scale=1.3,
			inner sep=0.1mm,
		]
		\draw[->,thick] (0,0) -- (1,0);
		\node[anchor=south] at (1,0.05) {\small $e_1 - e_3$};
		\draw[->,thick] (0,0) -- (-1,0);
		\node[anchor=south] at (-1,0.05) {\small $e_3 - e_1$};
		\draw[->,thick] (0,0) -- (0.5,-0.86);
		\node[anchor=north] at (0.5,-0.86) {\small $e_2 - e_3$};
		\draw[->,thick] (0,0) -- (-0.5,0.86);
		\node[anchor=south] at (-0.5,0.86) {\small $e_3 - e_2$};
		\draw[->,thick] (0,0) -- (0.5,0.86);
		\node[anchor=south] at (0.5,0.86) {\small $e_1 - e_2$};
		\draw[->,thick] (0,0) -- (-0.5,-0.86);
		\node[anchor=north] at (-0.5,-0.86) {\small $e_2 - e_1$};
	  \node at (0,-1.2) {};
	\end{tikzpicture}
	\caption{The directions $e_i - e_j$.}
		\label{fig:m-convex-directions}
\end{subfigure}
		\caption{
			The lattice points in 	\textsc{(\subref{fig:m-convex-P})} and 	\textsc{(\subref{fig:m-convex-Q})} are the \M-convex sets from \Cref{ex:running1}, depicted inside the plane of points with coordinate sum equal to $1$ and $-5$, respectively.
		The shaded regions are their convex hulls, which are the submodular base polyhedra from \Cref{ex:running2}. %The directions $e_i - e_j$ are depicted in \textsc{(\subref{fig:m-convex-directions})}.
		}
		\label{fig:running1}
	\end{figure}
\end{example}

\begin{example} \label{ex:matroids-M-convex}
Let $B \subseteq \{0,1\}^E$ be the set of bases of a matroid.
The basis exchange axiom is precisely the \M-convex exchange axiom restricted to the unit hypercube, and so $B$ is an \M-convex set.
A notable example is the $m$-hypersimplex $\Delta(m,E) = \SetOf{x \in \{0,1\}^E}{x(E) = m}$, or equivalently the uniform matroid of rank $m$.
\end{example}

\begin{example}
  \label{ex:M-convex-subspaces}
    Recall the setup from \Cref{subsec:realizable-quotients}, where we are given a vector space $V$ and $(L_i)_{i \in E}$ subspaces of $V$.
    Let $(K_i)_{i \in E}$ be a tuple of subspaces with (i) $K_i \subseteq L_i$ and (ii) $\bigoplus_{i \in E} K_i = V$. 
    This gives rise to a lattice point $(\dim(K_i))_{i \in E}$ with coordinate sum $\dim(V)$. 
    Ranging over all choices of $(K_i)_{i \in E}$ gives rise to an \M-convex set.
\end{example}

\begin{definition}\label{def:submodularity}
	A map $p \colon 2^E \to \R$ is a \emph{submodular set function} if 
	\begin{equation*} \label{eq:submodular-inequalities}
	p(A \cup B) +  p(A \cap B) \leq  p(A) +  p(B) \forall A,B \subseteq E \, .
	\end{equation*}
%  We will always assume that $p(\emptyset) \neq \infty$.
  If $p$ is submodular, then $p + c$ is also submodular for all $c \in \R$.
  Hence, we will always normalize by setting $p(\emptyset) = 0$.
  A submodular set function is \emph{$\Z$-valued} (or \emph{integral}) if its range is contained in $\Z$. 
        
	The \emph{base polyhedron} of a submodular function is 
	\[
	B(p) = \set{x \in \R^E \mid x(A) \leq p(A) \forall A \subseteq E, \ x(E) = p(E)}
	\]
	and the \emph{submodular polyhedron} is 
	\[
	S(p) = \set{x \in \R^E \mid x(A) \leq p(A) \forall A \subseteq E} \, .
	\]
	It is straightforward to see that the submodular polyhedron is the base polyhedron summed with a negative orthant, i.e., $S(p) = B(p) + \RR^E_{\leq 0}$.
\end{definition}

Base polyhedra can also be characterized in terms of \emph{generalized permutohedra}~\cite{Postnikov:2009}, deformations of the permutohedron
\[
\Pi_E = \conv\left((\sigma(1), \sigma(2), \dots, \sigma(E)\right) \in \RR^E \mid \sigma \in \Sym(E))  \, ,
\]
such that edge directions are preserved. 
Explicitly, they are the convex hull of $|E|!$ (not-necessarily distinct) points $x_\sigma \in \RR^E$ labelled by permutations $\sigma \in \Sym(E)$ on the set $E$ such that for any adjacent transposition $s_i = (i, i+1)$, we have $x_{\sigma} - x_{\sigma \cdot s_i} = k_{\sigma, i}(e_{\sigma(i)} - e_{\sigma(i+1)})$ for some non-negative number $k_{\sigma, i} \in \RR_{\geq 0}$.
We can give an alternative characterization of \M-convex sets as the lattice points of \emph{integral generalized permutohedra}, those whose vertices have only integral coordinates; in~\cite{Murota:2003} these are called \emph{\M-convex polyhedra}.

\begin{theorem}[{\cite[\S 4.4 \& \S 4.8]{Murota:2003}}]\label{thm:Mconv+submod+permutohedra+correspondence}
There is a one-to-one correspondence between \M-convex sets, $\ZZ$-valued submodular set functions and integral generalized permutohedra: 
\begin{enumerate}
	\item If $p$ is a $\Z$-valued submodular set function then $P = B(p)\cap \Z^E$ is an \M-convex set.
	\item If $P$ is a \M-convex set then $p\colon 2^E \rightarrow \ZZ$ defined as
	\[
	p(A) = \max\SetOf{x(A)}{x \in P}
	\]
	is a submodular set function with $P = B(p) \cap \ZZ^E$.
  \item The set $P$ is \M-convex if and only if $P = \conv(P) \cap \Z^E$ and $\conv(P)$ is an integral generalized permutohedron, i.e., $x_\sigma \in \ZZ^E$ for all $\sigma \in \Sym(E)$.
\end{enumerate}
\end{theorem}

Given this theorem, we define the \emph{vertices} of $P$ to be the vertices $x_{\sigma}$ of the integral generalized permutohedron $\conv(P)$.
We can also define $x_\sigma$ independently of $\conv(P)$ as the unique maximizer of the objective function $\sum c_i\cdot x_i$ where $c_{\sigma(1)} > c_{\sigma(2)} > \cdots > c_{\sigma(n)}$.
Note that it is possible that several permutations $\sigma$ will define the same vertex. 

\begin{example}\label{ex:running2}
	Let $E = \{1,2,3\}$, and recall the $\M$-convex sets $P,Q \subset \ZZ^E$ from \Cref{ex:running1}.
	Using the correspondence in \Cref{thm:Mconv+submod+permutohedra+correspondence}, we construct their associated submodular functions $p,q \colon 2^E \rightarrow \ZZ$.
	Writing $ab$ for the set $\{a,b\}$, we consider the two set functions \\
		\begin{minipage}{0.49\textwidth}
		\begin{align*}
			p : 2^{E} &\to \Z &&&&& \\[-.5em]
			\emptyset &\mapsto 0 &
			1 &\mapsto 1 &
			2 &\mapsto 1 &
			3 &\mapsto 2 \\[-.5em]
			12 &\mapsto 2 &
			13 &\mapsto 2 &
			23 &\mapsto 2 &
			123 &\mapsto 1 , 
		\end{align*}
	\end{minipage}
	\begin{minipage}{0.49\textwidth}
		\begin{align*}
		q : 2^{E} &\to \Z &&&&& \\[-.5em]
			\emptyset &\mapsto 0 &
			1 &\mapsto -1 &
			2 &\mapsto -1 &
			3 &\mapsto -1 \\[-.5em]
			12 &\mapsto -2 &
			13 &\mapsto -2 &
			23 &\mapsto -2 &
			123 &\mapsto -5 , 
		\end{align*}
	\end{minipage}
	
	defined by the relation $p(A) = \max\{ x(A) \mid x \in P\}$ and $q(A) = \max\{y(A) \mid y \in Q\}$ for all $A \subseteq E$.
	It can be verified that both set functions are submodular.
	\Cref{fig:running1,fig:running2} show the base polyhedra of both functions.
	%Observe that $p$ is the submodular function corresponding to the \M-convex set $P$ from \Cref{ex:running1}, as given by the correspondence in \Cref{thm:Mconv+submod+permutohedra+correspondence}.
	As given \Cref{thm:Mconv+submod+permutohedra+correspondence}, the \M-convex set $P = B(p) \cap \ZZ^E$ is precisely the lattice points of the base polyhedron $B(p)$. 
	Conversely, the base polyhedron $B(p) = \conv(P)$ is an integral generalized permutohedron as all edges are parallel to some $e_i - e_j$.
	%Moreover, it can be checked that $p(A) = \max\{ x(A) \mid x \in P\}$ holds for all $A \subseteq \{1,2,3\}$.
	The same correspondence holds between $Q$ and $q$: we have $Q = B(q) \cap \Z^E$ and the base polyhedron $B(q) = \conv(Q)$ has all edges in directions $e_i - e_j$.%, and $q(A) = \max\{y(A) \mid y \in Q\}$ for all $A \subseteq \{1,2,3\}$. 
		\begin{figure}
		\centering
		\begin{subfigure}{0.45\textwidth}
	\centering
	\captionsetup{width=\linewidth}
\begin{tikzpicture}[scale = 0.5]
	\filldraw[red!20] (8, 0) -- (4, 6) -- (0, 6) -- (-2, 3) --  (0, 0) -- (4, 0)  -- (8, 0);
	\draw[red!40, very thick] (8, 0) -- (4, 6) -- (0, 6) -- (-2, 3) --  (0, 0) -- (4, 0)  -- (8, 0);
	%\filldraw (8, 0) circle (3pt);
	%\node[anchor=north] at (8, 0) {\footnotesize{$ (1, 1, $-$1) $}};
	\node[anchor = north] at (8,0) {$x_{231} = x_{321}$};
	%\filldraw (4, 6) circle (3pt);
	%\node[anchor=north] at (4, 6) {\footnotesize{$ (1, $-$1, 1) $}};
	\node[anchor = south] at (4,6) {$x_{312}$};
	%\filldraw (0, 0) circle (3pt);
	%\node[anchor=north] at (0, 0) {\footnotesize{$ ($-$1, 1, 1) $}};
	\node[anchor = north east] at (0,0) {$x_{132}$};
	%\filldraw (-2, 3) circle (3pt);
	%\node[anchor=north] at (-2, 3) {\footnotesize{$ ($-$1, 0, 2) $}};
	\node[anchor = east] at (-2,3) {$x_{123}$};
	%\filldraw (0, 6) circle (3pt);
	%\node[anchor=north] at (0, 6) {\footnotesize{$ (0, $-$1, 2) $}};
	\node[anchor = south] at (0,6) {$x_{213}$};
	%\filldraw (6, 3) circle (3pt);
	%\node[anchor=north] at (6, 3) {\footnotesize{$ (1, 0, 0) $}};
	%\filldraw (4, 0) circle (3pt);
	%\node[anchor=north] at (4, 0) {\footnotesize{$ (0, 1, 0) $}};
	%\filldraw (2, 3) circle (3pt);
	%\node[anchor=north] at (2, 3) {\footnotesize{$ (0, 0, 1) $}};
\end{tikzpicture}
\caption{The submodular base polyhedron $B(p)$.}
\end{subfigure}
\begin{subfigure}{0.45\textwidth}
	\centering
	\captionsetup{width=\linewidth}
\begin{tikzpicture}[scale = 0.5]
	\filldraw[blue!20] (8, 0) -- (4, 6) -- (0, 0) -- (4, 0)  -- (8, 0);
	\draw[blue!40, very thick] (8, 0) -- (4, 6) -- (0, 0) -- (4, 0)  -- (8, 0);
	%\filldraw (4, 6) circle (3pt);
	%\node[anchor=north] at (4, 6) {\footnotesize{$ ($-$1, $-$3, $-$1) $}};
	\node[anchor = south] at (4,6) {$y_{213} = y_{312}$};
	%\filldraw (0, 0) circle (3pt);
	%\node[anchor=north] at (0, 0) {\footnotesize{$ ($-$3, $-$1, $-$1) $}};
	\node[anchor = north] at (0,0) {$y_{123} = y_{132}$};
	%\filldraw (8, 0) circle (3pt);
	%\node[anchor=north] at (8, 0) {\footnotesize{$ ($-$1, $-$1, $-$3) $}};
	\node[anchor = north] at (8,0) {$y_{231} = y_{321}$};
	%\filldraw (6, 3) circle (3pt);
	%\node[anchor=north] at (6, 3) {\footnotesize{$ ($-$1, $-$2, $-$2) $}};
	%\filldraw (2, 3) circle (3pt);
	%\node[anchor=north] at (2, 3) {\footnotesize{$ ($-$2, $-$2, $-$1) $}};
	%\filldraw (4, 0) circle (3pt);
	%\node[anchor=north] at (4, 0) {\footnotesize{$ ($-$2, $-$1, $-$2) $}};
\end{tikzpicture}
\caption{The submodular base polyhedron $B(q)$.}
\end{subfigure}
		\caption{The submodular base polyhedra from \Cref{ex:running2}, whose vertices $x_\sigma, y_\sigma$ are labelled by permutations $\sigma \in \operatorname{Sym}(E)$ in one-line notation. Coordinates of the vertices are given in \Cref{fig:running1}.}
		\label{fig:running2}
	\end{figure}
\end{example}

\begin{remark}
The one-to-one correspondence between \M-convex sets and submodular functions from \Cref{thm:Mconv+submod+permutohedra+correspondence} can be extended to unbounded \M-convex sets by allowing the submodular function $p$ to take infinite values.
Explicitly, the function $p \colon 2^E \rightarrow \ZZ \cup \{\infty\}$ defined as
\[
p(A) = \sup\SetOf{x(A)}{x \in P}
\]
is the unique submodular function such that $P = B(p) \cap \ZZ^E$.
Allowing infinite values can be problematic for a number of quotient characterizations later, and so we restrict to submodular functions taking finite values unless explicitly stated.
\end{remark}

We are now ready to define quotients of \M-convex sets.

\begin{definition}\label{def:compliant}
  Let $P,Q \subseteq \ZZ^E$ be \M-convex sets with corresponding submodular set functions $p,q \colon 2^E \rightarrow \ZZ$.
We say the submodular functions are \emph{compliant} and write $p \quotient q$ if they satisfy
\begin{equation*}
q(Y) - q(X) \leq p(Y) - p(X) \, , \quad \forall X \subseteq Y \subseteq E \, . \label{eq:compliant}
\end{equation*}
We say $Q$ is a \emph{quotient} of $P$, and write $P \quotient Q$, if their submodular functions are compliant.
In the case that $\rank(P) = \rank(Q) + 1$, we call it an \emph{elementary quotient}.
\end{definition}

Our main \Cref{thm:quotient} states that there are many equivalent ways of defining quotients.
We can immediately state an equivalence in terms of vertices due to~\cite{IwataMurotaShigeno:1997}.

\begin{theorem}[\cite{IwataMurotaShigeno:1997}]\label{thm:quotient+iff+bases+contained}
Let $P,Q \subseteq \ZZ^E$ be \M-convex sets.
Then $P \quotient Q$ if and only if for every $\sigma \in \Sym(E)$, the vertices $x_\sigma \in P$ and $y_\sigma \in Q$ satisfy $x_\sigma \geq y_\sigma$.
\end{theorem}

\begin{example}
	The submodular functions $p$ and $q$ from \Cref{ex:running2} form a compliant pair of submodular set functions. Thus, the M-convex sets from \Cref{ex:running1} form a quotient $P \quotient Q$. Furthermore, \Cref{fig:running1,fig:running2} show that the vertices $x_\sigma \in P$ and $y_\sigma \in Q$ satisfy $x_\sigma \geq y_\sigma$ in the partial order on $\ZZ^E$, as implied by \Cref{thm:quotient+iff+bases+contained}.
\end{example}

 \Cref{thm:quotient+iff+bases+contained} gives (\eqref{submod-fctns-compliant} $\iff$ \eqref{bases+containment}) in \Cref{thm:quotient}.
The remaining equivalences will be given in \Cref{sec:M-convex}, \Cref{sec:induction} and \Cref{sec:lifts} once necessary concepts have been introduced.

\subsection{\texorpdfstring{$\Mnat$}{Mnat}-convex sets and generalized polymatroids}

\M-convex sets are closed under a number of operations as listed in \Cref{sec:operations}.
However, they are not closed under coordinate projections $\pi_{E} \colon \ZZ^{\tE} \rightarrow \ZZ^E$ where $E \subseteq \tE$.
This leads to a wider class of discrete convex sets.

\begin{definition}\label{def:mnatural-sets}
	A set $R\subseteq \ZZ^E$ is an \emph{$\Mnat$-convex set} if for all $x, y \in R$:
	\begin{enumerate}
		\item if $x(E) > y(E)$, there exists $i \in \supp^+(x - y)$ such that $x - e_i \in R$ and $y + e_i \in R$, \label{Mnat:aug}
		\item if $x(E) = y(E)$, for all $i \in \supp^+(x-y)$ there exists some $j \in \supp^-(x-y)$ such that $x + e_i - e_j \in R$ and $y - e_i + e_j \in R$. \label{Mnat:exc}
	\end{enumerate}
\end{definition}
Note that there are a number of equivalent definitions of $\Mnat$-convex sets.
We choose to use the definition from~\cite{MurotaShioura:2018simpler} as the second condition just recovers the definition of an \M-convex set for points with constant coordinate sum. 

\begin{example}\label{ex:mnat}
	Let $E = \{1,2\}$ and 
	\[
		R = \left\{
			\sma -1 \\ 0 \strix, 
			\sma 0 \\ -1 \strix,
			\sma -1 \\ 1 \strix,
			\sma 0 \\ 0 \strix,
			\sma 1 \\ -1 \strix,
			\sma 0 \\ 1 \strix,
			\sma 1 \\ 0 \strix,
			\sma 1 \\ 1 \strix,
		\right\} \subset \Z^E \ ,
	\]
	as depicted in \Cref{fig:mnat}. It can be checked that $R$ satisfies the conditions from \Cref{def:mnatural-sets}, i.e., is an M$^\natural$-convex set. This set of points is obtained by projecting away the last coordinate of the $M$-convex set $P$ from \Cref{ex:running1}.

	\begin{figure}
		\centering
		\begin{tikzpicture}[scale = 1]
	\filldraw[red!20] (-1,1) -- (-1,0) -- (0,-1) -- (1,-1) -- (1,1) -- (-1,1);
	\filldraw (-1, 0) circle (2pt);
	\node[anchor=east] at (-1, 0) {\small{$($-$1,0)$}};
	\filldraw (0, -1) circle (2pt);
	\node[anchor=east] at (0,-1) {\small{$ (0,$-$1) $}};
	\filldraw (-1, 1) circle (2pt);
	\node[anchor=east] at (-1,1) {\small{$ ($-$1,1) $}};
	\filldraw (0, 0) circle (2pt);
	\node[anchor=south] at (0,0) {\small{$ (0,0) $}};
	\filldraw (1, -1) circle (2pt);
	\node[anchor=west] at (1, -1) {\small{$ (1,$-$1) $}};
	\filldraw (0, 1) circle (2pt);
	\node[anchor=south] at (0, 1) {\small{$ (0, 1) $}};
	\filldraw (1, 0) circle (2pt);
	\node[anchor=west] at (1, 0) {\small{$ (1, 0) $}};
	\filldraw (1, 1) circle (2pt);
	\node[anchor=west] at (1, 1) {\small{$ (1,1) $}};
\end{tikzpicture}
		\caption{The M$^\natural$-convex set from \Cref{ex:mnat} and its convex hull.}
		\label{fig:mnat}
	\end{figure}
\end{example}

\begin{example}
The prototypical example of an $\Mnat$-convex set is the collection of independent sets $I \subseteq \{0,1\}^E$ of a matroid.
In particular, axiom~\eqref{Mnat:aug} can be deduced directly from the augmentation axiom for independent sets.
Axiom~\eqref{Mnat:exc} holds as independent sets of the same cardinality again form the bases of a matroid, namely a truncation of the original matroid.
By \Cref{ex:matroids-M-convex}, these form an \M-convex set.
More generally, $\Mnat$-convex sets within the unit cube are known as \emph{generalized matroids}. 
\end{example}

The following result gives an alternative characterization of $\Mnat$-convex sets as projections of \M-convex sets.
\begin{proposition}[{\cite[Section 4.7]{Murota:2003}}]\label{prop:Mnat+lifts+to+M}
A set $R\subseteq\ZZ^E$ is $\Mnat$-convex if and only if there exists an \M-convex set $P \subseteq \ZZ^\tE$ such that $\pi_E(P) = R$ where $\pi_E$ is the coordinate projection from $\ZZ^\tE$ to~$\ZZ^E$.
\end{proposition}
	%{\color{red} It follows from \cite{MurotaShioura:2018simpler} that if $x(E)<y(E)$ then (ii) always holds.}
	Let $R\subseteq \ZZ^E$ be an $\Mnat$-convex set and consider the hyperplane $H_k = \{x \in \R^E \mid \sum_{i \in E} x_i = k\}$.
	If $R \cap H_k \neq \emptyset$ then we call $R \cap H_k$ a \emph{layer} of $R$.
	We denote the top and bottom layer of $R$ respectively by:
	\[
	R^\uparrow = R \cap H_a \, , \, a = \max\SetOf{k \in \ZZ}{R \cap H_k \neq \emptyset} \quad , \quad 	R^\downarrow = R \cap H_b \, , \, b = \min\SetOf{k \in \ZZ}{R \cap H_k \neq \emptyset} \, .	
	\]
It follows directly from~\eqref{Mnat:exc} of \Cref{def:mnatural-sets} that a layer of an $\Mnat$-convex set is \M-convex.
We note that \M-convex sets are $\Mnat$-convex sets with a single layer, as~\eqref{Mnat:aug} of \Cref{def:mnatural-sets} is trivially satisfied.

We can give an additional characterization of $\Mnat$-convex sets in terms of `generalized polymatroids', polyhedra defined by a pair of submodular and supermodular set functions.
A map $q$ is an (integral) \emph{supermodular set function} if $-q$ is an (integral) submodular set function. 
  %% A map $q: 2^E \to \Z\cup\{-\infty\}$ is an (integral) \emph{supermodular set function} if 
  %%       \[
  %%       q(A \cup B) +  q(A \cap B) \geq  q(A) +  q(B) \forall A,B \subseteq E \, .
  %%       \]
  %% As with submodular set functions, we can normalize $q$ such that $q(\emptyset) = 0$.
Its associated \emph{supermodular polyhedron} and base polyhedron are
	\[
	S^\#(q) = \set{x \in \R^E \mid x(A) \geq q(A) \forall A \subseteq E} \quad , \quad B(q) = S^\#(q) \cap \{x \in \RR^E \mid x(E) = q(E)\} \, .
	\]
In analogy with submodular polyhedra, the supermodular polyhedron is the base polyhedron summed with the positive orthant, i.e., $S^\#(q) = B(q) + \RR^E_{\geq 0}$.

Given a submodular set function $p$, we define its \emph{dual supermodular function}  to be
\begin{align*}
p^\# : 2^E \to \ZZ \quad , \quad p^\#(A) = p(E) - p(E\setminus A) \, .
\end{align*}
Note that dualizing reflects the submodular polyhedron through the hyperplane $H_{p(E)}$ to obtain the supermodular polyhedron, i.e., $S^\#(p^\#) = B(p) + \RR^E_{\geq 0}$.
In particular, the base polyhedra $B(p) = B(p^\#)$ are the same.
%\bs{I am fairly sure this is where the referee wanted the remark that under condition (2) we have $B(p)=B(q)$ when $p(E)=q(E)$. Its previous location completely breaks up the flow of the `proof' of Lemma 2.13}

\begin{definition}\label{def:g-polymatroid}
Let $p$ be a ($\ZZ$-valued) submodular function and $q$ a ($\ZZ$-valued) supermodular function such that 
\begin{equation} \label{eq:g-polymatroid-compliant}
p(A)-q(B)\geq p(A \setminus B)-q(B \setminus A) \; \text{ for all } \; A,B \subseteq E \, .
\end{equation}
For such a pair, we define the (integral) \emph{generalized polymatroid} $G(p,q)$ to be
\[
G(p,q) = \set{x \in \R^E \mid q(A) \leq x(A) \leq p(A) \text{ for all } A \subseteq E }.
\]
\end{definition}
Note that 
\begin{equation*}
G(p,q) = S(p) \cap S^{\#}(q) = \left(B(p) + \R_{\leq 0}^E\right) \cap \left(B(q) + \R_{\geq 0}^E\right) \, .
\end{equation*}
Furthermore, observe that if $p(E) = q(E)$ then \eqref{eq:g-polymatroid-compliant} implies that $B(p) = B(q)$.

From now on, we assume that all functions are $\ZZ$-valued and polyhedra are integral unless otherwise stated, so we drop the terms.

\begin{figure}
	\includegraphics[width=0.5\textwidth]{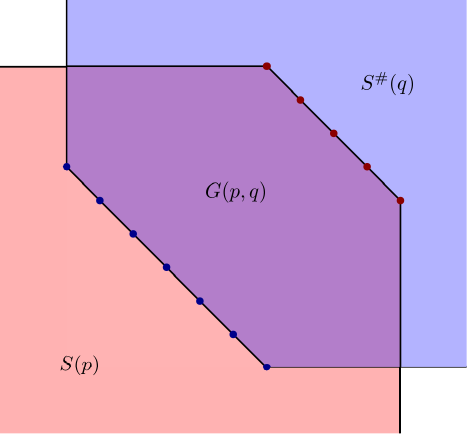}
	\caption{A generalized polymatroid $G(p,q)$, marked purple, as the intersection of the submodular and supermodular polyhedra $S(p)$ and $S^\#(q)$.
		The corresponding \M-convex sets are the lattice points marked in red and blue respectively.}
	\label{fig:g-polymatroid}
\end{figure}

The generalized polymatroid $G(p,q)$ is the intersection $S(p) \cap S^\#(q)$ of the submodular and supermodular polyhedra corresponding to $p$ and $q$, as depicted in \Cref{fig:g-polymatroid}.
The condition~\eqref{eq:g-polymatroid-compliant} ensures that $B(q) \subseteq S(p)$ and $B(p) \subseteq S^{\#}(q)$, hence the base polyhedra $B(p)$ and $B(q)$ are untouched by this intersection; %"In particular, for $p(E) = q(E)$ one has $B(p) = B(q)$
this condition was already identified in~\cite{Hassin:1982}.
It follows that these form the top and bottoms layers of $G(p,q)$ respectively, see~\cite[Lemma 2.4]{FrankKiraly:2009}. 

\begin{lemma}\label{lem:top-bottom-layers}
	Let $p,q$ be submodular functions such that $G(p,q^\#)$ is a generalized polymatroid.
	Then the top layer is $B(p) = G(p,q^\#)^\uparrow$ and the bottom layer is $B(q) = G(p,q^\#)^\downarrow$. 
\end{lemma}

The following theorem shows that there is a one-to-one correspondence between $\Mnat$-convex sets and generalized polymatroids.

\begin{theorem}[{\cite[Section 4.7]{Murota:2003}, \cite[Theorem 3.58]{Fujishige:2005}}]\label{thm:Mnat+convex}
	The following are equivalent:
	\begin{enumerate}
		\item The set $R \subseteq \ZZ^E$ is $\Mnat$-convex.
		\item There exist submodular functions $p,q: 2^E \to \Z$ such that $R = G(p,q^\#) \cap \Z^E$ where $G(p,q^\#)$ is a generalized polymatroid.
		\item There exists a submodular function $r: 2^{E \sqcup e} \to \Z$ such that 
		\[
		p(X) = r(X) \quad , \quad  q^\#(X) = r^\#(X) %= r(E\cup e) - r((E \setminus X) \cup e) 
		\, \text{ for all } X \subseteq E
		\] 
		and $G(p,q^\#)$ is the coordinate projection of $B(r)$ under $\pi_E\colon \RR^{E \sqcup e} \rightarrow \RR^E$.
	\end{enumerate}
\end{theorem}

\begin{example}\label{ex:running3}
	\Cref{fig:running3} shows the generalized polymatroid $G(p,q^\#)$ associated to the submodular functions $p,q$ in \Cref{ex:running2}. By \Cref{thm:Mnat+convex}, its lattice points $R = G(p,q^\#) \cap \Z^E$ form an $\Mnat$-convex set. It is the coordinate projection of the submodular polyhedron $B(r) \subset \R^{E \sqcup e}$, where $E \sqcup e = \{1,2,3,4\}$ and $r : 2^{E \sqcup e} \to \Z$ is defined by
	\vspace{-1em}
	\begin{center}
		\begin{minipage}{0.15\textwidth}
		\begin{align*}
			\emptyset &\mapsto 0 \\[-.5em]
			1 &\mapsto 1 \\[-.5em]
			2 &\mapsto 1 \\[-.5em]
			3 &\mapsto 2 
					\end{align*}
	\end{minipage}
	\begin{minipage}{0.15\textwidth}
		\begin{align*}
			12 &\mapsto 2 \\[-.5em]
			13 &\mapsto 2 \\[-.5em]
			23 &\mapsto 2 \\[-.5em]
			123 &\mapsto 1 
		\end{align*}
	\end{minipage}
	\begin{minipage}{0.15\textwidth}
		\begin{align*}
			4 &\mapsto c + 5 \\[-.5em]
			14 &\mapsto c+4 \\[-.5em]
			24 &\mapsto c+4 \\[-.5em]
			34 &\mapsto c+4
					\end{align*}
	\end{minipage}
	\begin{minipage}{0.15\textwidth}
		\begin{align*}
			124 &\mapsto c+3 \\[-.5em]
			134 &\mapsto c+3 \\[-.5em]
			234 &\mapsto c+3 \\[-.5em]
			1234 &\mapsto c \   
		\end{align*}
	\end{minipage}
	\end{center}
 where $c \in \Z$ is fixed but arbitrary.
\end{example}
	
	\begin{figure}
		\includegraphics[height=18em]{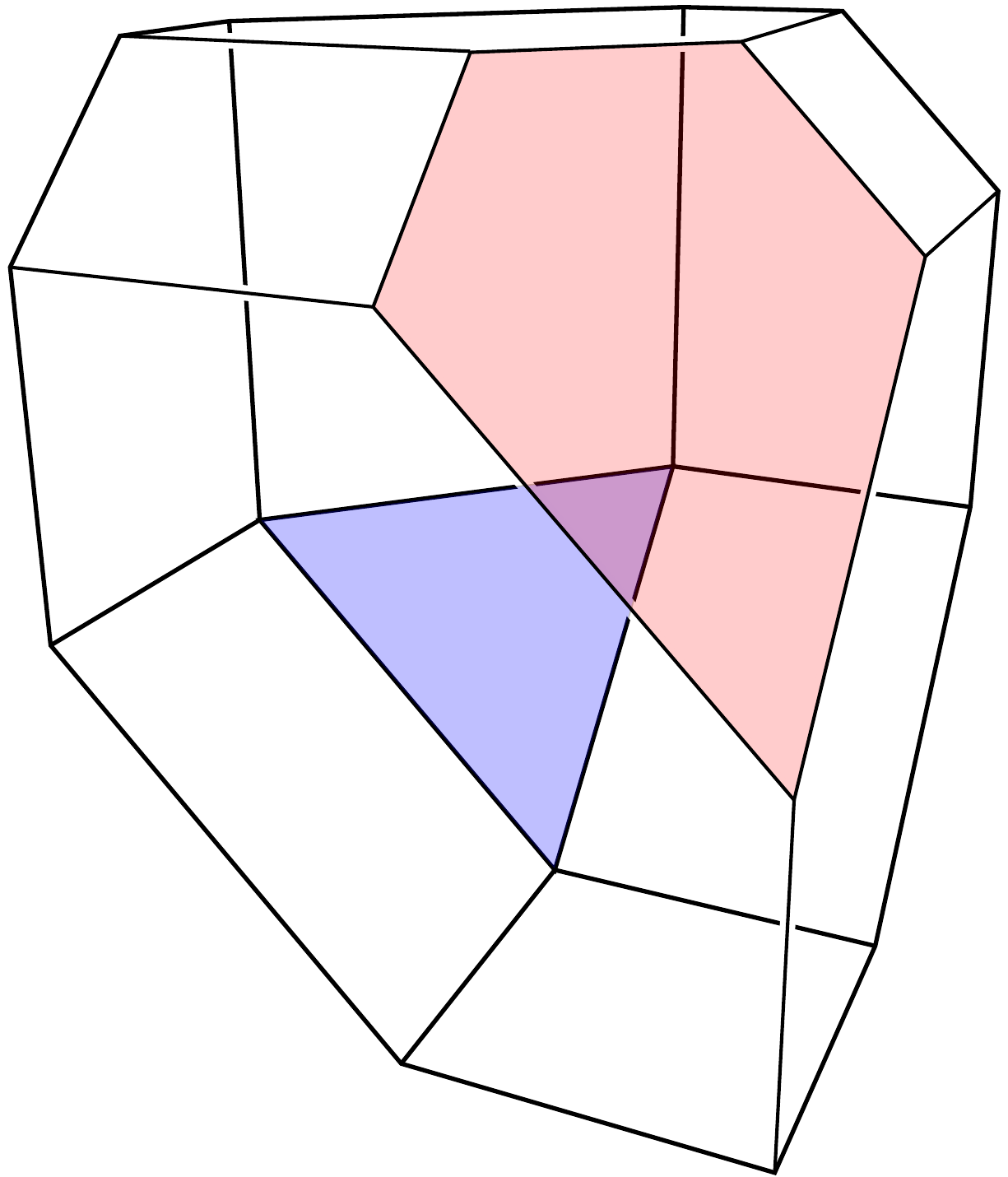}
		\caption{The generalized polymatroid from \Cref{ex:running3}, having the two generalized permutohedra from \Cref{ex:running2} as top and bottom face.}
		\label{fig:running3}
	\end{figure}

\subsection{Basic operations}\label{sec:operations}

We are now ready to recall a number of operations on \M-convex sets.
Throughout we let $P \subseteq \ZZ^E$ be an \M-convex set of rank $r$.
Recall that $H_k$ is the set of points in $\RR^E$ with coordinate sum $k$.

	Write $E = V \sqcup U$ as a disjoint union.
The \emph{restriction} of $P$ to $V$ is the \M-convex set
	\[
	P|_V =	\{x \in \Z^V \mid (x,\0) \in P\}\, .
	\]
	The \emph{projection} of $P$ onto $V$ is the $\Mnat$-convex set
	\[
	\pi_V(P) = \{x \in \Z^V \mid \exists y\in \Z^U \text{ such that } (x,y) \in P\} \, .
	\]

	The \emph{Minkowski sum} of two \M-convex sets $P_1,P_2 \subseteq \Z^E$ is the \M-convex set
	\[
	P_1 + P_2 = \{x_1 + x_2 \in \ZZ^E \mid x_1 \in P_1 \, , \, x_2 \in P_2\}.
	\]
As a special case, when one of these is just a single point, we call $P + v$ the \emph{translation} of $P$ by $v$.

The \emph{truncation} $P^{tr}$ of $P$ is the \M-convex set
\[
P^{tr} = (P + \ZZ^E_{\leq 0}) \cap H_{r-1} \subseteq \ZZ^E \, ,
\]
the layer below $P$ in its submodular polyhedron.
The $k$-th truncation $P_k^{tr}$ is obtained by truncating $P$ $k$ times, or equivalently as the integer points $k$ layers below $P$ in its submodular polyhedron.

The \emph{elongation} $P^{el}$ of $P$ is the \M-convex set
\[
P^{el} = (P + \ZZ^E_{\geq 0}) \cap H_{r+1} \subseteq \ZZ^E \, ,
\]
the layer above $P$ in its supermodular polyhedron. 
The $k$-th elongation $P_k^{el}$ is obtained by elongating $P$ $k$ times, or equivalently as the integer points $k$ layers above $P$ in its supermodular polyhedron.

\begin{remark}
  \label{rem:truncation-submodular-vs-polymatroid}
This notion of truncation differs from the usual matroid theoretic definition.
Given a rank $r$ matroid with independent sets $I \subseteq 2^E$ viewed as an $\Mnat$-convex set, its bases are the $r$-th layer of $I$ while the bases of its (matroid) truncation are the $(r-1)$-th layer of $I$.
This is because matroids are naturally restricted to the unit hypercube $2^E$.
Moreover, they come with a canonical $\Mnat$-convex set in the unit hypercube, namely its independent sets.
More general \M-convex sets are not restricted in this way, and so their canonical associated $\Mnat$-convex set is the submodular polyhedron.

If instead one considers \emph{polymatroids}, namely non-negative, non-decreasing integral submodular functions, then the associated \emph{polymatroid polytope} is the submodular polyhedron intersected with the non-negative orthant.
The \M-convex set $P \subseteq \ZZ_{\geq 0}^E$ associated with a polymatroid is also contained in the non-negative orthant, and so its polymatroid truncation is $P^{tr} \cap \ZZ_{\geq 0}^E$. 
\end{remark}

Most of these operations correspond to an operation at the level of submodular functions, as the following lemma summarizes from~\cite[Section 3.1]{Fujishige:2005}.
\begin{lemma}
  Let $P$ be an \M-convex set with corresponding submodular set function $p: 2^E \to \Z$.

The submodular set function associated to the truncation $P^{tr}$ is 
\[
p^{tr} \colon 2^E \to \Z \quad , \quad p^{tr}(A) = 
\begin{cases} p(A) & A \subsetneq E \\ p(E) - 1 & A = E  \end{cases} \, .
\]
The supermodular set function associated to the elongation $P^{el}$ is 
\[
p^{el} \colon 2^E \to \Z \quad , \quad p^{el}(A) = 
\begin{cases} p^\#(A) & A \subsetneq E \\ p^\#(E) + 1 & A = E  \end{cases} \, .
\]
The submodular set function associated to the Minkowski sum $P_1 + P_2$ is $p_1 + p_2$.
\end{lemma}

We note that the projection of an \M-convex set is an $\Mnat$-convex set in general, and so the projection has no associated submodular function.
The restriction of an \M-convex set is still an \M-convex set, however we shall recover its submodular function as a special case in \Cref{sec:minors}.

\M-convex and $\Mnat$-convex sets are not closed under intersections in general, but they are when one of the sets is of a special form.
Given $a, b \in \ZZ^E$ with $a \leq b$, the associated \emph{box} is
\[
[a,b] = \SetOf{x \in \ZZ^E}{a \leq x \leq b} \, .
\]
Given $\alpha, \beta \in \ZZ$ with $\alpha \leq \beta$, the \emph{plank} $K(\alpha, \beta)$ is
\[
K(\alpha,\beta) = \SetOf{x \in \ZZ^E}{\alpha \leq x(E) \leq \beta} \, .
\]
Intersection with boxes and planks preserves $\Mnat$-convexity.

\begin{theorem}[{\cite[Section 14.3]{Frank:2011}}] \label{thm:box+plank}
Boxes and planks are $\Mnat$-convex.
Given an $\Mnat$-convex set $P = G(p,q) \cap \ZZ^E$, the intersections $P \cap [a,b]$ and $P \cap K(\alpha, \beta)$ are also $\Mnat$-convex, where
\begin{align*}
P \cap [a,b] &= G(p',q') \cap \ZZ^E\, , & P \cap K(\alpha, \beta) &= G(\tilde{p},\tilde{q}) \cap \ZZ^E \, ,\\
p'(Z) &= \min_{X \subseteq E}(p(X) - a(X-Z) + b(Z-X))\, , & \tilde{p}(Z) &= \min(p(Z), \beta - q(E-Z)) \, , \\
q'(Z) &= \max_{X \subseteq E}(q(X) - b(X-Z) + a(Z-X))\, , & \tilde{q}(Z) &= \max(q(Z), \alpha - p(E-Z)) \, . 
\end{align*}
\end{theorem}
We end with a key observation linking quotients and truncations.
Not only are truncations quotients, they are in some sense maximal or generic quotients.

\begin{proposition} \label{prop:truncation+quotient}
Let $P$ be an \M-convex set.
For any $k \in \ZZ_{>0}$, the $k$-th truncation $P_k^{tr}$ is a quotient of $P$.
Moreover, for any quotient $P \quotient Q$ with $k = \rank(P) - \rank(Q)$, we have $Q \subseteq P_{k}^{tr}$.
\end{proposition}
\begin{proof}
	For the first statement, we use the compliant definition of quotient. 
	For all $X \subseteq Y \subsetneq E$, we immediately have $p^{tr}_k(Y) - p^{tr}_k(X)  = p(Y) - p(X)$ and $p^{tr}_k(E) - p^{tr}_k(X) < p(E) - p(X)$, so the truncation is a quotient of $P$.
	
	For the second statement, if $Q$ is a quotient of $P$ with submodular function $q$ such that $k = \rank(P) - \rank(Q)$, then $q(E) = p_k^{tr}(E)$.
	Recall that we can assume that the $q(\emptyset) = p(\emptyset)=0$.
	Then for any $A \subsetneq E$, we have $q(A) = q(A) - q(\emptyset) \leq p(A) - p(\emptyset) = p_k^{tr}(A)$.
\end{proof}

\begin{example}
  Recall the setup of \Cref{subsec:realizable-quotients}, and consider the case where the linear projection $\phi \colon V \rightarrow W$ is generic and $\dim(W) = \dim(V)-1$.
  Then, let $(\dim(K_i))_{i \in E}$ be a lattice point as considered in \Cref{ex:M-convex-subspaces}.
  Applying such a generic $\phi$ results in a point $(\dim(\phi(K_i)))_{i \in E}$ where exactly one of the coordinates is diminished by one.
  Overall this gives rise to a polymatroid truncation as discussed in \Cref{rem:truncation-submodular-vs-polymatroid}.
\end{example}

\subsection{Minors}\label{sec:minors}
Unlike with matroids, there is some flexibility for the definition of a minor for submodular functions and \M-convex sets. 
We propose the following definitions, noting some properties relating to quotients. As it turns out, these definitions will allow us to easily construct large classes of quotients of \M-convex sets (\Cref{prop:minors-quotients}).

  Throughout, we let $P \subseteq \ZZ^E$ be an \M-convex set where $E = V \sqcup U$.	
	Given some $k\in \ZZ$, the \emph{basic $k$-{th} minor} of $P$ with respect to $U$ is the bounded \M-convex set 
	\[
	P^U_k = \{x \in \Z^V \mid \exists y\in \Z^U \text{ such that } (x,y) \in P \text{ and } x(V) = k\} = \pi_V(P) \cap H_k \, .
	\] 
	A \emph{minor} of $P$ is a sequence of basic minors.
	
As special cases of basic minors, we define
  \begin{itemize}
  \item	the \emph{deletion} $P\setminus U$ of $U$ is the minor $P^U_k$ for $k = \max\{x(V) \mid (x,y) \in P\} = p(V)$, 
  \item the \emph{contraction} $P/U$ of $U$ is the minor $P^U_k$ for $k = \min\{x(V)\mid (x,y) \in P\} = p(E) - p(U)$.
  \end{itemize}
  We emphasize the following important observation relating deletion and contraction to layers of $\Mnat$-convex sets.
  \begin{lemma}\label{lem:deletion+contraction+layer}
  Let $P \subseteq \ZZ^E$ be an \M-convex set with $E = V \sqcup U$.
  The deletion $P \setminus U$ is the top layer of the $\Mnat$-convex set $\pi_V(P)$ and the contraction $P/ U$ is the bottom layer of $\pi_V(P)$.
  \end{lemma}  
  
  We also note that we can view restriction as a minor, namely a sequence of basic minors where $U$ is a singleton and $k=0$ in every step.

As with many of the operations from \Cref{sec:operations}, minors can be described in terms of submodular set functions.
 
\begin{lemma}\label{lem:minors-submodular-fctns}
  Let $P$ be an \M-convex set with corresponding submodular set function $p: 2^E \to \Z$.
	The submodular functions of the deletion and contraction of $P$ by $U$ are
	\begin{align*}
				\text{\textup{(}deletion\textup{)} } &p_{\setminus U}: 2^V \to \Z & \text{\textup{(}contraction\textup{)} } &p_{/U}: 2^V \to \Z \\
			&p_{\setminus U}(A) = p(A)  \text{ for all } A \subseteq V & &p_{/U}(A) = p(A \cup U) - p(U)  \text{ for all } A \subseteq V \, ,
	\end{align*}
 The submodular function of the basic minor $P^U_k$ with respect to $U$ is
	\begin{align*}
	 	\text{\textup{(}basic minor\textup{)} } &p^U_k: 2^V \to \Z \\
		&p^U_k(A) = \min(p(A), k + p(A \cup U) - p(E \setminus A))  \text{ for all } A \subseteq V \, ,
	\end{align*}
	for any $p(E)-p(U) \leq k \leq p(V)$.
\end{lemma}
\begin{proof}
  Consider the coordinate projection $\pi_V(P)$.
	By \Cref{thm:Mnat+convex}, there exist submodular functions $f, g \colon 2^V \rightarrow \ZZ$ such that $\pi_V(P)=G(f,g^\#)$.
	First, we compute $f$ and $g$ from $P$.
	As $B(f)$ and $B(g)$ are the top and bottom layers of $G(f,g^\#)$ by \Cref{lem:top-bottom-layers}, using the correspondence from \Cref{thm:Mconv+submod+permutohedra+correspondence} we have for all $A \subseteq V$
	\begin{align*}
	f(A) &= \max_{x \in P}(x(A)) = p(A) \, , \\
	g^{\#}(A) &= \min_{x \in P}(x(A)) = \min_{x \in P} (x(E) - x(E \setminus A)) = p(E) - \max_{x \in P} x(E \setminus A) = p(E) - p(E \setminus A) \, .
	\end{align*}
%As such, we have
%	\[
%	\pi_V(P) = \{x \in \Z^V \mid p(E) - p(E \setminus A) \leq x(A) \leq p(A) \forall A \subseteq V\} \, .
%	\]
	By \Cref{lem:top-bottom-layers}, we have that $P\setminus U = B(f)\cap \ZZ^V$ and $P/U = B(g^\#)\cap \ZZ^V$, and hence have $p_{\setminus U} = f$ and $p^\#_{/U} = g^\#$.
	This proves the claim for the deletion $p_{\setminus U}$.
	To prove the claim for the contraction $p_{/U}$ we observe that for all $A \subseteq V$ we have
	\begin{align*}
	p_{/U}(A) &= p^\#_{/U}(V) - p^\#_{/U}(V \setminus A) \\
		&= p(E) - p(E\setminus V) - p(E) + p(E \setminus (V \setminus A)) \\
		&= p(A \cup U) - p(U) \, .
	\end{align*}

	Let $p(E) - p(U) \leq k \leq p(V)$ and consider $P^U_k$.
	As $P^U_k$ is the intersection of $G(f, g^\#)$ with the single layer plank $H_k = K(k,k)$, \Cref{thm:box+plank} implies that 
	\[
	p_k^U(A) = \min(f(A), k - g^\#(V \setminus A)) = \min(p(A), k + p(A \cup U) - p(E \setminus A)) \, . \qedhere
	\]
\end{proof}

Note that our definitions of deletion and contraction match those in~\cite{Fujishige:2005}, although deletion is referred to as reduction or restriction (and is different from our notion of restriction).
Our notion of a minor is a direct generalization of their definition of a set minor, as we allow basic minors outside of deletion and contraction.

\begin{remark}
	Let $P$ be the bases of a matroid $M$, then its associated submodular function is the rank function of $M$.
	\Cref{lem:minors-submodular-fctns} implies that the minors as defined above are precisely the deletion and contraction, as defined in terms of the rank function of $M$.
\end{remark}

We end by remarking that basic minors of an \M-convex set form a quotient. 

\begin{proposition}\label{prop:minors-quotients}
	Let $P \subseteq \Z^E$ be an \M-convex set, $U \subseteq E$ and $k, \ell \in \ZZ$ such that
	\[
	\min_{x \in P} x(U) \leq \ell < k \leq \max_{x \in P} x(U) \, .
	\]
	Then the minors $P^U_k, P^U_\ell$ form a quotient $P^U_k \quotient P^U_\ell$.
\end{proposition}
\begin{proof}
	By construction, $P^U_k$ and $P^U_\ell$ are layers of the $\Mnat$-convex set $\pi_V(P)$.
	Moreover, they are the top and bottom layers respectively of the $\Mnat$-convex set $\pi_V(P) \cap K(\ell, k)$.
	Thus, by \eqref{Mnatural} of \Cref{thm:quotient}, they form a quotient.
\end{proof}

\subsection{\texorpdfstring{Proof of \eqref{submod-fctns-compliant}--\eqref{exchange-property} in \Cref{thm:quotient}}{Proof of (1)-(6) in Theorem 1.2}}
We now have all the necessary tools to prove the first six equivalent definitions of a quotient listed in \Cref{thm:quotient}.
The remainder of this section is entirely dedicated to these proofs. Recall that the equivalence of \eqref{submod-fctns-compliant} and \eqref{bases+containment} follows from \Cref{thm:quotient+iff+bases+contained}.

The first equivalences largely follow from work by Fujishige and Hirai~\cite{FujishigeHirai:2022}.
We recall one of their crucial insights connecting quotients with generalized polymatroids. 

\begin{theorem}[{\cite[Theorem 2.2]{FujishigeHirai:2022}}]\label{thm:compliant+gp}
Let $p,q: 2^E \to \Z$ be submodular set functions.
Then $p,q$ are compliant if and only if $G(p,q^\#)$ is a generalized polymatroid.
\end{theorem}

While many of the following statements are already implicit in their work, we spell out the details.
We first consider the equivalence between \eqref{submod-fctns-compliant} and  \eqref{submod-poly-containment}.
From the definition of submodular polyhedra, $S(q_{/X}) \subseteq S(p_{/X})$ if and only if $p_{/X}(A) \geq q_{/X}(A)$ for all $X \subseteq E$ and $A \subseteq E \setminus X$.
Setting $Y = A \sqcup X$ and expanding out the definition of contraction, this precisely gives that their submodular functions are compliant:
\[
q(Y) - q(X) \leq p(Y) - p(X) \quad \forall X \subseteq Y \subseteq E \, .
\]

\begin{proposition}[\eqref{submod-fctns-compliant} $\iff$ \eqref{submod-poly-containment}]\label{prop:compliant+poly+containments}
Let $p,q$ be $\ZZ$-valued submodular set functions on $E$.
Then $p,q$ are compliant if and only if $S(q_{/X}) \subseteq S(p_{/X})$ for all $X \subseteq E$.
\end{proposition}

We next show the equivalence between \eqref{submod-fctns-compliant} and \eqref{Mnatural}.
Given two compliant submodular functions $p,q$, \Cref{thm:Mnat+convex,thm:compliant+gp} ensure that $G(p,q^\#)$ is a generalized polymatroid and hence its lattice points form an $\Mnat$-convex set. 
Furthermore, \Cref{lem:top-bottom-layers} implies that the top and bottom layers of this $\Mnat$-convex set are precisely the \M-convex sets determined by $p$ and $q$, giving one direction of the equivalence.

\begin{proposition}[\eqref{submod-fctns-compliant} $\implies$ \eqref{Mnatural}]\label{prop:top-bottom-layer}
Let $p,q: 2^E \to \Z$ be compliant submodular functions.
Then $P := B(p) \cap \Z^E$ and $Q := B(q) \cap \Z^E$ are the top and bottom layers of the $\Mnat$-convex set $R = G(p,q^\#) \cap \Z^E$.
\end{proposition}

For the other direction, given an $\Mnat$-convex set $R$, \Cref{thm:Mnat+convex} implies there exists submodular functions $f,g\colon 2^E \rightarrow \ZZ$ such that $R = G(f,g^\#) \cap \ZZ^E$.
Moreover, \Cref{thm:compliant+gp} implies they must also be compliant.
Applying \Cref{prop:top-bottom-layer} implies that $P = B(f) \cap \ZZ^E$ and $Q = B(g) \cap \ZZ^E$, i.e., $f,p$ and $q,g$ define the same \M-convex sets.
However, \Cref{thm:Mconv+submod+permutohedra+correspondence} implies there is a one-to-one correspondence between integral submodular set functions and \M-convex sets, hence $p = f$ and $q = g$.
In particular, $p,q$ are also compliant.

\begin{proposition}[\eqref{Mnatural} $\implies$ \eqref{submod-fctns-compliant}]
Let $p,q$ be $\ZZ$-valued submodular set functions with associated \M-convex sets $P := B(p) \cap \Z^E$ and $Q := B(q) \cap \Z^E$.
If there exists an $\Mnat$-convex set $R \in \ZZ^E$ such that $P = R^\uparrow$ and $Q = R^\downarrow$, then $p,q$ are compliant.
\end{proposition}

\begin{remark}\label{rem:unique+Mnat+convex}
We note that $R = G(p,q^\#) \cap \ZZ^E$ is the unique $\Mnat$-convex set such that $P = R^\uparrow$ and $Q = R^\downarrow$.
This follows from $\Mnat$-convex sets being in bijection with generalized polymatroids from \Cref{thm:Mnat+convex}, and that generalized polymatroids are uniquely determined by their submodular functions~\cite[Theorem 14.2.8]{Frank:2011}.
\end{remark}

A few further equivalences are easily deducible from a number of already mentioned results.
The proof of \eqref{Mnatural} $\iff$ \eqref{deletion-contraction} follows from \Cref{lem:deletion+contraction+layer} along with \Cref{thm:Mnat+convex}.
%By \cite[Theorem 2.2]{FujishigeHirai:2022}, $R$ is an $\Mnat$-convex set if and only $R = P(f,g^\#)$ is an (integral) generalized polymatroid. In \cite[Theorem 3.58]{Fujishige:2005} it is shown that equivalently, there exists a submodular set function $r: E \sqcup e \to \Z$ such that $f$ is the restriction of $r$ and $g$ is the deletion (i.e. $g^\#$ is the restriction of $r^\#$).
The proof of \eqref{Mnatural} $\implies$ \eqref{exchange-property} follows directly from the exchange axiom $\rm{(P3[\ZZ])}$ in~\cite{MurotaShioura:2018simpler} that all $\Mnat$-convex sets satisfy.

The only remaining equivalence is to show \eqref{exchange-property} implies any of the other conditions.
We show that \eqref{exchange-property} implies \eqref{submod-fctns-compliant}.

\begin{lemma}
	Let $p,q\colon 2^E \rightarrow \ZZ$ be compliant submodular functions satisfying $p(E) = q(E)$. Then $p = q$.
\end{lemma}
\begin{proof}
	For any $A \subseteq E$, the inequality	$q(A) = q(A) - q(\emptyset) \leq p(A) - p(\emptyset) = p(A)$ holds.
  As $q(E) = p(E)$, the inequality	$q(E) -q(A) \leq p(E) - p(A)$	implies $q(A) \geq p(A)$.
\end{proof}

\begin{lemma}\label{lem:compatibility-single-element}
	Let $p, q$ be submodular functions, and $P, Q$ the corresponding \M-convex sets.
	If $P$ and $Q$ satisfy the asymmetric exchange property~\eqref{exchange-property}, then for all $X \subseteq E$ and $i \in X$, the following inequality holds:
	\[
	q(X) - q(X \setminus i) \leq p(X) - p(X\setminus i)  \, .
	\]
\end{lemma}

\begin{proof}
	Assume for contradiction that $q(X) - q(X \setminus i) > p(X) - p(X\setminus i)$, and let $x \in Q$ and $y \in P$ such that $x(X) = q(X)$  and $y(X\setminus i) = p(X \setminus i)$.
	Then
	\[
	y_i = y(X) - y(X\setminus i) \leq p(X) - p(X \setminus i) <  q(X) - q(X \setminus i) \leq x(X) - x(X \setminus i) = x_i.
	\]
	Set $x^{(0)} = x, y^{(0)} = y$ and $t = x_i - y_i > 0$.
	By the asymmetric exchange property, there exists some $j = j^{(1)} \in E$ with $x_j < y_j$ such that $y+ e_i - e_j \in P$ and $x - e_i + e_j \in Q$. 
	
	If $j \in X$, then set $y^{(1)} = y^{(0)}$ and $x^{(1)} = x - e_i + e_j$.
	If $j \not\in X$, then set $y^{(1)} = y+ e_i - e_j$ and $x^{(1)} = x^{(0)}$.
	Note that $x_i^{(1)} - y_i^{(1)} = t-1$.
	Applying this process iteratively creates a sequence $(j^{(k)}, x^{(k)}, y^{(k)})$ for $k\leq t$, where if $j^{(k)} \in X$, then we set $y^{(k)} = y^{(k-1)}$ and $x^{(k)} = x^{(k-1)} - e_i + e_j$, and if $j^{(k)} \not\in X$, then we set $y^{(k)} = y^{(k-1)}+ e_i - e_j$ and $x^{(k)} = x^{(k-1)}$.

	Consider the multiset $J^{(k)} = \{j^{(1)}, \dots, j^{(k)}\}$.
	Then
	\begin{align*}
	y^{(k)} = y + | J^{(k)} \setminus X | e_i - \sum_{j \in J^{(k)} \setminus X} e_j \in P \, , \\
	x^{(k)} = x - | J^{(k)} \cap X | e_i + \sum_{j \in J^{(k)} \cap X} e_j \in  Q \, .
	\end{align*}
	Thus, if $k<t$ then 
	\[
	x_i^{(k)} - y_i^{(k)} = x_i - y_i -|J^{(k)} \setminus X| - |J^{(k)} \cap X| = t - k > 0 \, ,
	\]
	so the construction of $x^{(k+1)}, y^{(k+1)}$ is valid.
	
	Let $J = J^{(t)}$.
	Since $i \in X$ and $i \not \in J$, the definition of $x^{(t)} \in Q$ implies $x^{(t)}(X \setminus i) = x(X \setminus i) + |J \cap X|$, and hence
	\[
	q(X) - x_i  =  x(X) - x_i  = x(X \setminus i)  = x^{(t)}(X \setminus i) -  |J \cap X| \leq q(X \setminus i) - |J \cap X| \, .
	\]
	This implies that $q(X) - q(X \setminus i) \leq x_i - |J \cap X|$. 
	Similarly, the definition of $y^{(t)} \in P$ gives 
	$y^{(t)}(X) = y(X) + |J \setminus X|$, and hence
	\[
	p(X \setminus i) + y_i  = y(X \setminus i) + y_i = y(X) = y^{(t)}(X) -  |J \setminus X| \leq p(X) - |J \setminus X| \, .
	\]
	This implies $y_i  + |J \setminus X| \leq p(X) - p(X \setminus i)$. 
	Since $p(X) - p(X\setminus i)<q(X) - q(X \setminus i)$ by assumption, we get
	\[
	y_i  + |J \setminus X| < x_i - |J \cap X| \, .
	\]
	Recall that by construction, we have $y_i + t = x_i$ and $|J \setminus X| + |J \cap X| = t$ which yields the desired contradiction.
\end{proof}

\begin{proposition}[\eqref{exchange-property} $\implies$ \eqref{submod-fctns-compliant}]
	Let $p, q$ be submodular functions, and $P, Q$ the corresponding \M-convex sets.
	If $P$ and $Q$ satisfy the asymmetric exchange property~\eqref{exchange-property},
	then for all $X \subseteq Y \subseteq E$, the following inequality holds:
	\[
	q(Y) - q(X) \leq p(Y) - p(X) \, ,
	\]
	i.e., $p$ and $q$ are compliant.
\end{proposition}
\begin{proof}
	We prove this by induction on $|Y \setminus X|$.
	If $|Y \setminus X| = 0$, then $X = Y$ and the statement holds trivially.
	Let $|Y \setminus X| > 0$ and pick any $y \in Y \setminus X$.
	Let $k, l \in \Z$ such that $p(Y) = p(Y \setminus y) + k$ and $q(Y) = q(Y \setminus y) + l$. 
	Note that by the previous lemma, 
	\[
	l = q(Y) - q(Y \setminus y) \leq p(Y) - p(Y \setminus y) = k,
	\]
	and so $k-l \geq 0$.
	Then 
	\begin{align*}
		p(Y) - p(X) &= p(Y \setminus y) - p(X) + k \\
		&\geq q(Y \setminus y) - q(X) + k  \qquad	\text{(induction)}\\
		&= q(Y) - q(X) + k - l \\
		&\geq q(Y) - q(X) \, . \qedhere
	\end{align*} 
\end{proof}

\subsection{Flags of \M-convex sets}
We introduce and discuss flags of \M-convex sets and their relation with $\Mnat$-convex sets.

\begin{definition}
Let $P_i \subseteq \ZZ^E$ be an \M-convex set for $0 \leq i \leq k$.
We call $(P_0, \dots, P_k)$ a \emph{flag} of \M-convex sets if $P_{i} \quotient P_{i-1}$ for all $1 \leq i \leq k$.
A flag is \emph{consecutive} if $\rank(P_{i+1}) = \rank(P_i) + 1$.
\end{definition}
%\bsinline{Choose notation consecutive rather than complete, as complete could imply one set for every rank.}
%\bsinline{Comparison with matroid flags? Hints towards later flag sections?}

This definition generalizes the one for flag matroids. 
In \Cref{sec:functions}, we extend this definition to functions on \M-convex sets. 
We first note that quotients of \M-convex sets are transitive, and hence all pairs in a flag form a quotient.
\begin{lemma} \label{lem:M-convex-quotients-transitive}
Let $P, Q, R \subseteq \ZZ^E$ be \M-convex sets.
If $P \quotient Q$ and $Q \quotient R$, then $P \quotient R$.
\end{lemma}
\begin{proof}
This follows directly from \eqref{submod-fctns-compliant}, the compliant submodular functions $p,q,r \colon 2^E \rightarrow \ZZ$:
\[
r(Y) - r(X) \leq q(Y) - q(X) \leq p(Y) - p(X)  \quad \forall X \subseteq Y \subseteq E \, . \qedhere
\]
\end{proof}

Flags of \M-convex sets are intimately related to $\Mnat$-convex sets.
Firstly, the layers of an $\Mnat$-convex set form a consecutive flag.
Conversely, any consecutive flag can be embedded inside a `minimal' $\Mnat$-convex set.

\begin{lemma}\label{lem:flags}
The layers of an $\Mnat$-convex set form a consecutive flag of \M-convex sets.
\end{lemma}
\begin{proof}
Given an $\Mnat$-convex set $R$, we show that any two layers form a quotient.
Let $P = R \cap H_k$ and $Q = R \cap H_l$ where $l \leq k$.
These are top and bottom layers respectively of the set $R' = \SetOf{x \in R}{l \leq x(E) \leq k}$.
Moreover, \Cref{thm:box+plank} implies $R'$ is also $\Mnat$-convex.
\end{proof}

We will call consecutive flags whose union forms an $\Mnat$-convex set an \emph{$\Mnat$-convex flag}.
Not all consecutive flags are $\Mnat$-convex flags as \Cref{ex:flags} demonstrates.
However, we can embed every consecutive flag into a canonical $\Mnat$-convex flag.

\begin{lemma}\label{lem:mnat+completion}
Let $(P_0, \dots, P_k)$ be a consecutive flag of \M-convex sets, and let
\[
\widetilde{P}_i = G(p_k, p_0^\#) \cap H_{\rank(P_0) + i} \cap \ZZ^E \, .
\]
where $p_0$ and $p_k$ are the submodular functions corresponding to $P_0$ and $P_k$.
Then $(\widetilde{P}_0, \dots, \widetilde{P}_k)$ is an $\Mnat$-convex flag with $P_i \subseteq \widetilde{P}_i$ for all $0 \leq i \leq k$.
\end{lemma}

\begin{proof}
To simplify notation, let $p = p_k$ and $q = p_0$, and without loss of generality $\rank(P_0) = 0$.
As $R = G(p, q^\#)\cap \ZZ^E$ is $\Mnat$-convex, $(\widetilde{P}_0, \dots, \widetilde{P}_k)$ is an $\Mnat$-convex flag by definition.

As $\rank(P_i) = \rank(\widetilde{P}_i)$, it suffices to show $P_{i} \subseteq R$.
Let $\rho$ be the submodular function corresponding to $P_i$ and let $x \in P_i$, i.e., $x(A) \leq \rho(A)$ with equality at $A = E$.
As $P \quotient P_i$, we first note that
\[
x(A) \leq \rho(A) = \rho(A) - \rho(\emptyset) \leq p(A) - p(\emptyset) = p(A) \, .
\]
Moreover, we see from $P_i \quotient P_0$ and that
\begin{align*}
q^\#(A) &= q(E) - q(E \setminus A) \leq \rho(E) - \rho(E \setminus A) \\
&\leq x(E) - x(E \setminus A) = x(A) \, .
\end{align*}
As such, we get that $x \in R$.
\end{proof}

We next show that any non-consecutive flag has a canonical completion to a consecutive flag.
This is a generalization of Higgs lift of a matroid quotient as described in \cite[Section 1.7.6]{borovik03_coxetermatroids}. 
\begin{lemma} \label{lem:flag+completion}
  Any flag of \M-convex sets can be completed to a consecutive flag.
\end{lemma}
\begin{proof}
Let $(P_0, \dots, P_k)$ be a flag, we prove by induction on $k$.
When $k=1$, we take the $\Mnat$-convex set $R$ with $P_0 = R^{\downarrow}$ and $P_1 = R^{\uparrow}$.
By \Cref{lem:flags}, the layers of $R$ completes $(P_0, P_1)$ to a consecutive flag.
Given arbitrary $k$, we can split $(P_0, \dots, P_k)$ into $(P_0, \dots, P_{k-1})$ and $(P_{k-1}, P_k)$.
The former can be completed by the induction hypothesis, the latter by the same argument as $k=1$.
\end{proof}

Combining \Cref{lem:mnat+completion} and \Cref{lem:flag+completion} shows that every flag can be embedded into a canonical $\Mnat$-convex flag.

\begin{example}\label{ex:flags}
Consider the \M-convex sets in $\ZZ^2$
\[
P = \{(4,2),(3,3),(2,4)\} \, , \quad Q = \{(3,1),(2,2),(1,3)\} \, , \quad R = \{(2,0),(1,1),(0,2)\} \, ,
\]
with corresponding submodular functions $p,q,r$.
These form a non-consecutive flag of \M-convex sets, displayed on the left in \Cref{fig:flags}.
By \Cref{lem:flag+completion}, we can complete this to a consecutive flag by considering the $\Mnat$-convex sets $G(p,q^\#) \cap \ZZ^E$ and $G(q,r^\#) \cap \ZZ^E$ and including the middle layers.
These are the sets
\[
P' = \{(4,1),(3,2),(2,3),(1,4)\} \, , \quad Q' = \{(3,0),(2,1),(1,2),(0,3)\} \, .
\]
Now $(R,Q',Q,P',P)$ is a consecutive flag, displayed in the center of \Cref{fig:flags}.
Note that this is not an $\Mnat$-convex flag, but applying \Cref{lem:mnat+completion} allows us to embed it inside an $\Mnat$-convex flag.
This is the $\Mnat$-convex set $G(p,r^\#) \cap \ZZ^E$ displayed on the right of \Cref{fig:flags}, whose layers are $(R,Q',\widetilde{Q},P',P)$ where
\[
\widetilde{Q} = \{(4,0),(3,1),(2,2),(1,3),(0,4)\} \, .
\]
This gives a chain of embeddings
\[
(R,Q,P) \hookrightarrow (R,Q',Q,P',P) \hookrightarrow (R,Q',\widetilde{Q},P',P)
\]
into a canonical consecutive flag and a canonical $\Mnat$-convex flag respectively.
\end{example}
\begin{figure}
\includegraphics[width=\textwidth]{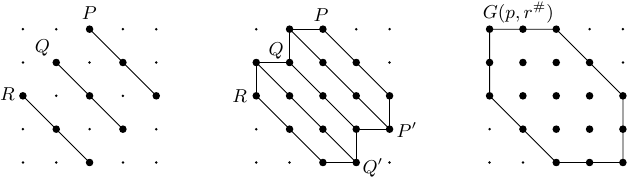}
\caption{A flag of \M-convex sets, its completion to a consecutive flag, and an $\Mnat$-convex flag from \Cref{ex:flags}.}
\label{fig:flags}
\end{figure}

\begin{remark}
There is a close connection between matroid quotients and \emph{$\Delta$-matroids} as elaborated in \cite{BoninChunNoble:2021}.
Specifically, the top and bottom layers of a $\Delta$-matroid give rise to a matroid quotient and each matroid quotient arises in this way from a $\Delta$-matroid.
In particular, \emph{saturated} $\Delta$-matroids are the same as \emph{generalized matroids}, while non-saturated $\Delta$-matroids are generalized matroids with some non-consecutive layers removed.
Note that $\Delta$-matroids and flag matroids are mutually incomparable generalizations of generalized matroids.
%The analogous concept of submodular functions for matroids are bisubmodular functions for $\Delta$-matroids.
%There may be a similar theory of strong maps for bisubmodular functions. 

Analogously to how $\Mnat$-convex sets are the lattice generalization of generalized matroids, \emph{jump systems} are the lattice generalization of $\Delta$-matroids~\cite{BouchetCunningham:1995}.
Moreover, $\Mnat$-convex sets are precisely the saturated jump systems.
As such, \eqref{Mnatural} can be rephrased as the existence of some jump system with $P$ as its top layer and $Q$ as its bottom layer, although this jump system is not unique.

There are similar exchange properties for valuated $\Delta$-matroids as for valuated matroids~\cite{Murota:1996:Delta}. 
They also describe geometric objects in tropical geometry~\cite{Rincon:2012}.
As such, they should be kept in mind when discussing quotients of valuated matroids and \M-convex functions in \Cref{sec:functions}.
\end{remark}

\section{Induction} \label{sec:induction}

In this section, we study a powerful operation for discrete convex sets known as induction.
After defining it and considering some examples, we give another characterization of quotients using induction.
We then end the section by introducing a monoid derived from induction and its defining objects, and give another characterization of quotients using Green's relations on this monoid.

\subsection{Linking sets}

We first recall a construction for matroids generalizing transversal matroids: induction by graphs~\cite{Brualdi:1971}.  %\cite{BrualdiScrimger:1968}.

%\glinline{I added this reference as it may be moreless the first one, but maybe actually \cite{Brualdi:1971} would be more illustrative. }

\begin{example} \label{ex:bipartite+graph+induction}
  Let $G=(V,U;\cE)$ be a bipartite graph on disjoint node sets $V,U$ with edges $\cE$, and let $M = (U, I)$ be a matroid on $U$ with independent sets $I$.
  The \emph{induction of $M$ through $G$} is the matroid $N = (V,I')$ with independent sets
\[
I' = \max\SetOf{A \subseteq V}{\exists B \in I, \mu \subseteq \cE  \text{ matching s.t. } \partial_V(\mu) = A, \partial_U(\mu) = B} ,
\]
where $\partial_V(\mu) \subseteq V$ is the set of nodes $\mu$ is adjacent to in $V$.

As a special case, if $M = (U, 2^U)$ is the free matroid on $U$ then $N$ is the transversal matroid associated to $G$.
\end{example}

%\todo[inline]{Ben: Add example of linking through bipartite graph: picture of 2-dim M-convex set, picture of bipartite graph on $2+3$ nodes, list the 3-dim points in the result from applying the linking system to the M-convex set. }

This construction can be generalized in a number of ways.
Firstly, one can relax that $G$ is a bipartite graph and consider more abstract `linkings' $L \subseteq 2^V \times 2^U$ satisfying certain exchange properties.
This leads to the notion of linking systems~\cite{Schrijver:1979} or bimatroids~\cite{Kung:1978}.
Secondly, we can generalize from matroids to more general discrete convex sets by considering linkings $L \subseteq \ZZ^V \times \ZZ^U$ between lattice points: this leads to poly-linking systems~\cite{Schrijver:1978}. %~\cite{KobayashiMurota:2007}.
These can equivalently be described in terms of bi-submodular functions as shown in \cite[Theorem 6.3]{Schrijver:1978}, generalizing the rank function of a bimatroid. 
Rather than defining these notions formally, we will instead consider a different point of view on them in terms of \M-convex sets.
%This is basically the same as a poly-linking systems, though we relax the condition that the zero vector must be in a linking set, as this allows for all $\Mnat$-convex sets to arise as left or right sets of a linking set.

\begin{definition}\label{def:linking+set}
Let $V, U$ be disjoint sets.
An \M-convex set $\Gamma \subseteq \ZZ^V \times \ZZ^U$ is called a \emph{linking set} from $V$ to $U$.
The \emph{left set of $\Gamma$} is the $\Mnat$-convex set $\pi_V(\Gamma) \subseteq \ZZ^V$.
The \emph{right set of $\Gamma$} is the $\Mnat$-convex set $\pi_U(\Gamma) \subseteq \ZZ^U$.
\end{definition}

\begin{example}\label{ex:left+set+transversal+matroid}
Recall the bipartite graph $G = (V, U; \cE)$ from \Cref{ex:bipartite+graph+induction}.
In this example, we identify subsets of $A \subseteq V$ with their indicator vectors in $e_A \in \Z^V$.
We define a linking set from the set of matchings in the graph as follows:
\[
\Gamma_G = \SetOf{(e_A,-e_B) \in \ZZ^V \times \ZZ^U}{\exists \mu \subseteq \cE \text{ matching s.t. } \partial_V(\mu) = A, \partial_U(\mu) = B} \, ,
\]
where $\partial_V(\mu)$ denotes the vertices in $V$ covered by $\mu$.
The left set $\pi_V(\Gamma_G)$ of $\Gamma_G$ is precisely the subsets of $V$ with a perfect matching to some subset of $U$.
This is the transversal matroid associated to $G$.
\end{example}

\begin{example}\label{ex:matrix+linking+set}
Consider a matrix $M\in \FF^{V \times U}$ over a field $\FF$ with rows and columns indexed by $V$ and $U$ respectively.
We can define a linking set $\Gamma_M \subseteq \{0,1\}^V \times \{-1,0\}^U$ by
\[
(e_I,-e_J) \in \Gamma_M \, \Leftrightarrow \, \det(M_{I,J})\neq 0 \forall I \subseteq V \, , \, J \subseteq U \, ,
\]
where $M_{I,J}$ is the square matrix given by restricting to the rows indexed by $I$ and columns indexed by $J$.
Note that we consider the empty matrix as having non-zero determinant.
It is straightforward to verify that the left set $\pi_V(\Gamma_M)$ is precisely the row matroid of $M$, while the (negative of the) right set $-\pi_U(\Gamma_M)$ is precisely the column matroid of $M$.
This perspective is what leads to Kung's notion of a bimatroid~\cite{Kung:1978}.
\end{example}

\iffalse

\begin{proposition}[\cite{KobayashiMurota:2007}]
Suppose that $V$ and $U$ are finite sets and $L\subseteq \ZZ^V \times \ZZ^U$.
Then $\Lambda = (V, U, L)$ is a poly-linking system if and only if the set $\Gamma_\Lambda \subseteq \ZZ^V \times \ZZ^U$ defined by
\[
\Gamma_\Lambda = \set{(x,-y) \mid (x,y) \in L}
\]
is an \M-convex set with $(\0_V,\0_U) \in \Gamma_\Lambda$.
\end{proposition}

%Note that the left and right sets of (poly)-linking systems are (poly)matroids.
We relax the condition that the zero vector must be in a linking set, as this allows for all $\Mnat$-convex sets to arise as left or right sets of a linking set.

\fi

Equipped with this notion, we can introduce the desired generalization of induction by a bipartite graph.

\begin{definition}
Let $\Gamma \subseteq \ZZ^V \times \ZZ^U$ be a linking set and $P \subseteq \ZZ^U$ an \M-convex (resp. $\Mnat$-convex) set.
The \emph{induction of $P$ through $\Gamma$} is the \M-convex (resp. $\Mnat$-convex) set
\begin{align*}
\ind{P}{\Gamma} &= \SetOf{x \in \ZZ^V}{\exists y \in P \text{ such that } (x,-y) \in  \Gamma} \\
& = \left.\left(\Gamma + (\0_V \times P)\right)\right|_V \subseteq \ZZ^V 
\end{align*}
where $+$ denotes Minkowski sum in the ambient space $\ZZ^V \times \ZZ^U$ and $\0_V$ denotes the zero lattice point in $\ZZ^V$.
\end{definition}

The proof that $\ind{P}{\Gamma}$ is \M-convex (resp. $\Mnat$-convex) follows from restriction and Minkowski sum preserving \M-convexity (resp. $\Mnat$-convexity) as discussed in \Cref{sec:operations}.
The submodular function of the resulting \M-convex set is described in~\cite[Theorem 6.4]{Schrijver:1978}.
This corresponding submodular function can also be derived from Fenchel duality using the interaction between projection, Minkowski sum and duality~\cite[Theorem 8.36]{Murota:2003}. 

One could also consider the action of a linking set on an \M-convex set as an analog of a linear map on a linear space instead of considering quotients as the morphisms for the category of matroids.
This viewpoint has very recently been explored in \cite{Purbhoo:2024}, not just for matroids but also defining a category of \M-convex sets.
This approach is also in the spirit of \cite{Frenk:2013} for tropical linear spaces but this goes beyond the scope of this article.

\begin{example}\label{ex:bipartite+induction+generalization}
Recall the bipartite graph $G = (V, U; \cE)$ and associated linking set $\Gamma_G$ from \cref{ex:bipartite+graph+induction,ex:left+set+transversal+matroid}. 
We encode the matroid $M$ via its independent sets $I\subseteq 2^U$ viewed as a $\Mnat$-convex set.
The induction of $I$ through $\Gamma_G$ is precisely the bipartite graph induction given in \Cref{ex:bipartite+graph+induction}, i.e.
\[
\ind{\Gamma_G}{I} = \SetOf{x \in 2^V}{\exists \text{ perfect matching between $x$ and } y \in I} \, .
\]
\end{example}

\begin{remark}
Given that the bases of a matroid determine the independent sets, one can also induce the \M-convex set $B \subseteq 2^U$ corresponding to the bases in the above example.
However, this can give a different answer to inducing the independent sets: specifically, if $\ind{\Gamma_G}{I}$ has smaller rank than~$B$, 
then $\ind{\Gamma_G}{B}$ will be empty.
Moreover, more general $\Mnat$-convex sets are not determined by their top layer, hence one requires their bottom layer also.
\end{remark}

\begin{example}\label{ex:induction+example}
To demonstrate the flexibility linking sets and induction offer, we give another method of constructing a linking set from a graph.%, different from the linking set $\Gamma_G$ introduced in \Cref{ex:left+set+transversal+matroid}.
Given a bipartite graph $G = (V,U; \cE)$, we define a linking set from all subsets of edges as follows:
\[
\widetilde{\Gamma}_G = \SetOf{\sum_{(v,u) \in \mu} (e_v, -e_u) \in \ZZ^V \times \ZZ^U}{ \mu \subseteq \cE} \, .
\]
Unlike $\Gamma_G$, this linking set has lattice points outside of $\{0,1\}^V \times \{-1,0\}^U$.
Restricting to $\widetilde{\Gamma}_G$ to $\{0,1\}^V \times \{-1,0\}^U$ recovers the linking set $\Gamma_G$.

As an explicit example, let $G \cong K_{3,2}$ where $V = \{v_1, v_2, v_3\}$, $U = \{u_1, u_2\}$ and $\cE = V \times U$.
As $G$ has six edges, the linking set $\widetilde{\Gamma}_G$ has $2^6$ lattice points, one for every subset of edges.
For example, the set of three edges adjacent to $u_1$ corresponds to the lattice point $(1,1,1,-3,0) \in \ZZ^V \times \ZZ^U$.

Consider the \M-convex set $P = \{(2,0),(1,1),(0,2)\} \subseteq \ZZ^U$ displayed in \Cref{fig:induction+example}.
The induction of $P$ through $\widetilde{\Gamma}_G$ is the $M$-convex set 
\[
\ind{\widetilde{\Gamma}_G}{P} = \{(2,0,0), (1,1,0),(0,2,0),(1,0,1),(0,1,1),(0,0,2)\} \subseteq \ZZ^V \, .
\]
\end{example}

\begin{figure}
\includegraphics[width=0.9\textwidth]{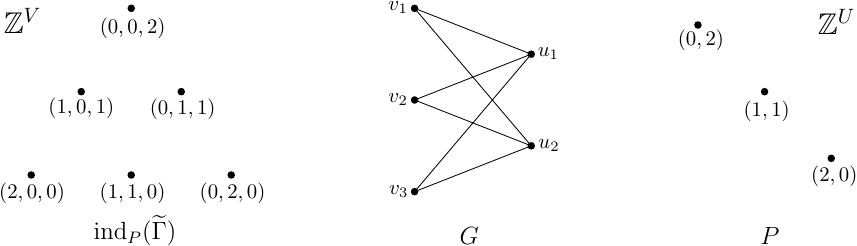}
\caption{An $M$-convex set $P$ and its induction through the linking set $\widetilde{\Gamma}_G$, introduced in Example~\ref{ex:induction+example}.}
\label{fig:induction+example}
\end{figure}

\begin{example}\label{ex:induction+operations}
Many operations on $M/M^\natural$-convex sets can be modelled by induction.
Given some $v \in \ZZ^V$, we define the (unbounded) linking set
\[
\Gamma_v = \{(x + v,-x) \in \ZZ^V \times \ZZ^V \mid x \in \ZZ^V\} \, .
\]
Then given any \M-convex $P \subseteq \ZZ^V$, it is straightforward to verify that its translation $P+v$ by $v$ is precisely $\ind{P}{\Gamma_v}$, the induction of $P$ through $\Gamma_v$.

For a slightly more involved example, consider the (unbounded) linking set
\[
\Gamma^{tr} = \{(x,-y) \in \ZZ^V \times \ZZ^V \mid y = x + e_i \text{ for some } i \in V\} \, .
\]
Then given any \M-convex $P \subseteq \ZZ^V$, its truncation $P^{tr}$ is precisely $\ind{P}{\Gamma^{tr}}$, the induction of $P$ through $\Gamma^{tr}$.
For $k$-th truncations, we can iteratively induce through $\Gamma^{tr}$, or consider the linking set obtained as the `product' of $\Gamma^{tr}$ with itself $k$ times.
This linking set product will be formally defined in \Cref{sec:monoid}.
We note that verifying $\Gamma^{tr}$ and $\Gamma_v$ are \M-convex, and therefore genuine linking sets, is a relatively straightforward check.
\end{example}

  We are not aware of a way to realize projections and minors of \M-convex sets via induction in an analogous way to \Cref{ex:induction+operations}.
  However, they are still intimately connected to induction in a different way: by viewing the \M-convex set as a linking set and inducing other sets through it.

\begin{example}
  Let $\Gamma \in \Z^V \times \Z^U$ be \M-convex.
  To obtain the projection $\pi_{V}(\Gamma)$, we simply induce the unbounded `free' $\Mnat$-convex set $\ZZ^U$ through $\Gamma$:
  \[
  \ind{\ZZ^U}{\Gamma} = \{x \in \Z^V \mid \exists y \in \Z^U \text{ such that } (x,-y) \in \Gamma\} = \pi_V(\Gamma) \, .
  \]
  Now suppose we wish to obtain $\Gamma^U_k$, the basic $k$-th minor of $\Gamma$ with respect to $U$.
  Consider the hyperplane $H_{k-\rank(\Gamma)} \subseteq \ZZ^U$, inducing this through $\Gamma$ gives us the basic $k$-th minor of $\Gamma$ with respect to $U$:
	\begin{align*}
		\ind{H_{k-\rank(\Gamma)}}{\Gamma} &= \{x \in \Z^V \mid \exists y \in H_{k-\rank(\Gamma)} \text{ such that } (x,-y) \in \Gamma\} \\
		&= \{x \in \Z^V \mid \exists y \in \ZZ^U \text{ such that } (x,-y) \in \Gamma \, , \, -y(U) = \rank(\Gamma)-k\}\\
		&= \{x \in \Z^V \mid \exists y \in \ZZ^U \text{ such that } (x,-y) \in \Gamma \, , \, x(V) = k\} = \Gamma^U_k \, .
	\end{align*}
\end{example}

\begin{remark}
Up until this point, we have only worked with bounded \M-convex sets, and so the reader may be nervous about allowing unbounded sets in the previous examples.
However, induction is just a combination of Minkowski sum and restriction, both of which are well defined for unbounded sets.
Moreover, the projection of an unbounded \M-convex set is a well-defined, but possibly unbounded, $\Mnat$-convex set.
This will continue to be relevant in \Cref{sec:monoid} where we will need unbounded sets to form a well-defined monoid structure on linking sets.
\end{remark}

\subsection{Quotients via induction}

In this section, we consider the relationship between induction and quotients.
We first show that quotients are preserved under induction.

\begin{lemma}\label{lem:induction+quotient}
	Let $P, Q \subseteq \ZZ^U$ be \M-convex sets, and $\Gamma \subseteq \ZZ^V \times \ZZ^U$ a linking set such that $\ind{P}{\Gamma}$ and $\ind{Q}{\Gamma}$ are non-empty.
	If $P \quotient Q$ then $\ind{P}{\Gamma} \quotient \ind{Q}{\Gamma}$.
%	Moreover, if $P \cup Q \subseteq \pi_U(\Gamma)$, then $P \quotient Q$ if and only if $\ind{P}{\Gamma} \quotient \ind{Q}{\Gamma}$
\end{lemma}

\begin{proof}
	First note that the condition $\ind{P}{\Gamma} \neq \emptyset$ is equivalent to $(-\pi_U(\Gamma))\cap P \neq \emptyset$.
	If $P \quotient Q$ then by \Cref{thm:quotient}~\eqref{Mnatural}, there exists an $\Mnat$-convex set $R \in \Z^U$ such that $P = R^\uparrow$ and $Q = R^\downarrow$.
	The induction of an $\Mnat$-convex set is $\Mnat$-convex, so we compute the rank of the top layer as
	\begin{align*}
		\rank(\ind{R}{\Gamma}^{\uparrow})
		&= \max\SetOf{x(V)}{(x,-y) \in \Gamma \, , \, y \in R}  \\
		&= \max\SetOf{x(V) - y(U) + y(U)}{(x,-y) \in \Gamma \, , \, y \in R} \\
		&= \rank(\Gamma) + \max\SetOf{y(U)}{(x,-y) \in \Gamma \, , \, y \in R} \\
		&= \rank(\Gamma) + \rank(P) \, ,
	\end{align*}
	where the last equality holds since $P$ is the top layer of $R$ and $(-\pi_U(\Gamma))\cap P \neq \emptyset$. Thus,
	\begin{align*}
		\ind{R}{\Gamma}^{\uparrow} 
		&= \{x \in \Z^V \mid (x,-y) \in \Gamma \, , \, y \in R \, , \, x(U) = \rank(\Gamma) + \rank(P)\} \\
		&= \{x \in \Z^V \mid (x,-y) \in \Gamma \, , \, y \in R \, , \, y(V) = \rank(P)\} 
		=	\ind{P}{\Gamma} \, .
	\end{align*}
A similar computation shows for the bottom layer that $\rank(\ind{R}{\Gamma}^{\downarrow}) =  \rank(\Gamma) + \rank(Q)$ and $\ind{R}{\Gamma}^{\downarrow} = \ind{Q}{\Gamma}$. Therefore, the inductions of $P$ and $Q$ are the top and bottom layers of an $\Mnat$-convex set, thus forming a quotient.
\end{proof}

\begin{remark}
While \Cref{lem:induction+quotient} only requires $(-\pi_U(\Gamma))\cap P \neq \emptyset$, if $P \nsubseteq -\pi_U(\Gamma)$ then there will be points of $P$ that are `forgotten' in the induction process.
This will be important when considering lifts of \M-convex sets in \Cref{sec:lifts}.
\end{remark}

The other relationship between induction and quotients is~\eqref{induction}: that $P \quotient Q$ if and only if there exists a linking set $\Gamma \subseteq \ZZ^E \times \ZZ^\tE$ and \M-convex set $W \in \ZZ^\tE$ such that $P= \pi_E(\Gamma)^\uparrow$ is the top layer of the left set of $\Gamma$, and $Q= \ind{W}{\Gamma}$ is the induction of $W$ through $\Gamma$.
The motivation for this construction is inspired by matrix multiplication.
Assume you have a realizable matroid given as the row matroid of a matrix whose rows are labelled by $E$ and columns are labelled by $\tE$.
The discrete analog of this matrix is the linking set $\Gamma \subset \Z^E \times \Z^{\tE}$, and the bases of the row matroid form the points of an \M-convex set, namely the left set of $\Gamma$.
Matrix multiplication from the right amounts to taking combinations of the columns of the matrix.
This is captured by the \M-convex set $W$, analogous to the row matroid of the matrix one multiplies by on the right.
The induction $\ind{W}{\Gamma}$ is analogous to the row matroid of the product of the matrices.

\begin{proposition}[\eqref{Mnatural} $\implies$ \eqref{induction}] \label{prop:Mnat+implies+induction}
	Let $P$ be the top layer of an $\Mnat$-convex set $R$ and $Q$ be the bottom layer. 
	Then there exist \M-convex sets $\Gamma \subset \Z^E \times \Z^{\tE}$ and $W \subset \Z^{\tE}$ such that $P=\pi_E(\Gamma)^\uparrow$ is the top layer of the coordinate projection and $Q=\ind{W}{\Gamma}$ is the induction of $W$ through $\Gamma$.
\end{proposition}

\begin{proof}
Let $R$ be the $\Mnat$-convex set with $R^\uparrow = P$ and $R^\downarrow = Q$.
We can lift $R$ to an \M-convex set $\Gamma = \widetilde{R} \subseteq \ZZ^{E'}$ where $E' = E \cup \{e\}$ such that $x(E') = \rank(P)$.
In particular, $x(e) = 0$ for all $x \in P$ and $y(e) = \rank(P) - \rank(Q)$ for all $y \in Q$.
Let $W = \rank(P) - \rank(Q) \in \ZZ^e$ be a point, a trivially \M-convex set.
Then $(\widetilde{R} + (\0_E \times W))|_E = Q$.
\end{proof}

\begin{proposition}[\eqref{induction} $\implies$ \eqref{exchange-property}] 
	Let $\Gamma \subset \Z^E \times \Z^{\tE}$ and $W \subset  \Z^{\tE}$ be \M-convex sets, $P = \pi_E(\Gamma)^{\uparrow}$ the top layer of the coordinate projection, and $Q=\ind{W}{\Gamma}$ the induction of $W$ through $\Gamma$.
	Then for all $x \in Q, y \in P$ and $i \in \supp^+(x-y)$, there exists some $j \in \supp^-(x-y)$ such that $x - e_i + e_j \in Q$ and $y + e_i - e_j \in P$.
\end{proposition}

\begin{proof}
	Let $x \in Q, y \in P$ and $i \in \supp^+(x-y)$.
	By definition of $Q$, there exists some $w \in W$ such that $(x,-w) \in \Gamma$.
	Moreover, by definition of $P$ there exists some $z \in \Z^\tE$ such that $(y,z) \in \Gamma$, and $y(E) = \rank(P)$ is maximal.
	
	As $\Gamma$ is an \M-convex set, there exists some $j \in \supp^-((x,-w) - (y,z))$ such that $(x,-w) - e_i + e_j$ and $(y,z) + e_i - e_j \in \Gamma$.
	If $j \in E$ then $j \in \supp^-(x-y)$, hence $x - e_i + e_j \in Q$ and  $y + e_i - e_j \in P$.
	To complete the proof, suppose that $j \in \tE$.
	Then $(y + e_i, z - e_j) \in \Gamma$ and $y + e_i \in P$.
	However, $(y + e_i)(E) = y(E) + 1 > \rank(P)$, which contradicts that $P$ is the top layer.
\end{proof}

\begin{example}
Recall that a transversal matroid $M \subseteq \ZZ^V$ is an \M-convex set that can be realized as the left set of some $\Gamma_G \subseteq \ZZ^V \times \ZZ^U$ where $G = (V,U; \cE)$ a bipartite graph, as in \Cref{ex:left+set+transversal+matroid}.
One may ask which matroids $N \subseteq \ZZ^V$ can arise as a quotient of $M$?
With \eqref{induction}, we can characterize quotients of $M$ as precisely the matroids that arise as the induction $\ind{W}{\Gamma_G}$ of a matroid $W \in \ZZ^U$ through the bipartite graph $G$, as described in \Cref{ex:bipartite+induction+generalization}, where $G$ is a graph that realizes $M$ as a transversal matroid.
\end{example}

\subsection{The monoid of linking sets}\label{sec:monoid}

As shown in \Cref{ex:matrix+linking+set}, linking sets from $V$ to $U$ are combinatorial abstractions of matrices with rows and columns labelled by $V$ and $U$ respectively.
For further reading on this point of view, we refer to~\cite{Murota:2000}. 
The square matrices form a monoid with multiplication given by matrix multiplication and the identity matrix as the identity element.
We can define a similar monoid for linking sets, but first must define a product operation.

\begin{definition}
The \emph{product} of two (possibly unbounded) linking sets $\Gamma \subseteq \ZZ^W \times \ZZ^V$ and $\Delta \subseteq \ZZ^V \times \ZZ^U$ is
\begin{align*}
\Gamma * \Delta &= \SetOf{(x,-z) \in \ZZ^W \times \ZZ^U}{\exists y \in \ZZ^V \text{ such that } (x,-y) \in \Gamma \, , \, (y, -z) \in \Delta} \\
&= \left((\Gamma \times\{\0_U\}) + (\{\0_W\} \times \Delta)\right)|_{W \cup U} \, ,
\end{align*}
where $+$ denotes Minkowski sum in the ambient space $\ZZ^W \times \ZZ^V \times \ZZ^U$.
\end{definition}

Its straightforward to verify that the product operation on linking sets is associative, i.e.
\[
\Gamma*(\Delta*\Sigma) = (\Gamma*\Delta)*\Sigma \quad , \quad \forall \Gamma \subseteq \ZZ^W \times \ZZ^V \, , \, \Delta \subseteq \ZZ^V \times \ZZ^U \, , \, \Sigma \subseteq \ZZ^U \times \ZZ^T \, .
\]

\begin{example}
Recall the construction from \Cref{ex:matrix+linking+set} of obtaining a linking set $\Gamma_M \subseteq \ZZ^V \times \ZZ^U$ from a matrix $M \in \FF^{V \times U}$ over some field $\FF$.
Given $N \in \FF^{W \times V}$ with linking set $\Gamma_N \subseteq \ZZ^W \times \ZZ^V$, the Cauchy-Binet Theorem implies that
\[
\Gamma_{N\cdot M} \subseteq \Gamma_N * \Gamma_M \, ,
\]
with equality if the matrices $M,N$ are sufficiently generic.
This is precisely the motivation for bimatroid multiplication, as defined in~\cite[Section 6]{Kung:1978}
\end{example}

\begin{example}
Recall the construction from \Cref{ex:left+set+transversal+matroid} of obtaining a linking set $\Gamma_G \subseteq \ZZ^V \times \ZZ^U$ from a bipartite graph $G = (V,U;\cE)$.
Given $H = (W,V; \cE')$, we can define a bipartite graph `product' by
\[
H\cdot G = (W, U; \cE' \cdot \cE) \quad , \quad \cE' \cdot \cE = \SetOf{(w,u) \in W \times U}{\exists v \in V \text{ s.t } (w,v) \in \cE' \, , \, (v,u) \in \cE} \, .
\]
It is straightforward to see that $\Gamma_{H\cdot G} \subseteq \Gamma_H * \Gamma_G$.
\end{example}

\begin{remark}\label{rem:induction+product}
Induction of $\Mnat$-convex sets is closely related to products of linking sets in the following way.
Given some $\Mnat$-convex set $P \subseteq \ZZ^U$, \Cref{prop:Mnat+lifts+to+M} ensures the existence of some linking set $\Delta_P \subseteq \ZZ^U \times \ZZ^W$ such that $\pi_U(\Delta_P) = P$, e.g.
\[
\Delta_P = \SetOf{(y,-y(U)) \in \ZZ^U \times \ZZ}{y \in P} \, .
\]
The induction $\ind{P}{\Gamma}$ of $P$ through $\Gamma \subseteq \ZZ^V \times \ZZ^U$ is equal to the left set $\pi_V(\Gamma * \Delta_P)$ of $\Gamma * \Delta_P$.
Note that we can substitute $\Delta_P$ for any linking set with left set $P$, as the induction projects the right set away after product.
\end{remark}

We are now ready to define a monoid structure on linking sets from a set to itself.

\begin{definition}
Let $V$ be a finite set.
The \emph{monoid of linking sets} $(\cM_V, *, I_V)$ on $V$ is the set of (possibly unbounded) \M-convex sets
\[
\cM_V = \SetOf{\Gamma \subseteq \ZZ^V \times \ZZ^V}{\Gamma \text{ is \M-convex}}
\]
with product $*$ as its associative operation and identity element
\[
I_V = \SetOf{(x, -x) \in \ZZ^V \times \ZZ^V}{x \in \ZZ^V} \, .
\]
\end{definition}

Note that to ensure $\cM_V$ is closed, we consider the empty set $\emptyset$ as an \M-convex set.
This is the unique two-sided absorbing element with $\Gamma * \emptyset = \emptyset * \Gamma = \emptyset$ for all $\Gamma \in \cM_V$.

It is immediate to see that $\cM_V$ does not have inverses: for example, for any $\Gamma$ with finite support we cannot hope to find $\Gamma^{-1}$ such that $\Gamma * \Gamma^{-1}$ has infinite support.
One might hope that every element has at least one \emph{pseudoinverse}, i.e., for all $\Gamma \in \cM_V$, there exists some $\Gamma^{-1} \in \cM_V$ such that
\[
\Gamma * \Gamma^{-1} * \Gamma = \Gamma \, .
\]
Monoids where every element has a pseudoinverse are called \emph{regular monoids}.

\begin{proposition}
For $|V| > 2$, the monoid $\cM_V$ is not a regular monoid.
\end{proposition}

\begin{proof}
For notational convenience, we define the partial operation $*$ on elements of $\ZZ^V \times \ZZ^V$ defined as
\[
(x, -y) * (z,-w) = \begin{cases}
(x,-w) & y = z \\ \text{undefined} & y \neq z \, .
\end{cases}
\]
It follows that the product of linking sets is the union of products of its elements, i.e., $\Gamma * \Delta = \left\{\gamma * \delta \mid \gamma \in \Gamma \, , \, \delta \in \Delta\right\}$.

Let $V \supseteq [3]$ and consider the linking set
\[
\Gamma = \SetOf{(e_i,-e_j) \in \ZZ^V \times \ZZ^V}{i, j \in [3], i \neq j} \cup \{\mathbf{0}\} \, .
\]
Suppose there exists some pseudoinverse $\Gamma^{-1}$ of $\Gamma$, and fix some $(e_i, -e_j) \in \Gamma$.
As $\Gamma^{-1}$ is a pseudoinverse, there must exist some $(e_k, -e_\ell) \in \Gamma^{-1}$ with $k \neq i$ and $\ell \neq j$ such that
\[
(e_i, -e_k) * (e_k, -e_\ell) * (e_\ell, -e_j) = (e_i, -e_j) \, .
\]
Pick any $a \in [3] \setminus \{k, \ell\}$, then
\[
(e_{a}, -e_k) * (e_k, -e_\ell) * (e_\ell, -e_{a}) = (e_a, -e_a) \in \Gamma \, ,
\]
a contradiction.
\end{proof}

To describe the basic structure of a semigroup, one first asks what are its Green's relations.
This will be particularly useful for us, as it transpires that the Green's relations of $\cM_V$ encode quotients of \M-convex sets on $\ZZ^V$.

\begin{definition}[Green's $\cR$ and $\cL$ relations]
We define the preorders $\preceq_{\cR}$ and $\preceq_{\cL}$ on $\cM_V$ given by 
\begin{align*}
\Gamma \preceq_{\cR} \Delta \quad &\Leftrightarrow \quad \exists X \in \cM_V \text{ such that } \Gamma*X = \Delta \, , \\
\Gamma \preceq_{\cL} \Delta \quad &\Leftrightarrow \quad \exists X \in \cM_V \text{ such that } X*\Gamma = \Delta \, .
\end{align*}
Green's $\cR$-relation $\sim_\cR$ is the equivalent relation $\preceq_{\cR} \cap \succeq_{\cR}$, i.e.
\[
\Gamma \sim_\cR \Delta \quad \Leftrightarrow \quad \exists X,Y \in \cM_V \text{ such that } \Gamma*X = \Delta \; , \; \Delta*Y = \Gamma \, .
\]
Green's $\cL$-relation $\sim_\cL$ is the equivalent relation $\preceq_{\cL} \cap \succeq_{\cL}$ defined analogously.
\end{definition}

The following lemma shows that we can characterize the preorders $\preceq_{\cR}$ and $\preceq_{\cL}$ via the quotient structures on the left and right sets of linking sets.
We denote the left (resp. right) set of $\Gamma$ as $\pi_L(\Gamma) \subseteq \ZZ^V$ (resp. $\pi_R(\Gamma) \subseteq \ZZ^V$); as both the left and right sets are on the same ground set, we deviate slightly from the notation of \Cref{def:linking+set}.

\begin{proposition}[\eqref{induction} $\Longleftrightarrow$ \eqref{r-order}]\label{prop:quotient+iff+R+related}
If $\Gamma, \Delta \in \cM_V$ are linking sets such that $\Gamma \preceq_{\cR} \Delta$ then $\pi_L(\Gamma)^\uparrow \quotient \pi_L(\Delta)^\uparrow$.
Conversely, for any \M-convex sets $P, Q \subseteq \ZZ^V$ such that $P \quotient Q$, there exists $\Gamma, \Delta \in \cM_V$ such that $P = \pi_L(\Gamma)^\uparrow, Q = \pi_L(\Delta)^\uparrow$ and $\Gamma \preceq_{\cR} \Delta$.
\end{proposition}
\begin{proof}
Suppose $\Gamma \preceq_{\cR} \Delta$, then $\Gamma * X = \Delta$ for some $X \in \cM_V$.
Moreover, note by \Cref{rem:induction+product} we have
\[
\pi_L(\Delta) = \pi_{L}(\Gamma * X) = \ind{\pi_{L}(X)}{\Gamma} \, .
\]
Let $\ell = \rank(\pi_{L}(\Delta)^\uparrow)$ and consider the \M-convex set $\pi_{L}(X) \cap H_{\ell - \rank(\Gamma)}$ obtained as a layer of $\pi_{L}(X)$.
Then
\begin{align*}
\ind{\pi_{L}(X) \cap H_{\ell - \rank(\Gamma)}}{\Gamma} &= \SetOf{x \in \ZZ^V}{(x,-y) \in \Gamma \, , \, y \in \pi_{L}(X) \, , \, y(V) = \ell - \rank(\Gamma)} \\
&= \SetOf{x \in \ZZ^V}{(x,-y) \in \Gamma \, , \, y \in \pi_{L}(X) \, , \, x(V) = \ell} \\
&= \ind{\pi_{L}(X)}{\Gamma}^\uparrow = \pi_L(\Delta)^\uparrow \, .
\end{align*}
Using the induction characterization \eqref{induction} of quotients, we get that $\pi_L(\Gamma)^\uparrow \quotient \pi_L(\Delta)^\uparrow$.

Conversely, if $P \quotient Q$ then \eqref{induction} implies there exists linking set $\Gamma \in \ZZ^V \times \ZZ^U$ and \M-convex set $W \in \ZZ^U$ such that $P = \pi_{V}(\Gamma)^\uparrow$ and $Q = \ind{W}{\Gamma}^\uparrow$, where $\ind{W}{\Gamma}$ is just a single layer.
Moreover, the proof of \Cref{prop:Mnat+implies+induction} implies that we can pick $|U| =1$, and so we can always increase the size of $U$ such that $\Gamma, \Delta_W \in \cM_V$, where $\Delta_W$ defined as in \Cref{rem:induction+product}.
This gives $P = \pi_{L}(\Gamma)^\uparrow$ and $Q = \ind{W}{\Gamma} = \pi_L(\Gamma*\Delta_W)^\uparrow$.
\end{proof}

We get an entirely analogous lemma by replacing $\preceq_{\cR}$ with $\preceq_{\cL}$, and left sets $L$ with right sets $R$.

\section{Lifts of \M-convex sets} \label{sec:lifts}

In this section, we will consider ways to `lift' general \M-convex sets to special families, such as matroids and $k$-polymatroids.
These lifts will be constructed via the induction machinery introduced in \Cref{sec:induction}.
We will then show that quotients are preserved under lifts, and hence we can relate general \M-convex quotients to matroid quotients.

\subsection{Box lifts}\label{sec:box+lifts}

There are special classes of \M-convex sets that are more widely studied, namely matroids and $k$-polymatroids.
In this section, we show one can lift an \M-convex set to one of these families by lifting to a higher dimensional space but with tighter bounds on the size of coordinate entries.

We first recall the definition of a $k$-polymatroid.
Let $k \in \ZZ_{\geq 0}$ be a non-negative integer.
An \M-convex set $P \subseteq \ZZ^U$ is a \emph{k-polymatroid} if it is contained in the box $[0,k]^U$, i.e., all lattice points are nonnegative and bounded above by $k$.
In particular, a matroid given by its bases is a 1-polymatroid.
Its associated $\Mnat$-convex set is 
\[
P^\natural = \SetOf{x \in \ZZ_{\geq 0}^U}{x \leq y \text{ for some } y \in P} \subseteq [0,k]^U \, .
\]
If $P$ were the bases of a matroid, then $P^\natural$ is precisely the independent sets.
Analogously to matroids, $k$-polymatroids are in one-to-one correspondence with submodular set functions $p\colon 2^U \rightarrow \ZZ$ satisfying
\[
p(A) \leq p(A \cup i) \leq p(A) + k \quad , \quad \forall A \subseteq U \, , \, i \in U \setminus A \, .
\]

Let $\phi \colon {V} \rightarrow U$ be a surjection.
This induces a projection $\pi_\phi$ from $\ZZ^V$ to $\ZZ^U$ defined by
\begin{equation} \label{eq:projection-from-partition}
\pi_\phi \colon \ZZ^V \rightarrow \ZZ^U \quad , \quad \pi_\phi(y) = \left(\sum_{j \in \phi^{-1}(i)} y_j\right)_{i \in U} \, .
\end{equation}
\begin{definition}
Given an \M-convex (resp. $\Mnat$-convex) set $P \subseteq \ZZ^U$, we say an \M-convex (resp. $\Mnat$-convex) set $Q \subseteq \ZZ^V$ is a \emph{lift of $P$} if there exists a surjection $\phi \colon {V} \rightarrow U$ and vector $v \in \ZZ^U$ such that $P = \pi_\phi(Q) + v$.
We say $Q$ is a \emph{$k$-polymatroid lift} of $P$ if $Q \subseteq [0,k]^V$.
When $k=1$, we call $Q$ a \emph{matroid lift} of $P$.
\end{definition}

Note that this lift construction is an inverse to the combination of `aggregation' and `translation', see \cite[Chapter 6.4]{Murota:2003} and \cite[\S II.3.1(d)]{Fujishige:2005} for details on aggregation. 

This section will be dedicated to constructing $k$-polymatroid lifts via the induction machinery introduced and studied in \Cref{sec:induction}.
To do this, we first show a more general construction that lifts boxes in $\ZZ^U$ to boxes in $\ZZ^V$ via induction.

\begin{construction}[Box lifts]
Let $P$ be an \M-convex set, we consider a box containing $P$.
Specifically, we let
\[
\underline{\omega}, \overline{\omega} \in \ZZ^U \quad \text{such that} \quad \underline{\omega}_i \leq \min\SetOf{x_i}{x \in P} \quad , \quad \overline{\omega}_i \geq \max\SetOf{x_i}{x \in P} \, .
\]
We write $\Omega_i = [ \underline{\omega}_i, \overline{\omega}_i] \subset \R$ and define $\Omega := \prod_i \Omega_i \subseteq \RR^U$ to be the lattice hypercube with vertices $\underline{\omega}$ and $\overline{\omega}$.
Clearly $P \subseteq \Omega$, and if the bounds on $\underline{\omega}$ and $\overline{\omega}$ are tight then this is the smallest such box containing $P$.
%Moreover, if $b_i = \overline{\omega}_i - \underline{\omega}_i$ then we can translate $P$ by $-\underline{\omega}$ to bound it by the nonnegative box $[0,b]$.

Let $\phi\colon V \rightarrow U$ be a surjection and $v \in \ZZ^U$; we show one can construct a box $\Psi \subseteq \ZZ^V$ such that $\Omega = \pi_\phi(\Psi) + v$.
Let $\underline{\psi}, \overline{\psi} \in \ZZ^V$ such that $\underline{\psi}_j \leq \overline{\psi}_j$ for all $j \in V$, and
\[
\Omega_i = \left(\sum_{j \in \phi^{-1}(i)} \Psi_j\right) + v_i \quad , \quad \Psi_j = [\underline{\psi}_j, \overline{\psi}_j] \subset \R \, .
\]
i.e., $\Omega_i$ is a line segment that can be written as the Minkowski sum of line segments $\Psi_j$, when viewed as subsets of $\RR$.
Such a decomposition is always possible by subdividing $\Omega_i$ into $|\phi^{-1}(i)|$ parts (possibly with zero width).
As a result, the lattice hypercube $\Psi:= \prod_j \Psi_j \subseteq \ZZ^V$ is a lift of $\Omega$.

We now define the linking set $\Gamma_{\phi,v}(\Psi) \subseteq \ZZ^{V} \times \ZZ^{U}$ by
\[
\Gamma_{\phi,v}(\Psi) = \SetOf{(y,-x) \in \ZZ^{V} \times \ZZ^{U}}{y \in \Psi \, , \, x = \pi_{\phi}(y) + v \in \Omega} \, .
\]
Note that by construction, the right and left sets are precisely $\Omega$ and $\Psi$ respectively.
Given our set $P \subseteq \Omega \subseteq \ZZ^U$, we can induce it through $\Gamma_{\phi,v}(\Psi)$ to find a lift $Q = \ind{P}{\Gamma_{\phi,v}(\Psi)} \subseteq \ZZ^V$ with the additional property that it is contained in the box $\Psi$.
We denote such a lift as
\[
\lift{\Psi}{\phi}{v}{P} = \ind{P}{\Gamma_{\phi,v}(\Psi)} \subseteq \Psi \, .
\]
\end{construction}

\begin{remark}
The proof that $\Gamma_{\phi,v}(\Psi)$ is \M-convex is as follows.
It is straightforward to verify that $\SetOf{(y,-x) \in \ZZ^{V} \times \ZZ^{U}}{x = \pi_{\phi}(y) + v}$ is an unbounded \M-convex set.
Restricting this to the box $\Psi \times -\Omega$ remains \M-convex by \Cref{thm:box+plank}.
Note that we cannot do more general restrictions, for example to arbitrary $\Mnat$-convex sets, without losing \M-convexity.
\end{remark}

\begin{construction}[Matroid lifts]\label{construction:matroid-lifts}
As a special case of box lifts, we show how one can lift any finite \M-convex set to a matroid.
Let $P \subseteq \ZZ^U$ be an \M-convex set contained in the box $\Omega \subseteq \ZZ^U$ as above.
Define $b \in \ZZ^U$ by $b_i = \overline{\omega}_i - \underline{\omega}_i$.
Note that we can translate $P$ by $-\underline{\omega}$ to bound it by the box $\prod_i [0,b_i]$.
In this case, $P -\underline{\omega}$ is already a $b^*$-polymatroid where $b^*= \max(b_i)$.

We define the ground set $V = \bigsqcup_{i \in U} V_i$ where $V_i = [b_i]$, and the associated surjection $\phi\colon V \rightarrow U$ that maps $j \in V_i$ to $i$.
Observe that
\[
\Omega_i = \underbrace{[0,1] + \cdots + [0,1]}_{b_i} + \underline{\omega}_i \, ,
\]
and so the box $\square = [0,1]^V$ maps onto $\Omega$ under $\pi_\phi$ and translation by $\underline{\omega}$.

We consider the linking set $\Gamma_{\phi, \underline{\omega}}(\square)$ from $\Omega$ to the box $\square$, i.e.
\[
\Gamma_{\phi, \underline{\omega}}(\square) = \BiggSetOf{(y,-x) \in \ZZ^{V} \times \ZZ^{U}}{x \in \Omega \, , \, y \in [0,1]^V \, , \, x_i = (\sum_{j \in V_i} y_j) + \underline{\omega}_i} \, .
\]
Then the lift of $P$ through $\Gamma_{\phi, \underline{\omega}}(\square)$ is
\[
\lift{\square}{\phi}{\underline{\omega}}{P} = \ind{P}{\Gamma_{\phi, \underline{\omega}}(\square)} = \SetOf{y \in \{0,1\}^V}{\pi_{\phi}(y) + \underline{\omega} \in P} \, .
\]
As $\lift{\square}{\phi}{\underline{\omega}}{P}$ is \M-convex and contained in $[0,1]^V$, it forms the bases of a matroid on $V$.
Moreover, by construction it is necessarily a matroid lift of $P$.
%Moreover, we can induce $P^\natural$ through $\Gamma_\phi$ to get $Q^\natural = \ind{P^\natural}{\Gamma_\phi}$, where $P = (P^\natural)^\uparrow$ and $Q = (Q^\natural)^\uparrow$, and the points in $Q^\natural$ are the independent sets of $Q$.
\end{construction}

\begin{remark}
Given a polymatroid $P \subseteq \ZZ_{\geq 0}^U$ with submodular function $p\colon 2^U \rightarrow \ZZ_{\geq 0}$, we can define a matroid lift that is minimal in the following sense.
The smallest box containing $P^\natural$ is precisely $\prod_{i \in U} [0,p(i)]$, hence we let $V = \bigsqcup_{i \in U} V_i$ with $|V_i| = p(i)$, and define the surjection $\phi \colon V \rightarrow U$ by $\phi(j) = i$ for all $j \in V_i$.
Then inducing $P$ through the linking set
\[
\Gamma_{\phi}(\square) = \BiggSetOf{(y,-x) \in \ZZ^{V} \times \ZZ^{U}}{x \in \prod_{i \in U} [0,p(i)] \, , \, y \in [0,1]^V \, , \, \pi_{\phi}(y) = x} \, .
\]
gives a matroid $\lift{\square}{\phi}{0}{P} = \ind{P}{\Gamma_{\phi}(\square)}$.
This is precisely the minimal (multisymmetric) matroid lift in the sense of~\cite{CrowleyEtAl:2022}, the natural matroid in the sense of~\cite{BoninChunFife:2023}, and it has appeared a number of times in the literature previously, for example~\cite{Helgason:1974,Lovasz:1977}.
In the former, this construction is used to define the Bergman fan of a polymatroid, and so this lift needs to be canonical.
However, we will need the freedom to lift to non-minimal matroids, as we will need to compare matroid lifts of polymatroids and \M-convex sets in some compatible way.
\end{remark}

\begin{construction}[$k$-polymatroid lifts]
More generally, we can lift any \M-convex set to a $k$-polymatroid for any non-negative integer $k$.
%\mb{Do we need $k \geq \rank(P)$?} \bs{I don't think so? If $k$ is small we can just increase the size of the ground set. Rank shouldn't matter as we can arbitrarily change the rank by translating}
Let $P \subseteq \Omega \subseteq \ZZ^U$ be as above.
In the notation of \Cref{construction:matroid-lifts}, we write each $b_i$ uniquely as $b_i = m_i \cdot k + r_i$ where $m_i, r_i$ non-negative integers with $0< r_i \leq k$.
Observe that we can write
\[
\Omega_i = \left(\sum_{i=1}^{m_i} [0,k]\right) + [0,r_i] + \underline{\omega}_i \, .
\]
As such, we define the ground set $V = \bigsqcup_{i \in U} V_i$ where $V_i = [m_i + 1]$, and the associated surjection $\phi\colon V \rightarrow U$ that maps $j \in V_i$ to $i$.
Define the box
\[
B = \prod_{i \in U} \left(\left(\prod_{j = 1}^{m_i} [0,k]\right) \times [0,r_i]\right)  \subseteq [0,k]^V \subseteq \ZZ^V \, ,
\]
then $\Omega = \pi_\phi(B) + \underline{\omega}$.

Consider the linking set $\Gamma_{\phi,\underline{\omega}}(B)$ from $\Omega$ to the box $B\subseteq [0,k]^V$, i.e.
\[
\Gamma_{\phi,\underline{\omega}}(B) = \BiggSetOf{(y,-x) \in \ZZ^{V} \times \ZZ^{U}}{x \in \Omega^U \, , \, y \in B \, , \, x_i = (\sum_{j \in V_i} y_j) + \underline{\omega}_i} \, .
\]
Then the induction of $P$ through $\Gamma_{\phi,\underline{\omega}}$ is the set
\[
\lift{B}{\phi}{\underline{\omega}}{P} = \ind{P}{\Gamma_{\phi,\underline{\omega}}(B)} = \SetOf{y \in B}{\pi_{\phi}(y) + \underline{\omega} \in P} \, .
\]
As $\lift{B}{\phi}{\underline{\omega}}{P}$ is \M-convex and contained in $[0,k]^V$, it forms a $k$-polymatroid.
%Moreover, we can induce $P^\natural$ through $\Gamma_\phi$ to get $Q^\natural = \ind{P^\natural}{\Gamma_\phi}$, where $P = (P^\natural)^\uparrow$ and $Q = (Q^\natural)^\uparrow$, and the points in $Q^\natural$ are the independent sets of $Q$.
\end{construction}

\subsection{Compatible lifts}
Consider two \M-convex sets $P, Q \subseteq \ZZ^U$.
We would like to lift these sets in some compatible way such that their quotient structure is preserved.
We first show that if the surjection and translation are the same, then lifts being quotients implies the original \M-convex sets were quotients.

\begin{proposition}[\eqref{lift-quotient} $\implies$ \eqref{Mnatural}]\label{prop:surjection-quotient}
Let $P,Q \subseteq \ZZ^U$ and $M, N \subseteq \ZZ^V$ be \M-convex sets such that
\[
P = \pi_\phi(M) + v \, , \, Q = \pi_\phi(N) + v \, ,
\]
for some surjection $\phi\colon V \rightarrow U$ and vector $v \in \ZZ^U$.
If $M \quotient N$ then $P \quotient Q$.
\end{proposition}
\begin{proof}
Observe that $\pi_\phi$ preserves the partial order given by coordinate sum:
\[
x(V) < y(V) \, \Longleftrightarrow \, \pi_\phi(x)(U) < \pi_\phi(y)(U) \, , \quad \forall x,y \in \ZZ^V \, .
\]
Recall that $M \quotient N$ if and only if there exists an $\Mnat$-convex set $R \subseteq \ZZ^V$ such that $M = R^\uparrow$ and $N = R^\downarrow$.
Let $S = \pi_\phi(R) + v \subseteq \ZZ^U$ be the translated $\Mnat$-convex set in the image of $\pi_\phi$.
As $\pi_\phi$ and translation preserves the partial order given by coordinate sum, we have
\[
S^\uparrow = \pi_\phi(M) + v = P \quad , \quad S^\downarrow = \pi_\phi(N) + v = Q \, .
\]
\end{proof}
The converse to this proposition is not true as the following example shows.

\begin{example}
Consider any two \M-convex sets $M, N \subseteq \ZZ^V$ with $\rank(M) \geq \rank(N)$.
Given the trivial surjection $\phi\colon V \rightarrow \{e\}$ that sends $\phi(i) = e$ for all $i \in V$, the induced projection is $\pi_\phi\colon\ZZ^V \rightarrow \ZZ^e$ where $\pi_\phi(x) = x(V)$.
Let $P = \pi_\phi(M)$ and $Q = \pi_\phi(N) \subseteq \ZZ^e$.
Then $M,N$ are lifts of $P,Q$ respectively with the same projection (and translation).
However, we have $P \quotient Q$ for any choice of $M,N$, but generally we will not have $M \quotient N$.
\end{example}

To rectify this issue, we instead consider a lift of not just $P$ and $Q$, but a local neighbourhood containing them both via a box lift.

\begin{definition}\label{def:compatible+lift}
Let $P, Q \subseteq \ZZ^U$ be \M-convex sets with respective lifts $M, N \in \ZZ^V$.
These lifts are \emph{compatible} if there exists a surjection $\phi\colon V \rightarrow U$, vector $v \in \ZZ^U$ and box $B \subseteq \ZZ^V$ with $M, N \subseteq B$ such that
\[
M = \lift{B}{\phi}{v}{P} \quad , \quad N = \lift{B}{\phi}{v}{Q} \, .
\]
\end{definition}

With this definition, the implication from \eqref{Mnatural} to \eqref{lift-quotient} follows immediately from \Cref{lem:induction+quotient}.

\begin{proposition}[\eqref{Mnatural} $\implies$ \eqref{lift-quotient}]\label{lem:compatible-boxlifts-quotients}
Let $P, Q \subseteq \ZZ^U$ be \M-convex sets such that $P \quotient Q$.
Then for any compatible matroid lifts $M, N \subseteq \ZZ^V$ of $P$ and $Q$ respectively, we have $M\quotient N$.
\end{proposition}

This completes the claim that for any \M-convex sets $P,Q \subseteq \ZZ^U$, we have $P \quotient Q$ if and only if $M \quotient N$ for every compatible matroid lift $M,N$ of $P,Q$.
Matroid quotients have many additional characterizations that we can translate into \M-convex quotient conditions.
For a more comprehensive list, see~\cite[Proposition 7.4.7]{Brylawski:1986}.

We noted in \Cref{prop:truncation+quotient} that truncations $P^{tr}_k$ are the maximal, or generic, quotients of an \M-convex set $P \subseteq \ZZ^U$.
We would like for this maximality to be preserved for any matroid lift $M \subseteq \ZZ^V$ of $P$.
However, as discussed in \Cref{rem:truncation-submodular-vs-polymatroid}, the matroid theoretic definition of the $k$-th truncation of $M$ is $M^{tr}_k \cap \ZZ_{\geq 0}^V$, where the intersection with the positive orthant ensures the resulting object is also a matroid.
The following proposition shows that truncations of \M-convex sets are lifted to matroid truncations of the lifted \M-convex set.

\begin{proposition}
Let $P \subseteq \ZZ^U$ be an \M-convex set and $P^{tr}_k \subseteq \ZZ^U$ its $k$-th truncation.
For any compatible matroid lifts $M, N \subseteq \ZZ^V$ of $P$ and $P^{tr}_k$ respectively, we have $N = M^{tr}_k \cap \ZZ_{\geq 0}^V$.
\end{proposition}
\begin{proof}
Fix $\phi \colon V \rightarrow U$, $v \in \ZZ^U$ and box $B \subseteq \ZZ^V$ such that
\[
M = \lift{B}{\phi}{v}{P} \quad , \quad N = \lift{B}{\phi}{v}{P_k^{tr}} \, .
\]
As $P \quotient P_k^{tr}$ by \Cref{prop:truncation+quotient}, it follows that $M \quotient N$.
Moreover, \Cref{prop:truncation+quotient} along with $N$ being a matroid implies that $N \subseteq M^{tr}_k \cap \ZZ_{\geq 0}^V$.
It suffices to show the other containment.

Let $x \in M^{tr}_k \cap \ZZ_{\geq 0}^V$, then $x \leq y$ in all coordinates for some $y \in M$.
As $\pi_\phi$ is a coordinate summing projection that preserves ranks, we have
\[
z = \pi_\phi(x) + v \leq \pi_\phi(y) + v \in P \, ,
\]
and so $z \in P_k^{tr}$.
As such, $x \in \lift{B}{\phi}{v}{P_k^{tr}} = N$ by definition of the linking set $\Gamma_{\phi,v}(B)$.
\end{proof}

\subsection{Flag \M-convex polytopes} \label{sec:flag+M-convex}
We finish this section by relating compatible matroid lifts to the structure of Minkowski sums of quotients.
Recall that an \M-convex set on $\ZZ^V$ is a matroid (or $1$-polymatroid) if it is contained in the box $[0,1]^V$; for example, the $m$-hypersimplex $\Delta(m,V)$ introduced in \Cref{ex:matroids-M-convex}. 
It was noted in \cite{borovik03_coxetermatroids} that one can determine when two matroids form a quotient from the structure of their Minkowski sum.
We rephrase their statement in terms of \M-convex sets. 

\begin{theorem}[{\cite[Theorem 1.11.1]{borovik03_coxetermatroids}}]
	\label{th:matroid-quotients-minkowski}
	Let $R \subseteq \ZZ^V$ be an \M-convex set.
	Then $R$ is the Minkowski sum $M + N$ of two matroids $M, N \subseteq \ZZ^V$ forming a quotient if and only if the vertices of $R$ form a subset of the vertices of $\Delta(\rank(M),V) + \Delta(\rank(N),V)$.
\end{theorem}
The \M-convex sets obtainable from \Cref{th:matroid-quotients-minkowski} are precisely (the lattice points of) \emph{flag matroid polytopes}.
These polytopes play a crucial role in the study quotients of valuated matroids, as valuated matroid quotients can be characterized by subdivisions of flag matroid polytopes~\cite{BrandtEurZhang:2021}.
We will take a similar approach in \Cref{sec:functions} when we begin characterising quotients of \M-convex functions.
With this in mind, the notion of a `flag \M-convex polytope' is a natural notion to consider.

\begin{definition}\label{def:flag+M+convex+set}
Let $r = (r_1,\dots,r_k)$ be a tuple of strictly increasing positive integers and $\phi: V \to U$ a surjection with corresponding projection $\pi_\phi: \R^V \to \R^U$.
A \emph{flag \M-convex polytope} of type $(r, \phi)$ is an integral translate of a polytope $\pi_{\phi}(S) \subseteq \R^U$ where $S$ is a generalized permutohedron whose vertices are a subset of the vertices of $\sum_{j=1}^k \Delta(r_j,V)$. 
 A \emph{flag \M-convex set} of type $(r, \phi)$ is the set of integer points of a flag \M-convex polytope of type $(r, \phi)$. 
\end{definition}

With this terminology, the polytope $R$ arising in \Cref{th:matroid-quotients-minkowski} is a flag \M-convex set of type $((\rank(M),\rank(N)), \id)$ where $\id\colon V \rightarrow V$ is the identity map, i.e., $\pi_{\id}$ is the trivial projection.

The next result justifies the name of these sets. 
%Making use of the previous results presented in this section, we can make the following statement.

\begin{proposition}[\eqref{lift-quotient} $\iff$ \eqref{compressed-quotient}]
  \label{prop:M-convex-quotients-minkowski}
Let $R \subseteq \ZZ^U$ be an \M-convex set.
Then $R$ is the Minkowski sum $P + Q$ of two \M-convex sets $P,Q \subseteq \ZZ^U$ forming a quotient if and only if it is a flag \M-convex set of type $((\rank(P) +\ell,\rank(Q)+\ell),\phi)$ for some surjection $\phi$ onto $U$ and $\ell \in \ZZ$.
\end{proposition}
%\bs{Adjusted statement and proof slightly to account for translating adding a constant to the ranks (Also in main theorem)}

\begin{proof}
By \Cref{prop:surjection-quotient,lem:compatible-boxlifts-quotients}, $P \quotient Q$ form a quotient if and only if every compatible matroid lift $M,N$ forms a quotient $M \quotient N$.
Fix a choice of a compatible lift, i.e., $\phi \colon V \rightarrow U$ and $v \in \ZZ^U$ such that $P + v = \pi_{\phi}(M)$ and $Q + v = \pi_{\phi}(N)$.
By \Cref{th:matroid-quotients-minkowski}, $M \quotient N$ if and only if the vertices of $M + N$ form a subset of the vertices of $\Delta(\rank(M),V) + \Delta(\rank(N),V)$.
Moreover, setting $\ell = v(U)$ gives that $\rank(M) = \rank(P) + \ell$ 
and $\rank(N) = \rank(Q) + \ell$.
Projection commutes with Minkowski sum, and so $P + Q = \pi_\phi(M + N) - 2v$, giving the claim. 
\end{proof}

The following corollary captures an important special case, namely the flag of \M-convex sets arising from the lattice points in a polymatroid polytope. 

\begin{corollary} 
  Let $P_i \subseteq \ZZ^U$ be \M-convex sets for $0 \leq i \leq k$ with $\rank(P_{i+1}) = \rank(P_i) + 1$ and $P_0$ be the origin.
  Then $(P_0, P_1,\dots, P_k)$ forms a flag of \M-convex sets if and only if there is a surjection $\phi \colon V \to U$ such that  $\sum_{i=1}^{k} P_k = \pi_{\phi}(Q)$ for a subpermutohedron $Q$ of the permutohedron $\Pi_V$. 
\end{corollary}

The definition of flag \M-convex polytopes heavily relies on the surjection $\phi$, which relates the flag \M-convex polytope to flag matroid polytopes via projections.
It would be desirable to obtain a more self-contained notion, which solely relies on notions of \M-convex sets.
We thus pose the following question.

\begin{question}
  Is there a more direct description of flag \M-convex polytopes? 
\end{question}

\section{Quotients of \M-convex functions}\label{sec:functions}

In this section we extend the notions of quotients of \M-convex sets to \M-convex functions.
These naturally generalize \M-convex sets and valuated matroids. 
The goal of this section is to examine various notions of quotients for \M-convex functions and prove a hierarchy between these notions.
We first define all necessary terms, and then state the main result \Cref{thm:function+quotient}.

\begin{definition}[\cite{MurotaShioura:2018simpler}]
	The \emph{effective domain} of a function $f: \Z^E \to \R \cup \{\infty\}$ is the set $\dom f = \set{x \in \Z^E \mid f(x) < \infty}$.

	A function $f : \Z^E \to \R \cup \{\infty\}$ is an \emph{$\Mnat$-convex function} if $\dom(f)\neq\emptyset$ is bounded and the following conditions are satisfied:
	
	\begin{align}
		 &\forall  x, y \in \dom f \text{ with } x(E)>y(E) : 
		 \quad \
		f(x) + f(y) \geq \min_{j \in \supp^{+}(x-y)} \left( f(x - e_j) + f (y + e_j) \right) \, ,  \label{eq:mnat-axiom-between-layers} \\
		 &\forall x, y \in \dom f \text{ with } x(E) = y(E) \, , \, \forall i \in \supp^+(x-y) :  \label{eq:mnat-axiom-within-layers}    \\  \nonumber
		  & \hspace*{13 em}
		  f(x) + f(y) \geq \min_{j \in \supp^-(x-y)} \left(f(x - e_i + e_j) + f(y + e_i - e_j)\right) \, . 
	\end{align}

	A \emph{layer} of an $\Mnat$-convex function is the function restricted to a layer of the effective domain, i.e.
	\[
	f|_{H_k}(x)  = \begin{cases}
		f(x) & \text{ if } x \in \dom (f) \cap H_k \\
		\infty & \text{ otherwise.}
	\end{cases}
	\]
        
	We define the \emph{top layer} $f^{\uparrow}$ and \emph{bottom} layer $f^{\downarrow}$ as the highest and lowest nonempty layer respectively.
 A function  $f: \Z^E \to \R \cup \{\infty\}$  is an \emph{\M-convex function} if it is a layer of an $\Mnat$-convex function, i.e., satisfies \eqref{eq:mnat-axiom-within-layers}.
 The \emph{rank} of an \M-convex function is $\rank(f) = \rank(\dom(f))$.
\end{definition}
As we work almost entirely with bounded \M-convex sets, we define \M-convex functions to have bounded effective domain.
We note that this restriction is relaxed in many places in the literature.

Right from the definition one obtains that the effective domain of an \M-convex (resp. $\Mnat$-convex) function is always an \M-convex (resp. $\Mnat$-convex) set, and so is the set of minimizers.
This also holds in the reverse direction leading to an alternative characterization of \M-convex (resp. $\Mnat$-convex) functions in terms of their minimizers.
Given a function $f\colon\ZZ^E \rightarrow \Rinf$ and any linear functional $u \in (\R^E)^*$, the \emph{minimizer} $f^u$ of $f$ is the set
\[
f^u = \SetOf{x \in \ZZ^E}{f(x) - \langle u, x\rangle \leq f(y) - \langle u, y\rangle  \, \forall y \in \dom(f)} \, .
\]
%In the notation of \cite{Murota:2003} we have $f^u = \argmin f[-u]$.

\begin{theorem}[{\cite[Theorem 6.30]{Murota:2003}}] \label{thm:minimizers}
  Let $f\colon \ZZ^E \rightarrow \RR \cup \{\infty\}$ be a function with bounded non-empty effective domain.
$f$ is an \M-convex (resp. $\Mnat$-convex) function if and only if all of its minimizers $f^u$ are \M-convex (resp. $\Mnat$-convex) sets for all $u \in (\R^E)^*$.
\end{theorem}

\begin{example}
Valuated matroids are precisely \M-convex functions $f\colon \ZZ^V \rightarrow \RR \cup \{\infty\}$ whose effective domain is contained in the unit hypercube $\dom(f) \subseteq \{0,1\}^V$.
As a concrete example, consider a weighted bipartite graph $G = (V, U; \cE)$ where $\cE \subseteq V \times U$, with edge weights $w \colon \cE \rightarrow \RR$.
One can define an \M-convex function $f_G \colon \{0,1\}^V \rightarrow \RR \cup \{\infty\}$ by
\[
f_G(A) = \min\BiggSetOf{\sum_{e \in \mu} w(e)}{\mu \subseteq \cE \text{ matching of maximal cardinality s.t. } \partial_V(\mu) = A} \, .
\]

The $\Mnat$-convex functions whose effective domains' are contained in the unit hypercube are the \emph{valuated generalized matroids}.
We can relax the previous example to get an \M-convex function $f^\natural_G \colon \{0,1\}^V \rightarrow \RR \cup \{\infty\}$ by
\[
f^\natural_G(A) = \min\BiggSetOf{\sum_{e \in \mu} w(e)}{\mu \subseteq \cE \text{ matching s.t. } \partial_V(\mu) = A} \, .
\]
\end{example}

Many of the operations for \Cref{sec:operations} extend to \M-convex and $\Mnat$-convex functions. 
Consider $f,g\colon\ZZ^E \to \Rinf$ functions, and write $E = V \sqcup U$ as a disjoint union.

The \emph{restriction} of $f$ to $V$ is 
\[
f|_V \colon \ZZ^V \rightarrow \Rinf \quad , \quad f|_V(x) = f(x,\0_U) \, .
\]
The \emph{projection} of $f$ onto $V$ is 
\[
\pi_V(f) \colon \ZZ^V \rightarrow \Rinf \quad , \quad \pi_V(f)(x) = \inf\{f(x,y) \mid y \in \Z^U \} \, .
\]
The \emph{convolution} of $f$ and $g$ is 
\[
f \square g \colon \Z^E \to \Rinf \quad , \quad (f \square g)(x) = \inf\{f(x_1) + g(x_2) \mid x_1 + x_2 = x, x_1,x_2 \in \Z^V\} \, .
\]
Convolution is the function equivalent of Minkowski sum for sets.
If $f,g$ are $\Mnat$-convex functions, their restriction, projection and convolution are also all $\Mnat$-convex functions.
Moreover, if they are \M-convex functions, their restriction and convolution are also \M-convex functions~\cite[Section~6.4]{Murota:2003}.

As with $\Mnat$-convex sets, we can define a notion of induction for $\Mnat$-convex functions.
Analogously to \Cref{def:linking+set}, we call an \M-convex function $\gamma: \ZZ^V \times \ZZ^U \to \Rinf$ a \emph{linking function} from $V$ to $U$.
The \emph{left function of $\gamma$} is the $\Mnat$-convex function $\pi_V(\gamma)$.
The following is an adaptation of the framework of valuated polylinking systems~\cite{KobayashiMurota:2007}.

\begin{definition}
	Let $\gamma : \ZZ^V \times \ZZ^U \to \Rinf$ be a linking function and $f : \ZZ^U \to \Rinf$ an $\Mnat$-convex function.
	The \emph{induction of $f$ through $\gamma$} is the $\Mnat$-convex function $\ind{f}{\gamma} \colon \ZZ^V \to \Rinf$
	\begin{align*}
		\ind{f}{\gamma}(x) &= \inf_{y \in \Z^U}(\gamma(x,-y) + f(y)) = (\gamma \square (\0_V \times f))|_V \, ,
	\end{align*}
		where $\0_V \times f : \Z^V \times \Z^U \to \Rinf$ is the function that sends $(\0_V,y)$ to $f(y)$ and everything else to $\infty$. 
\end{definition}

\begin{example}
Over \Cref{ex:bipartite+graph+induction,ex:left+set+transversal+matroid,ex:bipartite+induction+generalization}, we showed how induction by bipartite graph is a special case of induction by linking set.
Here we extend this by defining a linking function from a weighted bipartite graph, and considering induction by weighted bipartite graph as a special case of induction by linking function.

Let $G = (V,U;\cE)$ be a bipartite graph with edge weights $w\colon \cE \rightarrow \RR$.
Identifying subsets $A \subseteq V$ with their indicator vectors $e_A \in \ZZ^V$, we define a linking function $\gamma_G \colon \ZZ^V \times \ZZ^U \rightarrow \RR \cup \{\infty\}$ by
\[
\gamma_G(e_A, -e_B) = \min\BiggSetOf{\sum_{e \in \mu} w(e)}{\mu\subseteq \cE \text{ matching s.t } \partial_V(\mu) = A, \partial_U(\mu) = B} \, ,
\]
with $\gamma_G(x,y) = \infty$ if no such matching exists.

Now consider a valuated generalized matroid $f\colon \{0,1\}^U \rightarrow  \RR \cup\{\infty\}$.
The induction of $f$ through $\gamma_G$ is the valuated generalized matroid $\ind{f}{\gamma_G}\colon \{0,1\}^V \rightarrow \RR\cup \{\infty\}$ defined as
\begin{align*}
\ind{f}{\gamma_G}(e_A) &= \min_{e_B\in \{0,1\}^U}(\gamma_G(e_A,-e_B) + f(e_B)) \\
&= \min\biggSetOf{f(e_B) + \sum_{e \in \mu} w(e)}{\mu\subseteq \cE \text{ matching s.t } \partial_V(\mu) = A, \partial_U(\mu) = B} \, ,
\end{align*}
or taking the value $\infty$ if no such matching exists.
\end{example}

Now we are equipped to state the \M-convex function version of our main theorem.

\begin{manualtheorem}{\labelcref*{thm:functions-intro}}\label{thm:function+quotient}
  Let $f,g : \Z^E \to \R \cup \{\infty\}$ be \M-convex functions such that $\rank(g) < \rank(f)$.
  Consider the following statements. 
	\begin{enumerate}[label=(\Alph*)]
	  \item (top and bottom) There exists an $\Mnat$-convex function $h: \Z^E \to \Rinf$ such that $f$ and $g$ are the top and bottom layers of $h$ respectively. \label{quo:Mnat}
	  \item (induction)
			There exists a linking function $\gamma : \ZZ^E \times \ZZ^{\tE} \to \Rinf$ and an \M-convex function $r : \ZZ^{\tE} \to \Rinf$ such that $f$ is the left function of $\gamma$, and $g$ is the induction of $r$ through $\gamma$, i.e.
			\[
			f = \pi_E(\gamma)^\uparrow \quad , \quad g = \ind{r}{\gamma} \, .
			\]\label{quo:induction}
		      \item (exchange property)
                        For every $x \in \dom(f), y \in \dom(g), i \in \supp^+(y-x)$ there exists a $j \in \supp^-(y-x)$ such that \label{quo:exchange}
		\[
		f(x) + g(y) \geq f(x + e_i - e_j) + g(y - e_i + e_j).
		\]
		\item (minimizers) For every $u \in (\R^E)^*$ the minimizers $f^u \quotient g^u$ are quotients as \M-convex sets. \label{quo:minimizers}
	\end{enumerate}
	Then \ref{quo:Mnat} $\implies$ \ref{quo:induction} $\implies$ \ref{quo:exchange}$\implies$ \ref{quo:minimizers}.
	For elementary quotients, that is if $\rank(f) = \rank(g)+1$, these are all equivalences.         
\end{manualtheorem}

Observe that \ref{quo:Mnat}, \ref{quo:induction} and \ref{quo:exchange} are the valuated analogs of (\ref{Mnatural}), (\ref{induction}) and (\ref{exchange-property}), respectively.
We will write $f \exquotient g$ if $f,g$ satisfy the exchange property \ref{quo:exchange} and $f \minquotient g$ if $f,g$ satisfy the minimization property \ref{quo:minimizers}. 

\begin{proof}
{[\ref{quo:Mnat} $\implies$ \ref{quo:induction}]}
	Let $h$ be the $\Mnat$-convex function with $h^\uparrow = f$ and $h^\downarrow = g$.
	We define the \M-convex functions $\gamma : \ZZ^{V} \times \Z \to \Rinf$ and $r: \Z \to \Rinf$ as follows:
	\[
		\gamma(x,y) = \begin{cases}
			h(x) & \text{ if } y = -x(V) + \rank(f) \\
			\infty & \text{ otherwise}.
			\end{cases} \, , \quad 
		r(y) = \begin{cases}
			0 & \text{ if } y = \rank(f) - \rank(g) \\ 
			\infty & \text{ otherwise}.
		\end{cases}
	\]
	Then $\pi_V(\gamma) = h$ and thus $\pi_V(\gamma)^\uparrow = f$.
	Furthermore, 
	\[
		\gamma(x,-y) + r(y) = \begin{cases}
			\gamma(x,-y) & \text{ if } y = \rank(f) - \rank(g) \\
			\infty & \text{ otherwise}
		\end{cases}
	=
	\begin{cases}
		h(x) & \text{ if } x(V) =  \rank(g) \\
		\infty & \text{ otherwise}
	\end{cases},
	\]
	so $\ind{r}{\gamma} = g$.

{[\ref{quo:induction} $\implies$ \ref{quo:exchange}]}
	Let $x \in \dom(f), y \in \dom(g)$ and $i \in \supp^+(y-x)$.
There exist $z,w \in \Z^{\tE}$ such that $f(x) = \gamma(x,w)$ and $g(y) = \gamma(y,-z)+r(z)$.
Furthermore, $x(E) = \rank(f)$ is maximal since $f$ is the top layer.
Since $\gamma$ is \M-convex, there exists some $j \in \supp^-((y,-z) -(x,w))$ such that
	\begin{equation} \label{eq:M-convex-function-proof}
		\gamma(x,w) + \gamma(y,-z) \geq \gamma((x,w) + e_i - e_j) + \gamma((y,-z) - e_i + e_j) \, .
	\end{equation}
	If $j \in \supp^-(y-x)$, then the construction of $f$ and $g$ implies that
	\begin{align*}
		f(x) + g(y) &= \gamma(x,w) + \gamma(y,-z) + r(z) \\
		&\geq \gamma(x+ e_i- e_j,w) \gamma(y- e_i + e_j,-z) + r(z) \\
		&\geq  f(x + e_i - e_j) + g(y-e_i+e_j) \, .
	\end{align*}
	To complete the proof, suppose that $j \in \supp^-(-z-w)$.
	Then \eqref{eq:M-convex-function-proof} implies that $(x+e_i,w-e_j) \in \dom(\gamma)$.
	However, $(x + e_i)(E) = x(E) +1 > \rank(f)$, which contradicts that $f$ is the top layer.

{[\ref{quo:exchange}  $\implies$ \ref{quo:minimizers}]}
	Let  $x \in f^u, y \in g^u$ and $i \in \supp^+(y-x)$.
	As $f \exquotient g$, there exists some $j \in \supp^-(y-x)$ such that
	$
	f(x) + g(y) \geq f(x + e_i - e_j) + g(y - e_i + e_j).
	$
  Hence we have
	\begin{align*}
	f(x) - \langle u, x \rangle + g(y) - \langle u, y \rangle &\geq f(x + e_i - e_j) - \langle u, x \rangle + g(y - e_i + e_j) - \langle u, y \rangle \\
	&= f(x + e_i - e_j) - \langle u, x + e_i - e_j \rangle + g(y - e_i + e_j) - \langle u, y - e_i + e_j \rangle \, . 
	\end{align*}
	Since $x \in f^u$ and $y \in g^u$, the expression on the left hand side is minimal among all $x \in \dom(f)$ and $y \in \dom(g)$. 
	Thus, the above expression holds with equality and we have $x + e_i - e_j \in f^u$ and $y - e_i + e_j \in g^u$. 

{[\ref{quo:minimizers} $\implies$ \ref{quo:Mnat} when $\rank(f) = \rank(g) + 1$]}
  Define $h \colon \Z^E \to \Rinf$ by 
  \[
  h(x) = \begin{cases} f(x) & x\in \dom(f) \\ g(x) & x \in \dom(g) \\ \infty & \text{otherwise} \end{cases} \, .
  \]
  As $\dom(h) = \dom(f) \cup \dom(g)$, the minimizer $h^u$ is an element of $\{f^u, g^u, f^u \cup g^u\}$ for each $u \in (\R^E)^*$.
  Utilizing \Cref{thm:minimizers}, $f^u$ and $g^u$ are both \M-convex sets.
  Moreover, $f^u \quotient g^u$ and $\rank(f^u) = \rank(g^u) + 1$ coupled with~\eqref{Mnatural} imply that $f^u \cup g^u$ is an $\Mnat$-convex set.
As such, $h^u$ is an $\Mnat$-convex set for all $u \in (\R^E)^*$, and so $h$ is an $\Mnat$-convex function by \Cref{thm:minimizers}.
\end{proof}

Combining the definition of $\minquotient$ with \Cref{lem:M-convex-quotients-transitive}, the transitivity of quotients of \M-convex sets, yields the transitivity of $\minquotient$.

\begin{corollary}
  Let $f,g,h \colon \Z^E \to \R \cup \{\infty\}$ be \M-convex functions such that $\rank(h) < \rank(g) < \rank(f)$.
  Then $f \minquotient g$ together with $g \minquotient h$ implies $f \minquotient h$. 
\end{corollary}

%\bsinline{Added remark on equivalence of \ref{quo:exchange} and \ref{quo:minimizers} for valuated matroids}
\begin{remark}
In the special case of valuated matroids, we have an equivalence between \ref{quo:exchange} and \ref{quo:minimizers}.
If $f,g$ satisfy $f^u \quotient g^u$ for all $u \in (\RR^E)^*$, then \Cref{th:matroid-quotients-minkowski} implies that $f^u + g^u$ is (the lattice points of) a flag matroid polytope.
By~\cite[Theorem A]{BrandtEurZhang:2021}, this is equivalent to $f \exquotient g$.
\end{remark}

\begin{remark}
The implication \ref{quo:exchange} $\implies$ \ref{quo:Mnat} can be reframed as $f \exquotient g$ implies the existence of a \M-convex function $\tilde{h}$ on $E \cup A$ such that $f = h \setminus A$ is the deletion of $h$ and $g = h/A$ is the contraction of $h$.
While \cite[Theorem 2.23]{JarraLorscheid:2024} shows the reverse implication holds for all `matroids with coefficients', it is not clear for which generalizations of matroids this property holds, including valuated matroids.
For example, it is known to not be true for oriented matroids: one can construct a strong map of oriented matroids that does not factor as an extension followed by a contraction~\cite{Richter-Gebert:1993}.
To determine for which `matroids with coefficients' this property holds is \cite[Problem 2]{JarraLorscheid:2024}.
\end{remark}

\begin{example}\label{ex:valuated+truncation}
  The paper~\cite{Murota:1997} introduced truncation and elongation of a valuated matroid. 
  Let $f\colon \{0,1\}^E \to \RR \cup \{\infty\}$ be a valuated matroid of rank $d$.  
  For $d \geq 1$, its \emph{truncation} is the function $f^{tr} \colon \{0,1\}^E \to \RR \cup \{\infty\}$ of rank $d-1$ with $f^{tr}(x) = \min(f(y) \mid x \leq y)$.  
  Dually for $d \leq n-1$, its \emph{elongation} is the function $f^{el} \colon \{0,1\}^E \to \RR \cup \{\infty\}$ of rank $d+1$ with $f^{el}(x) = \min(f(y) \mid x \geq y)$.
  As they are elementary, they form quotients with $f$ in any of the ways of \Cref{thm:function+quotient}.

  These notions can be extended analogously to arbitrary \M-convex functions $f\colon \Z^E \to \RR \cup \{\infty\}$; the proof that they yield quotients via the exchange property follows in a similar fashion to the proof of \cite[Lemma 3.8]{JoswigLohoLuberOlarte:2023}.

We saw in \Cref{prop:truncation+quotient} that truncation of unvaluated matroids and \M-convex sets yields quotient that are maximal, or generic in some sense.
It would be interesting to investigate what the analogous result should be for valuated matroids and \M-convex functions.
\end{example}

\subsection{Flags of \M-convex functions}

When $\rank(f) > \rank(g) + 1$, the difficulty of showing any equivalence with \ref{quo:Mnat} comes from `filling in' the layers between $f$ and $g$ to make a coherent $\Mnat$-convex function.
%As far as the authors know, this is not made any easier by assuming \ref{quo:exchange} or \ref{quo:induction} either.
We show some of the pitfalls in proving \ref{quo:exchange} $\implies$ \ref{quo:Mnat}.
As the following example shows, it is not even sufficient to find functions on all the layers between $f$ and $g$ that satisfy the exchange property, and whose supports form an $\Mnat$-convex set.

\begin{example}
Consider the \M-convex functions $f,g \colon \ZZ^2 \to \Rinf$, defined as
\[	
f(x) = \begin{cases} l_1 & x = (1,1) \\ \infty & \text{ otherwise } \end{cases} \, , \quad g(y) = \begin{cases} l_2 & y = (0,0) \\ \infty & \text{ otherwise } \end{cases} \, , \quad l_1,l_2 \in \RR \, .
\]
These functions form a quotient $f \exquotient g$ as the exchange property is trivially satisfied.

When trying to extend them to a $\Mnat$-convex function $h$ with $h^\uparrow = f$ and $h^\downarrow = g$, we know that $\dom(h) = \{(0,0),(1,0),(0,1),(1,1)\}$, as this is the unique $\Mnat$-convex set with $\dom(h)^\uparrow = \dom(f)$ and $\dom(h)^\downarrow = \dom(g)$.
We need to find a rank $1$ \M-convex function $r$ to fill in the middle layer.
It turns out that we can pick the values on the effective domain entirely freely, i.e.
\[
r(x) = \begin{cases} k_1 & x = (1,0) \\ k_2 & x = (0,1) \\ \infty & \text{otherwise} \end{cases} \, , \quad k_1,k_2 \in \RR \, ,
\]
and get a flag of quotients, namely $f \exquotient r \exquotient g$.
However, the resulting function $h(x) = \inf(f(x),r(x),g(x))$ may not be $\Mnat$-convex without additional restrictions.
Notably, if $l_1 + l_2 < k_1 + k_2$, then
\[
h(1,1) + h(0,0) = l_1 + l_2 < k_1 + k_2 = h(1,0) + h(0,1) \, .
\]
This shows that a sequence of quotients $\exquotient$ does not always yield the layers of an $\Mnat$-convex function, even if their effective domains are layers of an $\Mnat$-convex set.

Despite this, we can remedy this in the following way.
Fixing arbitrary values for $l_1,l_2,k_1,k_2$, we can always find a constant $c \in \RR$ such that $l_1 + l_2 + c > k_1 + k_2$, so adding a constant function to one of the three layers suffices to yield layers of an $\Mnat$-convex function.
The following theorem shows this can be done in general.
\end{example}

\begin{manualtheorem}{\labelcref*{th:flags-intro}}
\label{th:flags-adding-constants}
	Let $f_0,f_1,\dots,f_k: \Z^E \to \Rinf$ be \M-convex functions such that $f_{i+1} \exquotient f_i$ and $\rank(f_i)=\rank(f_0)+i$ for all $i \in [k]$,	and $P = \bigcup_{i=0}^k \dom(f_i)$ is an $\Mnat$-convex set.
	Then there exist constants $c_0,c_1,\dots,c_k \in \RR$ such that $f_0 - c_0, f_1 - c_1, \dots, f_k - c_k$ are the layers of an $\Mnat$-convex function
	\[
	h(x) = \inf(f_0(x)-c_0,\dots,f_k(x)-c_k) \, .
	\]

  Concretely, we can choose $c_0=c_1=0$ and for $m = 2,\dots, k$ we can recursively choose
	\[
	 c_{m} = \min\biggSetOf{\begin{array}{l}
	 f_{m}(x) - f_{m-1}(x-e_j) + c_{m-1} \\ - f_{l+1}(y+e_j) + c_{l+1} + f_l(y) - c_l
	 \end{array}}{\begin{array}{l}
	 y \in \dom(f_{l}) \, , \, 0 \leq l \leq m-2 \, ,\\ x \in \dom(f_m) \, , \, j \in \supp^-(y-x)
	 \end{array}}
    %		c_{m} = \min_{y \in \dom(f_m)}   \min_{l \in [m-1]}   \min_{x \in \dom(f_{l})}   \min_{j \in \supp^-(x-y)}	f_{m}(y) - f_{m-1}(y-e_j) + c_{m-1} - f_{l+1}(x+e_j) + c_{l+1} + f_l(x) - c_l .
    %\min(\{f_{k+1}(y) - f_k(y-e_j)+c_k-f_{l+1}(x+e_j)+c_{l+1}+f_l(x)-c_l \mid l \in [k], x \in \dom(f_l), y \in \dom(f_{k+1}, j \in \supp^-(y-x))\}) \, .
	\]
\end{manualtheorem}

\begin{proof}
	We prove the statement by induction on $k$.
	For $k=1$, the proof of \Cref{thm:function+quotient} shows that we can choose $c_0 = c_1 = 0$.
	Let $f_0,\dots,f_k,f_{k+1}: \Z^E \to \Rinf$ be \M-convex functions satisfying the assumptions.
	Recall from \Cref{thm:box+plank} that if $\bigcup_{l = 0}^{k+1} \dom(f_l)$ is an $\Mnat$-convex set, then so is $\bigcup_{l = 0}^{k} \dom(f_l)$.
	Thus, by induction there exist constants $c_0,\dots,c_k \in \RR$ such that $h^k(x) = \inf(f_0-c_0,\dots,f_k-c_k)$ is an $\Mnat$-convex function.
  Define $h^{k+1}$ analogously where $c_{k+1}$ is as defined in the statement, and note that
	\[
	h^{k+1}(x) = \begin{cases}
		f_{k+1}(x)-c_{k+1} & \text{if } x \in \dom(f_{k+1}) \\
		h^k(x) & \text{otherwise}.
	\end{cases}
	\]
	Note that the minimum in the definition of $c_{k+1}$ does indeed exist, as $P$ is a bounded $\Mnat$-convex set and $x(E) > y(E)$ guarantees the existence of some $j \in \supp^-(y-x)$.
	We show that $h^{k+1}$  is an $\Mnat$-convex function.
	
  By induction, both conditions	\eqref{eq:mnat-axiom-between-layers} and \eqref{eq:mnat-axiom-within-layers} are satisfied for all $x,y \in \dom(h^k)$.
  Moreover, condition \eqref{eq:mnat-axiom-within-layers} holds for all $x,y \in \dom(f_{k+1})$ as $f_{k+1}$ is \M-convex.
	It remains to show~\eqref{eq:mnat-axiom-between-layers} holds for all $x \in \dom(f_{k+1})$ and $y \in \dom(h^k)$.
	There exists a unique $l \in [k]$ such that $y \in \dom(f_l)$.
	If $l = k$, it is clear that $f_{k+1} \exquotient f_{k}$ implies $f_{k+1} -c_{k+1} \exquotient f_{k}-c_{k}$.
	As such, \eqref{eq:mnat-axiom-between-layers} follows from the conditions in  \Cref{thm:function+quotient} all being equivalent when $\rank(f_{k+1}) = \rank(f_{k}) + 1$.
	When $l \neq k$, the following holds for any $j \in \supp^-(y-x)$ by definition of $c_{k+1}$:
	\begin{align*}
		h^{k+1}(x) + &h^{k+1}(y) 
			= f_{k+1}(x) - c_{k+1} + f_l(y) - c_l \\
			&\geq f_{k+1}(x) - (f_{k+1}(x) - f_k(x-e_j)+c_k-f_{l+1}(y+e_j)+c_{l+1}+f_l(y)-c_l) + f_l(y) - c_l  \\
			&=  f_k(x-e_j) - c_k + f_{l+1}(y+e_j) - c_{l+1} \\
			&= h^{k+1}(x - e_j) + h^{k+1}(y+e_j) \, .
	\end{align*}
\end{proof}

\begin{example}
Let $f_0,f_1,f_2: \Z^E \to \Rinf$ be \M-convex functions such that $f_2 \exquotient f_1 \exquotient f_0$ and $\dom(f_0)\cup \dom(f_1) \cup \dom(f_2)$ is an $\Mnat$-convex set.
Applying \Cref{th:flags-adding-constants}, we can find a constant
\[
c = \min\SetOf{f_2(x) + f_0(y) - f_1(x-e_j) - f_1(y+e_j)}{x \in \dom(f_2) \, , \, y \in \dom(f_0) \, , \, j \in \supp^-(y-x)}
\]
such that $h(x) = \inf(f_0(x),f_1(x),f_2(x) -c)$ is $\Mnat$-convex.
However, this is not a unique choice: any $c' \leq c$ will also satisfy the conditions of $\Mnat$-convexity.
Moreover we can choose to do the scaling on $f_0$ or $f_1$ instead.
As the non-trivial inequality $h$ must satisfy is for all $x \in \dom(f_2)$ and $y \in \dom(f_0)$
\begin{align*}
h(x) + h(y) &= f_2(x) - c + f_0(y) \\
 &\geq \min_{j \in \supp^+(x-y)}(f_1(x - e_j) + f_1(y + e_j)) = \min_{j \in \supp^+(x-y)}(h(x - e_j) + h(y + e_j)) \, ,
\end{align*}
we could equally let $h(x) = \inf(f_0(x) -c ,f_1(x),f_2(x))$, or even $h(x) = \inf(f_0(x),f_1(x)+ c/2,f_2(x))$.
This flexibility only increases in complexity as the number of consecutive quotients grows.
\end{example}

We end by noting an important class of $\M$-convex functions for which Theorem~\ref{thm:function+quotient} is an equivalence.
A matroid $M \subseteq [0,1]^E$ of rank $m$ is \emph{sparse paving} if for all $x, y \in \Delta(m,E) \setminus M$, we have $z(E) \leq m-2$ where $z = \min(x,y)$ coordinatewise minimum of $x$ and $y$.
Sparse paving matroids are an important class of matroids as it is conjectured that asymptotically all matroids are sparse paving.
Moreover, they have been used to derive bounds on the dimension and number of cells of the Dressian, the polyhedral fan parametrising all valuated matroids~\cite{JoswigSchroter:2017,Pendavingh:2024}.

We define a \emph{sparse paving valuated matroid} $f \colon \{0,1\}^E \rightarrow \RR \cup \{\infty\}$ to be a function such that $\argmin(f)$ is a sparse paving matroid.
The class of sparse paving valuated matroids is very flexible, as the values $f(x)$ for $x \notin \argmin(f)$ can be arbitrary values greater than $\min(f(x))$.
The proof that $f$ is a valuated matroid is given in~\cite[Lemma A.2, Remark A.3]{HusicLohoSmithVegh:2022}.

\begin{proposition}
Let $f,g$ be sparse paving valuated matroids.
If $f \minquotient g$ then there exists an $\Mnat$-convex function $h \colon \{0,1\}^E \rightarrow \RR \cup \{\infty\}$ such that $h^\uparrow = f$ and $h^\downarrow = g$.

In particular, the conditions in \Cref{thm:function+quotient} are equivalent for sparse paving valuated matroids. 
\end{proposition}

\begin{proof}
Write $M = \argmin(f)$ and $N = \argmin(g)$ for the associated sparse paving matroids of $f$ and $g$ of ranks $\rank(f) = m$ and $\rank(g) = n$.
Furthermore, write $c_f := f(x)$ for $x \in M$ and $c_g := g(x)$ for $x \in N$.
We define $h \colon \{0,1\}^E \rightarrow \RR \cup \{\infty\}$ by
\[
h(x) := \begin{cases}
f(x) & x \in \dom(f) \\
g(x) & x \in \dom(g) \\
\min(c_f, c_g) & n < x(E) < m \\
\end{cases}
\]
and $h(x) = \infty$ otherwise.
Each layer of $h$ is $\M$-convex, it remains to show \eqref{eq:mnat-axiom-between-layers} holds.
We show this for the case where $c_g \leq c_f$, the case where $c_f \leq c_g$ is identical.

Suppose that \eqref{eq:mnat-axiom-between-layers} does not hold, and there exists some $x, y \in \dom(h)$ with $x(E)>y(E)$ such that
\begin{equation} \label{eq:spvm+not+mnat}
h(x) + h(y) < h(x - e_j) + h(y + e_j) \quad \forall j \in \supp^+(x-y) \, .
\end{equation}
Note that if $|\supp^+(x-y)| = 1$, then we have $x = y + e_j$ and \eqref{eq:mnat-axiom-between-layers} trivially holds with equality. %\todo[inline]{GL: Why does it hold then?\\
%BS: We can rewrite $h(x) +(y) = h(y+e_j) + h(x-e_j) = \min_{i \in \supp^+(x-y)}(h(y+e_i) + h(x-e_i))$}
Hence we can assume that there exists distinct $j, j' \in \supp^+(x-y)$.

If $x(E) \neq m$, note that by construction we have $2c_g \leq h(x) + h(y)$ with equality if $x,y \notin\dom(g) \setminus N$.
As $h(x-e_j) +h(y+e_j) > 2c_g$, at least one of $x-e_j$ and $y+e_j$ must be contained in $
\dom(g) \setminus N$.
As $(y+e_j)(E) \geq n+1$,
 we must have $x-e_j \in \dom(g) \setminus N$, and moreover that $x(E) = n + 1$.
As \eqref{eq:spvm+not+mnat} holds for $j'$, an identical argument shows $x-e_{j'} \in \dom(g) \setminus N$ also.
%\todo[inline]{GL: I feel like I am very confused again: doesn't $x(E) = n + 1$ imply that $(x-e_j)(E) = (x-e_{j'})(E) = n$?\\
%BS: Yes, hence they are `removed bases' of the sparse paving matroid $N$ that overlap in $n-1$ elements, giving a contradiction. Tried to flesh this paragraph out}
However, setting $z = x -e_j - e_{j'} = \min(x-e_{j},x-e_{j'})$, we have $z(E) = n-1 > n-2$, contradicting that $N$ is sparse paving.

If $x(E) = m$, note that by construction we have $c_f + c_g \leq h(x) + h(y)$.
A similar argument as above shows \eqref{eq:spvm+not+mnat} can only hold if $y+e_j, y+e_{j'} \in \dom(f) \setminus M$, implying that $y(E) = m - 1$.
As $y = \min(y+e_j, y+e_{j'})$, this again contradicts that $M$ is sparse paving.
\end{proof}

\subsection{Subdivisions from flags of \M-convex functions}

We introduced flag \M-convex sets as Minkowski sums of flags of \M-convex sets in \Cref{sec:flag+M-convex}.
This can even be used as a characterization of quotients. 
Recall that convolution is the generalization of Minkowski sum for functions. 
Now, we investigate the interplay between quotients of \M-convex functions and convolution. 

The following basic property can be derived from the fact that faces of Minkowski sums are sums of faces of the summands.
Applying this to the convex hull of the epigraph of the functions yields the following. 

\begin{lemma}[{\cite[Lemma 2.2.6]{BrandtEurZhang:2021}}] \label{lem:minimizers-convolution}
  Let $f_0,f_1,\dots,f_k: \Z^E \to \Rinf$ be \M-convex functions with finite support and $u \in (\R^E)^*$. 
  Then $\left(f_0 \square \dots \square f_k\right)^u = \sum_{i=0}^{k} f_i^u$. 
\end{lemma}

Next, we need a simple observation for sums. 
Recall that an \M-convex set $P \subseteq \ZZ_{\geq 0}^E$ is contained in the nonnegative orthant if and only if its submodular function $p$ is non-decreasing, as discussed in \Cref{sec:box+lifts} for $k$-polymatroids.

\begin{lemma} \label{lem:quotient-of-sum}
Let $P, Q, S \subseteq \ZZ^E$ be \M-convex sets such that $P \quotient Q$ and $S \subseteq \ZZ_{\geq 0}^E$ contained in the nonnegative orthant.
%  Let $p,q,s \colon 2^E \rightarrow \ZZ$ be submodular functions such that $p \quotient q$ and $s$ is nondecreasing.
%  Then $(p+s) \quotient q$.
  Then $P + S \quotient Q$. 
\end{lemma}
\begin{proof}
Consider the corresponding submodular functions $p,q,s$.
  For $X \subseteq Y \subseteq E$, we obtain $q(Y) - q(X) \leq p(Y) - p(X)$ by~\Cref{def:compliant}.
  $S \subseteq \ZZ_{\geq 0}^E$ implies that $s$ is nondecreasing, hence one additionally has $0 \leq s(Y) - s(X)$.
  Adding the inequalities yields the claim. 
\end{proof}

%Recall that matroid rank functions are submodular and nondecreasing.
%Therefore, applying the lemma to a flag $(r_1,\dots,r_k)$ of matroids yields that $r \quotient r_i$ for each $i \in [k]$ where $r = \sum_{i=1}^{k}r_i$.
As a special case of this, note that the constituents of a flag of matroids $(R_1, \dots, R_k)$ viewed as \M-convex sets are all contained in nonnegative orthant.
Therefore, applying the lemma yields that $R \quotient R_i$ for each $i \in [k]$ where $R = \sum_{i=1}^{k} R_i$.
In particular, the matroid base polytope of each constituent of a flag matroid is a quotient of the flag matroid polytope.
This idea even extends to valuated matroids.
We derive its \M-convex function generalization. 

\begin{proposition}\label{thm:aggregation+quotient}
  Let $f,g,h \colon \Z^E \to \R \cup \{\infty\}$ be \M-convex functions with support in the nonnegative orthant such that $g \minquotient f$ and $h \minquotient f$. 
  Then $(g \square h) \minquotient f$.
\end{proposition}
\begin{proof}
  Fix an arbitrary $u \in (\R^E)^*$.
  By \Cref{lem:minimizers-convolution}, we have $\left(g \square h\right)^u = g^u + h^u$.
%    It is a basic property of a submodular function (following from the greedy algorithm representation of the submodular base polytope) that it is nondecreasing exactly if the corresponding base polytope is in the nonnegative orthant.
  Moreover, $g^u \subseteq \dom(g)$ is contained in the nonnegative orthant.
  Applying \Cref{lem:quotient-of-sum}, we get that $\left(g \square h\right)^u \quotient f^u$.
  This implies the claim. 
\end{proof}

We conclude with a generalization of \cite[Theorem 5.3]{FujishigeHirai:2022} and \cite[Theorem 4.4.2]{BrandtEurZhang:2021}. 
It follows from \Cref{prop:M-convex-quotients-minkowski} and \Cref{lem:minimizers-convolution}. 

\begin{theorem} \label{thm:valuted+flag+subdivision}
  Let $f_0,f_1,\dots,f_k: \Z^E \to \Rinf$ be \M-convex functions such that $f_{i+1} \minquotient f_i$. 
  Then, for each $u \in (\R^E)^*$, the minimizer $\left(f_0 \square \dots \square f_k\right)^u$ is a flag \M-convex set of type $\left((\rank(f_0)+\ell,\dots,\rank(f_k)+\ell),\phi\right)$ for some surjection $\phi$ onto $E$ and $\ell \in \ZZ$.
\end{theorem}

\section{Generalizing to non-integral quotients} \label{sec:non-integral}

The geometric viewpoint on quotients allows us to generalize beyond discrete structures.
Many of the equivalent characterizations of quotients of \M-convex sets, or equivalently \emph{integral} submodular functions, carry over to generalized permutohedra, or equivalently real-valued submodular functions. 
These characterizations are captured by the interplay between generalized permutohedra and generalized polymatroids, analogously to the interplay between \M-convex and $\Mnat$-convex sets.
For these, there are even further characterizations for which we refer to \cite{Frank:2011,FrankKiralyPapPritchard:2014,Kiraly:2018}. 

We briefly collect the necessary concepts to state \Cref{thm:real+quotients}, a characterization of quotients of generalized permutohedra. 
The proofs of the equivalences are very similar to the discrete setting in most cases and are omitted. 

For a generalized polymatroid $Q \subseteq \RR^E$, the \emph{top layer} $Q^\uparrow$ and the \emph{bottom layer} $Q^{\downarrow}$ are the faces with $x(E)$ maximal and minimal respectively.
These are both generalized permutohedra.
All of the operations defined in \Cref{sec:operations} and \Cref{sec:minors} carry over to generalized permutohedra and generalized polymatroids.

Let $\Gamma \subseteq \R^V \times \R^U$ be a generalized permutohedron and $P \subseteq \R^U$ a generalized permutohedron (resp. generalized polymatroid). 
The \emph{induction of $P$ through $\Gamma$} is the generalized permutohedron (resp. generalized polymatroid)
\begin{align*}
\ind{P}{\Gamma} = \left.\left(\Gamma + (\0_V \times P)\right)\right|_V \subseteq \RR^V \enspace .
\end{align*}
In the following, we also allow a relaxation to possibly unbounded permutohedra, in which facets can be translated to infinity.
Given a finite set $V$, we define the \emph{monoid of linking permutohedra} $(\cP_V, *, J_V)$ on $V$ to be the set of (possibly unbounded) generalized permutohedra with associative operation $*$ and identity $J_V$:
\begin{align*}
\cP_V &= \SetOf{\Gamma \subseteq \RR^V \times \RR^V}{\Gamma \text{ generalized permutohedron }} \, , \\
\Gamma * \Delta &= \SetOf{(x,-z) \in \RR^V \times \RR^V}{\exists y \in \RR^V \text{ such that } (x,-y) \in \Gamma \, , \, (y, -z) \in \Delta} \, , \\
J_V &= \SetOf{(x,-x) \in \RR^V \times \RR^V}{x \in \RR^V} \, .
\end{align*}
Green's right pre-order $\preceq_\cR$ on $\cP_V$ is defined as
\[
\Gamma \preceq_{\cR} \Delta \quad \Leftrightarrow \quad \exists X \in \cP_V \text{ such that } \Gamma*X = \Delta \, .
\]

A crucial difference to the discrete setting arises in the formulation of the analog of \eqref{exchange-property}.
For this, we recall the `real' of `continuous' analog or the exchange property from \cite[\S 8]{Murota:2003}.    
A polyhedron $P \in \R^E$ is the base polytope of a submodular function with values in $\R \cup \{+\infty\}$ if and only if for all $x,y \in P$ and $i \in \supp^+(x-y)$, there exists a $j \in \supp^-(x-y)$ and a positive number $\rho_0$ such that $x - \rho(e_i - e_j) \in P$ and $y + \rho(e_i - e_j) \in P$ for all $\rho \in [0,\rho_0]$.

The final unfamiliar equivalence is \eqref{eq:edge+directions}.
This follows from $\left(B(p) + \R_{\leq 0}^E\right) \cap \left(B(q) + \R_{\geq 0}^E\right)$ being the generalized polymatroid $G(p,q^\#)$, and generalized polymatroids having an alternative characterization as polytopes with only edge directions $e_i - e_j$ and $e_i$; see~\cite[Theorem 17.1]{Fujishige:2005}.

  \begin{theorem} \label{thm:real+quotients}
	Let $P, Q \in \R^E$ be generalized permutohedra, and $p,q \colon 2^E \to \R$ be the corresponding submodular set functions. 
	Then the following are equivalent. 
	\begin{enumerate}
	  \item (compliant functions) For all $X \subseteq Y \subseteq E$, the inequality $q(Y) - q(X) \leq p(Y) - p(X)$ holds. 
	  \item (containment of bases) For every $\sigma \in \Sym(E)$, the vertices $x_\sigma \in P$ and $y_\sigma \in Q$ satisfy $x_\sigma \geq y_\sigma$. 	  
		\item (submodular polyhedron containment) For all $X \subseteq E$, the containment $S(q/X) \subseteq S(p/X)$ holds, i.e., the submodular polyhedra are contained for all contractions.
		\item (top and bottom) The polyhedron $R = G(p,q^\#)$ is a generalized polymatroid with $P = R^{\uparrow}$ and $Q = R^{\downarrow}$. 
		\item (deletion-contraction) There exists a generalized permutohedron $R \subseteq \RR^{\tE}$ with $\tE = E \sqcup X$ such that $P = R \setminus X$ is the deletion of $R$ and $Q = R / X$ is the contraction of $R$.
		Moreover, $X$ can be picked to be a singleton.
		\item (exchange property) For all $x \in Q, y \in P$ and $i \in \supp^+(x-y)$, there exists a $j \in \supp^-(x-y)$ and a positive number $\rho_0$ such that $x - \rho(e_i - e_j) \in Q$ and $y + \rho(e_i - e_j) \in P$ for all $\rho \in [0,\rho_0]$.
    \item (edge directions) The polyhedron $\left(B(p) + \R_{\leq 0}^E\right) \cap \left(B(q) + \R_{\geq 0}^E\right)$ has only edge directions $e_i - e_j$ and $e_i$. \label{eq:edge+directions}
		\item (induction)
		There exists a generalized permutohedron $\Gamma \subseteq \R^E \times \R^\tE$ and $W \in \R^\tE$ such that $P$ is the top layer of the left set of $\Gamma$, and $Q$ is the induction of $W$ through $\Gamma$, i.e.
		\[
		P = \pi_E(\Gamma)^\uparrow \quad , \quad Q = \ind{W}{\Gamma} \, .
		\]
		\item ($\cR$-ordered linking permutohedra) 
    There exist linking permutohedra $\Gamma, \Delta \subseteq \RR^E \times \RR^E$ such that $P$ and $Q$ are top layers of the left sets of $\Gamma$ and $\Delta$ respectively, and $\Gamma \preceq_\cR \Delta$ where $\preceq_\cR$ is Green's right preorder on the monoid of linking permutohedra.
	\end{enumerate}
\end{theorem}

\begin{remark}
Recall that the set of all submodular functions forms a polyhedral cone in $\R^{2^E}$ cut out by the submodular inequalities from \Cref{def:submodularity}.
Equivalently, it is the facet deformation cone of the regular permutohedron \cite[Proposition 3.2]{PostnikovReinerWilliams:2008}.
As each generalized polymatroid is exactly the orthogonal projection (forgetting one coordinate) of a generalized permutohedron, the space of quotients of generalized permutohedra is exactly the deformation cone of an orthogonal projection of the regular permutohedron.
This also contains the set of all \M-convex set quotients as the integer points, and the set of all matroid quotients as $0/1$-lattice points. 
\end{remark}

\printbibliography

\vspace*{\fill}

\subsection*{Affiliations} 
$ $
\vspace*{0.2cm}

\noindent \textsc{Marie-Charlotte Brandenburg} \\
\textsc{Ruhr-Universit\"at Bochum } \\
%	Lindstedtsvägen 25, 114 28 Stockholm, Sweden} \\
\url{marie-charlotte.brandenburg@rub.de} \\

\noindent \textsc{Georg Loho} \\
\textsc{University of Twente \& Freie Universit\"at Berlin} \\
%	Address } \\
\url{georg.loho@math.fu-berlin.de} \\

\noindent \textsc{Ben Smith} \\
\textsc{Lancaster University} \\
%	Address} \\
\url{b.smith9@lancaster.ac.uk}

\end{document}